\documentclass[11pt]{amsart}

\usepackage{amsmath, amssymb, amsfonts}
\usepackage{graphicx} 
\usepackage{fullpage}
\usepackage{multicol}
\usepackage[usenames,dvipsnames,svgnames,table]{xcolor}
\usepackage{enumitem}
\usepackage{hyperref}
\usepackage{tikz}
\usepackage{tabularx}
\usepackage[margin=1in]{geometry}
\usepackage{overpic}
\usetikzlibrary{shapes.geometric,calc}
\definecolor{mygrey}{RGB}{180,180,180}
\definecolor{mygreen}{RGB}{0,128,0}
\theoremstyle{plain}
\newtheorem{thm}{Theorem}[section]
\newtheorem{lemma}[thm]{Lemma}
\newtheorem{prop}[thm]{Proposition}
\newtheorem{cor}[thm]{Corollary}

\newtheorem{conjecture}[thm]{Conjecture}

\numberwithin{equation}{section}

\theoremstyle{definition}
\newtheorem{definition}[thm]{Definition}
\newtheorem{example}[thm]{Example}
\newtheorem{remark}[thm]{Remark}

\newtheorem{questions}[thm]{Questions}

\newcommand{\G}{\Gamma}
\newcommand{\D}{\Delta}
\newcommand{\Z}{\mathbb{Z}}
\newcommand{\R}{\mathbb{R}}
\newcommand{\N}{\mathbb{N}}

\newcommand{\bT}{\mathbb{T}}
\newcommand{\Cayley}[1]{\mathcal {C}_#1}
\newcommand{\CAT}{\operatorname{CAT}}
\newcommand{\cC}{\mathcal{C}}
\newcommand{\cH}{\mathcal{H}}

\newcommand{\cL}{\mathcal{L}}
\newcommand{\cN}{\mathcal{N}}

\newcommand{\cP}{\mathcal{P}}

\newcommand{\betti}[1]{\mathrm{b_1}(#1)}

\renewcommand{\div}[1]{\mathrm{div}_{#1}}
\renewcommand{\implies}{\Longrightarrow}

\newcommand{\cT}{\mathcal T}
\newcommand{\cK}{\mathcal{K}}

\begin{document}

\title{Divergence, thickness and hypergraph index for \\ general Coxeter groups}

\author{Pallavi Dani}
\address{Pallavi Dani, Department of Mathematics, Louisiana State University, Baton Rouge, LA 70803--4918, USA}
\email{pdani@math.lsu.edu}

\author{Yusra Naqvi}
\address{Yusra Naqvi, Department of Mathematics, University College London, 25 Gordon Street, London WC1H 0AY, United Kingdom}
\email{y.naqvi@ucl.ac.uk}

\author{Ignat Soroko}
%\address{Ignat Soroko, Department of Mathematics, 
%University of North Texas,
%1155 Union Circle \#311430,
%Denton, TX 76203--5017,
%USA}
%\email{ignat.soroko@unt.edu, ignat.soroko@gmail.com}
\address{Ignat Soroko, Department of Mathematics, Southern Methodist University, Dallas, TX 75205, USA}
\email{isoroko@smu.edu}

\author{Anne Thomas}
\address{Anne Thomas, School of Mathematics \& Statistics, Carslaw Building F07,  University of Sydney NSW 2006, Australia}
\email{anne.thomas@sydney.edu.au}

\date{ \today }

\begin{abstract}  We study divergence and thickness for general Coxeter groups $W$.  We first characterise linear divergence, and show that if $W$ has superlinear divergence then its divergence is at least quadratic.  We then formulate a computable combinatorial invariant, 
hypergraph index, for arbitrary Coxeter systems $(W,S)$.  This generalises Levcovitz's definition
for the right-angled case.  We prove that if $(W,S)$ has finite hypergraph index $h$,  then $W$ is (strongly algebraically) thick of order at most~$h$, hence has divergence bounded above by a polynomial of degree $h+1$.  We conjecture that these upper bounds on the order of thickness and divergence are in fact equalities, and we prove our conjecture for certain families of Coxeter groups.  These families are obtained by a new construction which, given any right-angled Coxeter group, produces infinitely many examples of non-right-angled Coxeter systems with the same hypergraph index.  Finally, we give an upper bound on the hypergraph index of any Coxeter system $(W,S)$, and hence on the divergence of $W$, in terms of, unexpectedly, the topology of its associated Dynkin diagram.  
\end{abstract}

\maketitle

\section{Introduction}

We study the quasi-isometry invariants divergence and thickness for general Coxeter groups.  We define a combinatorial invariant of any Coxeter system $(W,S)$, its \emph{hypergraph index} $h(W,S) \in \{0,1,2,\dots \} \cup \{ \infty \}$, which is computed from its Coxeter presentation $W = \langle S \mid (st)^{m_{st}} \rangle$,  and which conjecturally determines both the divergence and the order of thickness of $W$.   This generalises the hypergraph index introduced by Levcovitz in the right-angled case~\cite{levcovitz-RACG}.

Roughly speaking, the \emph{divergence} of a pair of geodesic rays with common basepoint measures, as a function of~$r$, the length of a shortest path connecting their time-$r$ points which stays at least distance $r$ from the basepoint.  
Gersten~\cite{gersten-quadratic} used this idea to define a quasi-isometry invariant, also called divergence.  The divergence of a finitely generated group is then the divergence of one of its Cayley graphs.  

In symmetric spaces of noncompact type, divergence is either linear  or exponential, and Gromov suggested in~\cite[{\protect{$6.\mathrm B_2$(h), p.\,133}}]{gromov}
that the same dichotomy should hold for more general nonpositively curved spaces.  However for many important families, it is now known that divergence can be linear, quadratic or exponential; these families include $3$-manifold groups~\cite{gersten-3mfld}, mapping class groups and Teichm\"uller space~\cite{behrstock, duchin-rafi}, and right-angled Artin groups~\cite{abddy,behrstock-charney}. 
The first constructions of families of $\CAT(0)$ groups with divergence polynomial of any integer degree are due to Macura~\cite{macura} and independently  Behrstock--Dru\c{t}u~\cite{behrstock-drutu}.  Finitely generated groups with divergence neither a polynomial of integer degree nor an exponential are constructed in \cite{olshanskii-osin-sapir, gruber-sisto}.  Recently, Brady and Tran~\cite{brady-tran-div, brady-tran-divcat0} constructed finitely presented groups (including subgroups of $\CAT(0)$ groups) with divergence $r^\alpha$ for $\alpha$ running through an irrational dense subset of~$[2,\infty)$, and with divergence $r^n\log(r)$ for all integers $n\ge2$. 

The quasi-isometry invariant \emph{thickness} was introduced by Behrstock, Dru\c{t}u and Mosher \cite{behrstock-drutu-mosher} as a means to understand the geometry of non-relatively hyperbolic groups.  A metric space $X$ is (strongly) thick of order $0$ if none of its asymptotic cones has a cut-point, and for integers $n \geq 1$, the space $X$ is (strongly) thick of order $n$ if, roughly speaking, any two points in $X$ can be connected by a chain of subsets, each of which is thick of order $n-1$.  Behrstock and Dru\c{t}u proved in~\cite[Corollary~4.17]{behrstock-drutu} that a group which is strongly thick of order at most $n$ has divergence at most $r^{n+1}$.   

 In the setting of right-angled Coxeter groups, the study of divergence was begun by Dani and Thomas~\cite{dani-thomas-div}, and continued by Behrstock, Falgas-Ravry, Hagen and Susse~\cite{behrstock-falgas-ravry-hagen-susse},  Behrstock, Hagen and Sisto~\cite{behrstock-hagen-sisto-caprace}, and Levcovitz~\cite{levcovitz-div,levcovitz-RACG,levcovitz-RACG-2020}; Levcovitz~\cite{levcovitz-div} also considered divergence for certain non-right-angled Coxeter groups.  We recall many results from these works later in this introduction.  For arbitrary Coxeter systems~$(W,S)$, Caprace~\cite{caprace, caprace-erratum} characterised the collections of special subgroups of~$W$ such that $W$ is hyperbolic relative to such a collection, while  Behrstock, Hagen, Sisto and Caprace proved in the Appendix to~\cite{behrstock-hagen-sisto-caprace} that~$W$ admits a canonical minimal relatively hyperbolic structure, whose peripheral subgroups are in fact special subgroups.  In the same Appendix, these four authors also proved that any Coxeter group is either strongly algebraically thick, or is relatively hyperbolic, and provided an inductive construction of all Coxeter systems $(W,S)$ such that $W$ is strongly algebraically thick.

We are now ready to describe our contributions.  Throughout, we assume $(W,S)$ is a Coxeter system with~$S$ finite.  We recall  that right-angled Coxeter groups are \emph{rigid}, meaning that up to isomorphism they have unique defining graphs~\cite{radcliffe}, but in general Coxeter groups are not rigid (see, for instance,~\cite{brady-mccammond-muhlherr-neumann}).  Thus we often refer to Coxeter systems rather than Coxeter groups.  

In all our results on divergence, we assume that the Coxeter group $W$ is $1$-ended.  This assumption ensures the equivalence of Gersten's definition of divergence and others in the literature (see the discussion in~\cite[Section~3]{drutu-mozes-sapir}).  Indeed, the definition of divergence for finitely presented  groups with more than one end is not consistent in the literature. The usual definition implies that divergence takes infinite values for such groups, but there exist versions where their divergence is a real-valued function (see, for example,~\cite[Remark 3.9]{drutu-mozes-sapir}).  To avoid having to consider these varying definitions of divergence throughout, we have opted instead to restrict to the $1$-ended case in this paper. Moreover, the quasi-isometry classification of Coxeter groups reduces to that of $1$-ended Coxeter groups. (For right-angled Coxeter groups this was explained in the introduction to~\cite{dani-thomas-qi-hyp}, and the general case follows from the results of Davis~\cite[Proposition 8.8.2]{davis-book} and Papasoglu--Whyte~\cite[Theorems 0.2 and 0.3]{papasoglu-whyte} by similar reasoning.)
Given a Coxeter system $(W,S)$, there is a concrete condition for determining whether $W$ is $1$-ended~\cite[Theorem~8.7.2]{davis-book}, and we provide a corresponding function to check $1$-endedness in our accompanying code~\cite{hindex_gap} written for the computer algebra system \textsf{GAP}~\cite{GAP4}. 
We refer to Section~\ref{sec:CoxeterDavis} for additional background on Coxeter groups. 

Our first results for general Coxeter groups concern linear and quadratic divergence.  We prove:

\begin{thm}\label{thm:linearQuadraticIntro} Let $(W,S)$ be a Coxeter system such that $W$ is $1$-ended and $(W,S)$ is irreducible and nonaffine.  Then the divergence of $W$ is at least quadratic.
\end{thm}

\noindent The proof of Theorem~\ref{thm:linearQuadraticIntro} is brief and combines results from the literature (see Section~~\ref{sec:linearQuadratic}).  We observe that $W$ has a rank one element, by work of Caprace and Fujiwara~\cite{caprace-fujiwara}, and then apply a general result of Kapovich and Leeb~\cite[Proposition~3.3]{kapovich-leeb} to conclude that the divergence of $W$ is at least $r^2$.  A similar strategy is used in, for instance, the proofs of~\cite[Theorem~6.10]{hagen} and~\cite[Theorem~3.3]{behrstock-hagen-sisto-caprace}.

We next characterise linear divergence for arbitrary Coxeter groups.  

\begin{cor}\label{cor:linearIntro}  Let $(W,S)$ be a Coxeter system such that $W$ is $1$-ended.  Then $W$ has linear divergence if and only if $(W,S)=(W_1,S_1) \times (W_2,S_2)$ where either:
\begin{enumerate}
\item both $W_1$ and $W_2$ are infinite; or
\item $W_1$ is finite (possibly trivial) and $(W_2,S_2)$ is irreducible affine of rank $\ge 3$.
\end{enumerate}
\end{cor}

\noindent This generalises the characterisation of linear divergence for right-angled Coxeter groups (see~\cite[Theorem~4.1]{dani-thomas-div} and~\cite[Proposition~2.11]{behrstock-hagen-sisto-caprace}).  Note that possibility (2) does not occur in the right-angled case.  We prove Corollary~\ref{cor:linearIntro} using Theorem~\ref{thm:linearQuadraticIntro} and standard facts about Coxeter systems, together with an argument from~\cite{abddy}.

We record the following useful consequence of Corollary~\ref{cor:linearIntro} and Theorem~\ref{thm:linearQuadraticIntro} (which relies essentially on Proposition~3.3 of~\cite{kapovich-leeb}, as we mentioned above).

\begin{cor}\label{cor:superlinear}
If a $1$-ended Coxeter group has superlinear divergence, then its divergence is at least quadratic.\qed
\end{cor}

\noindent In the right-angled case this result seems to have been known to the authors of~\cite{behrstock-hagen-sisto-caprace}, as indicated by the proof of~\cite[Theorem~3.3]{behrstock-hagen-sisto-caprace}, but it is not stated explicitly there.

The main theoretical contribution of this paper is the definition of \emph{hypergraph index} for general Coxeter systems $(W,S)$, which generalises the definition introduced by Levcovitz in the right-angled setting~\cite{levcovitz-RACG}. Given any $(W,S)$, we construct a  
sequence $\Lambda_0, \Lambda_1, \Lambda_2 \dots$ of sets of subsets of~$S$, where $\Lambda_0$ contains all $T \subset S$ such that the corresponding special subgroup $W_T$ is (strongly) thick of order $0$, and each $\Lambda_i$ is the union of certain elements of $\Lambda_{i-1}$.
By definition, the hypergraph index $h(W,S)$ is then the smallest $h$ for which $S \in \Lambda_h$ (if it exists), and $\infty$ otherwise.  It is evident from the definition that $h(W,S)$ is either a non-negative integer or equals $\infty$. We provide detailed motivation for our formulation of hypergraph index in Section~\ref{sec:motiv}.

A key feature of the construction is that $h(W,S)$ can be explicitly computed from the Coxeter presentation $W = \langle S \mid (st)^{m_{st}}\rangle$.
In fact, in our general setting we show that  for finite $m_{st} \geq 7$, replacing all relations $(st)^{m_{st}}$ by~$(st)^7$ does not change the hypergraph index (see Proposition~\ref{prop:labels7up}).  Thus the determination of possible hypergraph indexes for Coxeter systems with any fixed number of generators is a finite computational problem. 

Our results relating hypergraph index, thickness and divergence are as follows.

\begin{thm}\label{thm:hyp_thick_div_Intro}  Let $(W,S)$ be a Coxeter system such that~$W$ is $1$-ended. Let $h = h(W,S)$. Then:
\begin{enumerate}
    \item $h = 0$ if and only if $W$ has linear divergence.
    \item If $h = 1$, then $W$ has quadratic divergence.
    \item If $h$ is finite, then $W$ is strongly thick of order at most $h$, hence the divergence of $W$ is bounded above by a polynomial of degree $h + 1$.
    \item $h = \infty$ if and only if $W$ is relatively hyperbolic.
\end{enumerate}
\end{thm}

\noindent These statements generalise results of Levcovitz~\cite{levcovitz-RACG} in the right-angled setting, as follows. Theorem~\ref{thm:hyp_thick_div_Intro}(1) generalises the equivalence of (4) and (5) in~\cite[Theorem 6.1]{levcovitz-RACG}, Theorem~\ref{thm:hyp_thick_div_Intro}(2) generalises the implication (5) $\implies$ (4) in~\cite[Theorem 6.2]{levcovitz-RACG}, Theorem~\ref{thm:hyp_thick_div_Intro}(3) generalises~\cite[Theorem~6.3 and Corollary~6.4]{levcovitz-RACG}, and Theorem~\ref{thm:hyp_thick_div_Intro}(4)  generalises~\cite[Remark 3.10]{levcovitz-RACG}.  Our proof of the first claim in Part (3) above is considerably shorter and more direct than that of~\cite[Theorem~6.3]{levcovitz-RACG}, and makes natural use of strong algebraic thickness.

We expect that the converse to Theorem~\ref{thm:hyp_thick_div_Intro}(2) holds.  (For right-angled Coxeter groups, this converse is the implication (4) $\implies$ (5) of \cite[Theorem~6.2]{levcovitz-RACG}, and relies upon the characterisation of quadratic divergence for right-angled Coxeter groups from~\cite{dani-thomas-div, levcovitz-div}.)  More generally, we make the following conjecture.

\begin{conjecture}\label{conj:lower}  Let $(W,S)$ be a Coxeter system such that $W$ is $1$-ended.  Let $h = h(W,S)$.  Then the following are equivalent:
\begin{enumerate}
    \item $(W,S)$ has finite hypergraph index $h$;
    \item $W$ is strongly thick of order $h$; and
    \item the divergence of $W$ is polynomial of degree $h + 1$.
\end{enumerate}
\end{conjecture}

In the right-angled case, this is Conjecture~1.1 of~\cite{levcovitz-RACG}.  Levcovitz recently proved this conjecture as Theorem~A of~\cite{levcovitz-RACG-2020}, by showing  that if $h$ is finite, the Cayley graph of $W$ contains a periodic geodesic with divergence polynomial of degree $h+1$.  
If Conjecture~\ref{conj:lower} holds, there are many interesting corollaries, as in~\cite{levcovitz-RACG-2020}.  
For instance, obtaining exact bounds on divergence and thickness is notoriously hard, but the equivalences in Conjecture~\ref{conj:lower}   provide an algorithmic way to compute these. Furthermore, it follows that the divergence of any $1$-ended Coxeter group is either polynomial of integer degree or exponential. In addition, Conjecture~\ref{conj:lower} implies that hypergraph index is a (computable) quasi-isometry invariant of $1$-ended Coxeter groups, which is not obvious from the definition.

As evidence for Conjecture~\ref{conj:lower}, we prove that for all finite $h$, there are ``many" examples of non-right-angled Coxeter systems for which this conjecture holds.  For this, we introduce a procedure, called the \emph{duplex construction}, which takes as input a right-angled Coxeter system $(W,S)$, and produces a Coxeter system on twice as many generators whose defining graph has edge labels~$2$, $m \in \{2,3,\dots\}$ and $n \in \{5,6,\dots\} \cup \{\infty\}$, and which has the same hypergraph index as $(W,S)$.  Then by applying the duplex construction to a family of examples from~\cite{dani-thomas-div}, 
we obtain the following statement.

\begin{cor}\label{cor:duplexIntro}  For every integer $h \geq 0$, there are infinitely many pairwise nonisomorphic, non-right-angled Coxeter groups, obtained via the duplex construction, whose associated Coxeter systems have hypergraph index $h$.   Moreover, if $m$ is even and $n = \infty$ in the duplex construction, these Coxeter groups are strongly thick of order $h$ and have divergence polynomial of degree $h+1$.
\end{cor}

\noindent In particular, Corollary~\ref{cor:duplexIntro} implies that there are families of non-right-angled Coxeter groups with divergence polynomial of any integer degree. An explicit construction of such groups is given in Proposition~\ref{prop:divd}.

The divergence statement in Corollary~\ref{cor:duplexIntro} is obtained by combining the upper bound provided by Theorem~\ref{thm:hyp_thick_div_Intro}(3) with a lower bound which we establish for these specific examples.  Our proof of the lower bound is similar in outline to that of the analogous result in~\cite{dani-thomas-div}, but the details are considerably more technical.  We restrict to $m$ even to ensure that each wall in the associated Davis complex has a well-defined type in $S$; the types of walls were used extensively in~\cite{dani-thomas-div,levcovitz-div,levcovitz-RACG,levcovitz-RACG-2020}.  In the right-angled case, two walls of the same type never intersect, but if $m \ge 4$ is even then distinct walls of the same type can intersect, which introduces greater complexity to the arguments.  We note that Levcovitz in~\cite{levcovitz-div} obtained lower bounds on divergence for not necessarily right-angled Coxeter groups such that, among other conditions, some vertex in the defining graph has a triangle-free neighbourhood. However in the defining graphs of duplex groups, every vertex is contained in a triangle, so the lower bounds from~\cite{levcovitz-div} cannot be applied.

We conclude by establishing a surprising upper bound on hypergraph index and hence divergence.   Recall that if a (topological) graph $\Delta$ has $v$ vertices, $e$ edges and~$k$ components, then its \emph{first Betti number} is given by $\betti\D = e-v+k$, and this number equals the rank of its first homology group.

\begin{thm}\label{thm:bettiIntro}
Let $(W,S)$ be a Coxeter system with Dynkin diagram $\Delta = \Delta(W,S)$ and hypergraph index~$h = h(W,S)$. If $h$ is finite then $h \leq \betti\D + 1$. 
\end{thm}

\noindent Combining this with  Theorem~\ref{thm:hyp_thick_div_Intro}(3) and (4), we obtain:
\begin{cor}\label{cor:bettiIntro}
Let $(W,S)$ be a Coxeter system with Dynkin diagram $\Delta$.  If $W$ is not relatively hyperbolic, then the divergence of $W$ is bounded above by a polynomial of degree $\betti{\Delta}+2$.  
\end{cor}

\noindent We do not know of any previous connection between large-scale geometry and the topology of Dynkin diagrams.  
 It follows from Corollary~\ref{cor:bettiIntro}
that for all integers $n \geq 2$, if $G$ is any non-relatively hyperbolic group with divergence
strictly greater than~$r^n$, then~$G$ is not quasi-isometric to any Coxeter group whose Dynkin diagram satisfies $\betti\D \leq n - 2$.

We prove Theorem~\ref{thm:bettiIntro} by induction on~$\betti\D$, and en route obtain the following nesting result:

\begin{cor}\label{cor:spectrumIntro}  Suppose $(W, S)$ has finite hypergraph index $h \geq 1$. Then there exist subsets $S \supset T_1 \supset T_2 \supset \dots \supset T_h$ such that the subsystem $(W_{T_i}, T_i)$ has hypergraph index $h-i$.\qed 
\end{cor}

\noindent If Conjecture~\ref{conj:lower} holds, this implies that any Coxeter group $W$ with divergence polynomial of integer degree $n \geq 2$ has nested special subgroups with  divergence $r^{k}$, for all integers $1 \leq k \leq n$.

We do not know whether the upper bound on hypergraph index provided by Theorem~\ref{thm:bettiIntro} is sharp, although we provide examples with $\betti\D  = 0,1,2,3,4,5$ and hypergraph index equal to $\betti\D + 1$ (see Figure~\ref{fig:GenThetaGraphs}). We found our examples for $\betti\D = 2,3,4,5$ using our \textsf{GAP} code to compute hypergraph index~
\cite{hindex_gap}.  
For $\D$ a tree (respectively, a cycle) we also give examples which realise all hypergraph indexes in the spectrum $\{0,1,\infty\}$ (respectively, $\{0,1,2,\infty\}$).  Hence we obtain the interesting fact that Coxeter systems whose Dynkin diagram is a tree form a family whose divergence takes on the same limited values as in the families of groups for which divergence was first studied:

\begin{cor}\label{cor:DynkinTreeDivIntro}  If the Dynkin diagram of $(W,S)$ is a tree and $W$ is $1$-ended, then $W$ has divergence linear, quadratic or exponential only.  Moreover, each of these possibilities is realised.
\end{cor}
 
\subsection*{Organisation of the paper}  We give background on divergence, thickness and Coxeter systems in Section~\ref{sec:background}.  In Section~\ref{sec:linearQuadratic} we prove Theorem~\ref{thm:linearQuadraticIntro} and Corollaries~\ref{cor:linearIntro} and~\ref{cor:superlinear}, which concern linear and quadratic divergence.  Section~\ref{sec:hypergraph} gives our definition of hypergraph index and establishes Theorem~\ref{thm:hyp_thick_div_Intro}.  We introduce the duplex construction and prove Corollary~\ref{cor:duplexIntro} in Section~\ref{sec:duplex}.  Finally, we prove Theorem~\ref{thm:bettiIntro} along with Corollaries~\ref{cor:spectrumIntro} and~\ref{cor:DynkinTreeDivIntro} in Section~\ref{sec:Betti}.

\subsection*{Acknowledgements} We thank Mark Hagen and Ivan Levcovitz for helpful discussions, and an anonymous referee for extremely careful reading and thoughtful suggestions.  We are grateful to the University of Sydney, the School of Mathematics at the University of Sydney and the University of Sydney Mathematical Research Institute for supporting a visit by Ignat Soroko, who also acknowledges support from the AMS--Simons travel grant. Yusra Naqvi and this research were supported in part by ARC Grant DP180102437. Pallavi Dani was supported by NSF Grant No.~DMS--1812061.

\section{Background}\label{sec:background}

This section gives brief background.  We discuss divergence in Section~\ref{sec:div}, then in Section~\ref{sec:tightThick} recall relevant material concerning thickness and divergence from~\cite{behrstock-drutu}.   Section~\ref{sec:CoxeterDavis} gives some background on Coxeter groups, and Section~\ref{sec:hypCoxeter} recalls characterisations of (relative) hyperbolicity and thickness for Coxeter groups.

\subsection{Divergence}\label{sec:div}

We now recall Gersten's definition of divergence from~\cite{gersten-quadratic}. 

Let $(X,d)$ be a $1$-ended geodesic metric space.   For $p \in X$, write $S(p,r)$ and $B(p,r)$ for the sphere and open ball of radius $r$ about $p$, respectively.  A path in $X$ is  \emph{$(p,r)$-avoidant} if it lies in 
$X \setminus B(p,r)$.   Given  $x,y \in X \setminus B(p,r)$, the \emph{$(p,r)$-avoidant distance} $d^{\mathrm{av}}_{p,r}(x,y)$
between them is the infimum of the lengths of all $(p,r)$-avoidant paths from $x$ to $y$.  

Now fix a basepoint $e\in X$.  For each $0 <\rho \le 1$, let 
\[
\delta_{\rho}(r) = \sup_{x, y \in S(e,r)} d^{\mathrm{av}}_{e,\rho r}(x,y).
\]
Then the \emph{divergence} of $X$ is defined to be the resulting collection of functions 
\[
\div{X} = \{\delta_{\rho}
\mid 0 < \rho
\le 1 \}.
\]  

Let $(W,S)$ be a Coxeter system (with $S$ finite) such that $W$ is $1$-ended.  Then the Cayley graph of $W$ with respect to $S$, denoted $\cC(W,S)$, has the geodesic extension property, meaning that any finite geodesic segment can be extended to an infinite geodesic ray (see, for example, Lemmas 4.7.2 and 4.7.3 of~\cite{davis-book}).  It is not hard to show that if the metric space $X$ has the geodesic extension property, then 
$\delta_{\rho} \simeq \delta_{1}$ for all $0 < \rho \le 1$, 
 where $\simeq$ is the equivalence relation on functions which is defined as follows. For arbitrary functions $f,g\colon[0,\infty)\to \R$ we have $f\simeq g$ if and only if $f\preceq g$ and $g\preceq f$, where the relation $\preceq$ is defined as:
 \[
 f \preceq g \iff \; \exists \; A,B>0, C,D,E\ge0 \text{ such that } f(r) \le Ag(Br+C) + Dr+E \text{ for all } r\ge0.  
 \]
Thus we will think of $\div{X}$ as a function of $r$, defining it to be equal to $\delta_1(r)$. Up to the equivalence relation $\simeq$, the divergence $\div{X}$ is a quasi-isometry invariant of the space $X$, which is independent of the chosen basepoint (see~\cite[Prop.\,2.1]{gersten-quadratic}).  Thus for a Coxeter system $(W,S)$ as above, we define the divergence of the group $W$ to be the divergence of the Cayley graph $\cC(W,S)$. 

We will also need the notion of the divergence of a geodesic. Let $\gamma\colon\R\to X$ be a complete geodesic in a $1$-ended geodesic metric space $X$. We define $\div{\gamma}(r)$ to be the distance between points $\gamma(r)$ and $\gamma(-r)$ in $X\setminus B(\gamma(0),r)$ with respect to the path metric, and we will say that a path between $\gamma(r)$ and $\gamma(-r)$ is  \emph{$r$-avoidant} if it lies in $X \setminus B(\gamma(0),r)$.  
The $\simeq$ equivalence class of $\div{\gamma}(r)$ is called the \textit{divergence of the geodesic $\gamma$}, and it serves as a lower bound for the divergence $\div{X}(r)$ of the space $X$.

We will say that the divergence $\div{X}(r)$ or $\div{\gamma}(r)$ is \emph{linear} if it is $\simeq$ equivalent to $r$,  \emph{quadratic} if it is $\simeq$  equivalent to $r^2$, and so on. We will also say that the divergence is \emph{superlinear} if it is $\succeq r$ and $\not\simeq r$, and that it is \emph{subquadratic} if it is $\preceq r^2$ and $\not\simeq r^2$.  Finally, we say that the divergence is \emph{exponential} if it is $\simeq$ equivalent to $e^r$.

\subsection{Tight networks, thickness and divergence}\label{sec:tightThick}

In our proofs we will make careful use of several definitions and results from the work of Behrstock and Dru\c{t}u~\cite{behrstock-drutu}.  We recall these ingredients below.  Note that, unlike~\cite{levcovitz-RACG, levcovitz-RACG-2020}, we have chosen to retain the adjective ``strong" when referring to (algebraic) thickness.

We first recall some definitions related to networks of subspaces and of subgroups.  Given $C,D \geq 0$ and a subset $L$ of a metric space $X$, write $\cN_C(L)$ for the open  $C$-neighbourhood of $L$.  We say $L$ is \emph{$C$-path connected} if any two points in $L$ can be connected by a path in $\cN_C(L)$, and \emph{$(C,D)$-quasiconvex} if any two points in $L$ can be connected in $\cN_C(L)$ by a $(D,D)$-quasigeodesic; if $C = D$ then $L$ is \emph{$C$-quasiconvex}.  

\begin{definition}[tight network of subspaces, Definition~4.1 of~\cite{behrstock-drutu}]\label{defn:tight}
Let $X$ be a metric space, $C, D \geq 0$, and $\cL$ a collection of subsets of $X$.  We say that $X$ is a \emph{$(C,D)$-tight network with respect to $\cL$} if every $L \in \cL$, with the induced metric, is $(C,D)$-quasiconvex, $X = \cup_{L \in \cL} N_C(L)$, and  for any two $L, L' \in \cL$ and any $x \in X$ such that $B(x,3C)$ intersects both $L$ and $L'$, there is a finite sequence
$L = L_1,L_2,\dots,L_{n-1},L_n = L'$
of elements of $\cL$ such that $n \leq D$ and for all $1 \leq i < n$, the intersection $\cN_C(L_i) \cap \cN_C(L_{i+1})$ has infinite diameter, is $D$-path connected, and intersects $B(x,D)$.  We say that $X$ is a \emph{tight network with respect to $\cL$} if it is a $(C,D)$-tight network with respect to $\cL$ for some $C, D \geq 0$.
\end{definition}

\begin{definition}[tight algebraic network of subgroups, Definition~4.2 of~\cite{behrstock-drutu}]\label{defn:tightAlgebraic}
Let $G$ be a finitely generated group, $M \geq 0$ and $\cH$ a collection of subgroups of $G$.  We say that $G$ is an \emph{$M$-tight algebraic network with respect to $\cH$} if the elements of $\cH$ are $M$-quasiconvex subgroups whose union generates a finite-index subgroup of $G$, and for any two $H, H' \in \cH$ there exists a finite sequence $H = H_1,\dots,H_n = H'$ of subgroups in $\cH$ such that for all $1 \leq i < n$, the intersection $H_i \cap H_{i+1}$ is infinite and $M$-path connected. We say that $G$ is a \emph{tight algebraic network with respect to $\cH$} if it is an $M$-tight algebraic network with respect to $\cH$ for some $M \geq 0$.
\end{definition}

The following result connects the two previous definitions.

\begin{prop}[Proposition~4.3 of~\cite{behrstock-drutu}]\label{prop:tight}
Let $G$ be a finitely generated group with $\cH$ a collection of subgroups forming a tight algebraic network inside $G$, and $G_1$ the finite index subgroup of $G$ generated by the elements of $\cH$.  Then $G$ is a tight network with respect to the collection of left cosets 
\[
\pushQED{\qed}
\cL = \{ g H \mid g \in G_1, H \in \cH \}.\qedhere
\popQED
\]
\end{prop}

We now recall some definitions and results from~\cite{behrstock-drutu} relating to thickness. A metric space $X$ is said to be \emph{wide} if none of its asymptotic cones has a cut-point. 

\begin{definition}[strong thickness, Definition~4.13  of~\cite{behrstock-drutu}]\label{defn:strongThick}
Let $X$ be a metric space.  We say $X$ is \emph{strongly thick of order $0$} if it is wide.  Given $C,D \geq 0$ and $n \in \N$ we say $X$ is \emph{$(C,D)$-strongly thick of order at most $n$} if $X$ is a $(C,D)$-tight network with respect to a collection of subsets $\cL$, and the subsets in $\cL$, with the induced metric, are $(C,D)$-strongly thick of order at most $n-1$.  We say $X$ is \emph{$(C,D)$-strongly thick of order $n$} if it is strongly thick of order at most $n$, and is not strongly thick of order $n-1$.   We say $X$ is \emph{strongly thick of order $n$} if it is $(C,D)$-strongly thick of order $n$ for some $C, D \geq 0$.
\end{definition}

\begin{definition}[strong algebraic thickness, Definition~4.14 of~\cite{behrstock-drutu}]\label{defn:strongAlgThick}
Let $G$ be a finitely generated group.  We say $G$ is \emph{strongly algebraically thick of order $0$} if it is wide, and \emph{strongly algebraically thick of order at most $n$} if $G$ is a tight algebraic network with respect to a finite collection of subgroups $\cH$ such that all subgroups in $\cH$ are strongly algebraically thick of order at most $n-1$.  We say $G$ is \emph{strongly algebraically thick of order $n$} if it is strongly algebraically thick of order at most $n$, and is not strongly algebraically thick of order $n-1$.
\end{definition}

The implications we will use between these notions of thickness and divergence are given by the next two results. 

\begin{prop}\label{prop:wide}
Let $G$ be a finitely generated group.  The following are equivalent:
\begin{enumerate}
    \item $G$ is wide.
    \item $G$ has linear divergence.
    \item $G$ is strongly algebraically thick of order $0$.
    \item $G$ is strongly thick of order $0$.
\end{enumerate}
\end{prop}
\begin{proof}
The equivalence of (1), (3) and (4) is by definition. That (2) is equivalent to (1) is proven in~\cite[Lemma~3.17]{drutu-mozes-sapir}.
\end{proof}

\begin{thm}[Corollary~4.17 of~\cite{behrstock-drutu}]\label{cor:thickDiv}
If a finitely generated group $G$ is strongly thick of order at most $n$, then the divergence of $G$ is bounded above by a polynomial of degree~$n + 1$.\qed
\end{thm}

\subsection{Coxeter groups and the Davis complex}\label{sec:CoxeterDavis}

We now give some background on Coxeter groups, as well as recalling the construction of the Davis complex.  We mostly  follow~\cite{davis-book}.  

Recall that a \emph{Coxeter group} is a group $W$ with presentation of the form
\[ W = \langle S \mid (st)^{m_{st}} = 1 \mbox{ for all } s,t \in S \rangle \]
where $m_{ss} = 1$ for all $s \in S$, and $m_{st} = m_{ts} \in \{2,3,4,\dots\} \cup \{ \infty \}$ for all distinct $s, t \in S$.  Here $m_{st} = \infty$ means that the element $st$ has infinite order.  We write $C_2$ for the cyclic group of order~$2$, so that $\langle s \rangle \cong C_2$ for every $s \in S$, and $D_\infty$ for the infinite dihedral group, so that if $m_{st} = \infty$ then $\langle s, t \rangle \cong D_\infty$.  If $|S| = n$, we will sometimes let $S = \{ s_1,\dots,s_n \}$ and write $m_{ij}$ for $m_{s_i s_j}$.  

For $W$ and $S$ as in the previous paragraph, the pair $(W,S)$ is called a \emph{Coxeter system}, and the cardinality of the generating set $S$ is its \emph{rank}.  A Coxeter system is \emph{right-angled} if $m_{st} \in \{2,\infty\}$ for all distinct~$s,t \in S$.  Since right-angled Coxeter groups are rigid (as discussed in the introduction), we may then define a Coxeter group $W$ to be \emph{right-angled} if some (hence any) Coxeter system~$(W,S)$ is right-angled.

A Coxeter system $(W,S)$ can be encoded by two different edge-labelled simplicial graphs, the \emph{defining graph} $\G = \G(W,S)$, typically used in right-angled cases, and the \emph{Dynkin diagram} $\Delta = \Delta(W,S)$.  Both of these have vertex set $S$, and an edge between any $s, t \in S$ such that $3 \leq m_{st} < \infty$ labelled by the integer $m_{st}$.  The defining graph (respectively, Dynkin diagram) additionally has an edge labelled $2$ (respectively, $\infty$) between any $s,t \in S$ such that $m_{st} = 2$ (respectively, $m_{st} = \infty$). 
We will follow the convention that in a Dynkin diagram, edge labels~$3$ are omitted.   
We will also omit the edge label $2$ from the defining graph of a right-angled Coxeter system.  

A Coxeter system $(W,S)$ is \emph{irreducible} if there is no nontrivial partition of $S$ into two disjoint commuting subsets,  \emph{spherical} if $W$ is finite, and \emph{affine} if $S$ is the set of reflections in the faces of a compact polytope in Euclidean space.  The classifications of irreducible spherical and affine Coxeter systems are classical.  For the reader's convenience, we reproduce the Dynkin diagrams for the irreducible spherical and affine Coxeter systems in Figures~\ref{fig:spherical} and~\ref{fig:affine}, respectively. 

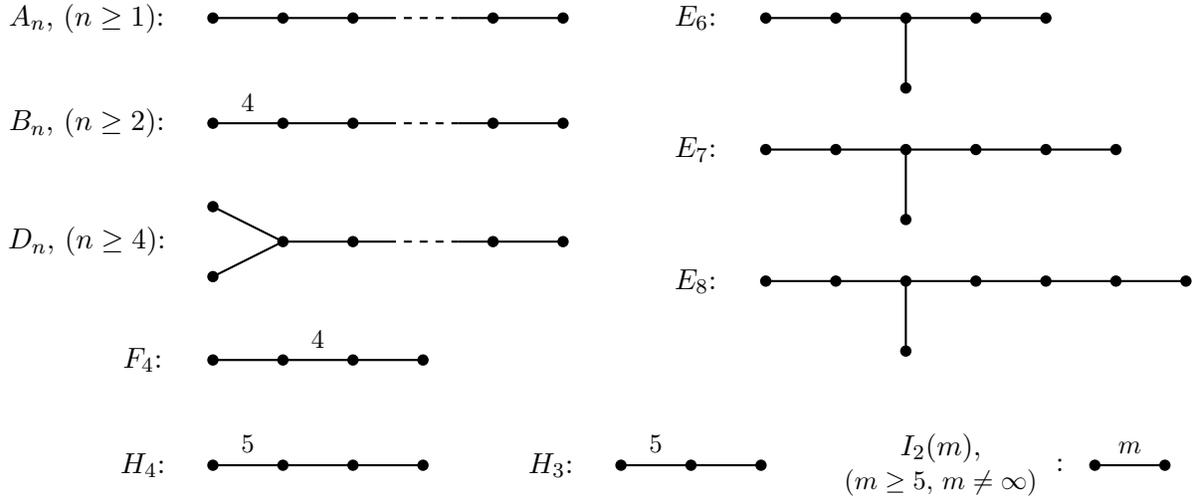
\begin{figure}
\begin{center}
\begin{tikzpicture}[scale=0.7,double distance=2.3pt,thick]
% A_n
\begin{scope}[xshift=-5cm,yshift=-3.5cm,scale=1.33] 
\fill (0,0) circle (2.3pt); 
\fill (1,0) circle (2.3pt); 
\fill (2,0) circle (2.3pt); 
\fill (4,0) circle (2.3pt); 
\fill (5,0) circle (2.3pt); 
\draw (0,0)--(1,0)--(2,0)--(2.5,0);
\draw [dashed] (2.5,0)--(3.5,0);
\draw (3.5,0)--(4,0)--(5,0);
\draw (-1.8,0) node {$A_n$, $(n\ge1)$:};	
\end{scope}

% B_n
\begin{scope}[xshift=-5cm,yshift=-5.5cm,scale=1.33] 
%\draw [double] (0,0)--(1,0);
\draw (0,0)--(1,0); \draw (0.5,0.3) node {\small$4$};
\fill (0,0) circle (2.3pt); 
\fill (1,0) circle (2.3pt); 
\fill (2,0) circle (2.3pt); 
\fill (4,0) circle (2.3pt); 
\fill (5,0) circle (2.3pt); 
\draw (1,0)--(2,0)--(2.5,0);
\draw [dashed] (2.5,0)--(3.5,0);
\draw (3.5,0)--(4,0)--(5,0);
\draw (-1.8,0) node {$B_n$, $(n\ge2)$:};	
\end{scope}

% D_n
\begin{scope}[xshift=-5cm,yshift=-7.75cm,scale=1.33] 
\fill (0,0.5) circle (2.3pt); 
\fill (0,-0.5) circle (2.3pt); 
\fill (1,0) circle (2.3pt); 
\fill (2,0) circle (2.3pt); 
\fill (4,0) circle (2.3pt); 
\fill (5,0) circle (2.3pt); 
\draw (0,0.5)--(1,0)--(0,-0.5);
\draw (1,0)--(2,0)--(2.5,0);
\draw [dashed] (2.5,0)--(3.5,0);
\draw (3.5,0)--(4,0)--(5,0);
\draw (-1.8,0) node {$D_n$, $(n\ge4)$:};	
\end{scope}

% E_6
\begin{scope}[xshift=5.5cm,yshift=-3.5cm,scale=1.33] 
\fill (0,0) circle (2.3pt); 
\fill (1,0) circle (2.3pt); 
\fill (2,0) circle (2.3pt); 
\fill (3,0) circle (2.3pt); 
\fill (4,0) circle (2.3pt); 
\fill (2,-1) circle (2.3pt); 
\draw (0,0)--(1,0)--(2,0)--(3,0)--(4,0);
\draw (2,0)--(2,-1);
\draw (-1,0) node {$E_6$:};	
\end{scope}

% E_7
\begin{scope}[xshift=5.5cm,yshift=-6cm,scale=1.33] 
\fill (0,0) circle (2.3pt); 
\fill (1,0) circle (2.3pt); 
\fill (2,0) circle (2.3pt); 
\fill (3,0) circle (2.3pt); 
\fill (4,0) circle (2.3pt); 
\fill (5,0) circle (2.3pt); 
\fill (2,-1) circle (2.3pt); 
\draw (0,0)--(1,0)--(2,0)--(3,0)--(4,0)--(5,0);
\draw (2,0)--(2,-1);
\draw (-1,0) node {$E_7$:};	
\end{scope}

% E_8
\begin{scope}[xshift=5.5cm,yshift=-8.5cm,scale=1.33] 
\fill (0,0) circle (2.3pt); 
\fill (1,0) circle (2.3pt); 
\fill (2,0) circle (2.3pt); 
\fill (3,0) circle (2.3pt); 
\fill (4,0) circle (2.3pt); 
\fill (5,0) circle (2.3pt); 
\fill (6,0) circle (2.3pt); 
\fill (2,-1) circle (2.3pt); 
\draw (0,0)--(1,0)--(2,0)--(3,0)--(4,0)--(5,0)--(6,0);
\draw (2,0)--(2,-1);
\draw (-1,0) node {$E_8$:};	
\end{scope}

% F_4
\begin{scope}[xshift=-5cm,yshift=-10cm,scale=1.33] 
%\draw [double] (1,0)--(2,0);
\draw (1,0)--(2,0); \draw (1.5,0.3) node {\small$4$};
\fill (0,0) circle (2.3pt); 
\fill (1,0) circle (2.3pt); 
\fill (2,0) circle (2.3pt); 
\fill (3,0) circle (2.3pt); 
\draw (0,0)--(1,0);
\draw (2,0)--(3,0);
\draw (-1,0) node {$F_4$:};	
\end{scope}

% H_4
\begin{scope}[xshift=-5cm,yshift=-12cm,scale=1.33] 
\fill (0,0) circle (2.3pt); 
\fill (1,0) circle (2.3pt); 
\fill (2,0) circle (2.3pt); 
\fill (3,0) circle (2.3pt); 
\draw (0,0)--(1,0)--(2,0)--(3,0);
\draw (0.5,0.3) node {\small$5$};
\draw (-1,0) node {$H_4$:};	
\end{scope}

% H_3
\begin{scope}[xshift=2.75cm,yshift=-12cm,scale=1.33] 
\fill (0,0) circle (2.3pt); 
\fill (1,0) circle (2.3pt); 
\fill (2,0) circle (2.3pt); 
\draw (0,0)--(1,0)--(2,0);
\draw (0.5,0.3) node {\small5};
\draw (-1,0) node {$H_3$:};	
\end{scope}

% I_2(m)
\begin{scope}[xshift=11.75cm,yshift=-12cm,scale=1.33] 
\fill (0,0) circle (2.3pt); 
\fill (1,0) circle (2.3pt); 
\draw (0,0)--(1,0);
\draw (0.5,0.25) node {\small $m$};
\draw(-2.2,0.25) node {$I_2(m),$};
\draw(-2.2,-0.25) node {\small$(m\ge5,\, m\ne\infty)$};
\draw(-0.5,0) node {:};
\end{scope}

\end{tikzpicture}
\end{center}
\caption{\small{Dynkin diagrams for the irreducible spherical Coxeter systems.}\label{fig:spherical}}
\end{figure}

\begin{figure}
\begin{center}
\begin{tikzpicture}[thick, scale=0.7, double distance=2.3pt]
% ~A_n
\begin{scope}[xshift=-5cm,yshift=-3.5cm,scale=1.33] 
\fill (0,0) circle (2.3pt); 
\fill (1,0) circle (2.3pt); 
\fill (3,0) circle (2.3pt); 
\fill (4,0) circle (2.3pt); 
\fill (2,1) circle (2.3pt); 
\fill [white] (2,2) circle (2.4pt);
\draw (0,0)--(1,0)--(1.5,0);
\draw [dashed] (1.5,0)--(2.5,0);
\draw (2.5,0)--(3,0)--(4,0);
\draw (0,0)--(2,1)--(4,0);
\draw (-1.8,0) node {$\widetilde A_n$, $(n\ge2)$:};	
\end{scope}

% ~A_1
\begin{scope}[xshift=7cm,yshift=-3.5cm,scale=1.33] 
\fill (0,0) circle (2.3pt); 
\fill (1,0) circle (2.3pt); 
\draw (0,0)--(1,0);
\draw (0.5,0.25) node {\small $\infty$};
\draw (-1,0) node {$\widetilde A_1$:};	
\end{scope}

% ~B_n
\begin{scope}[xshift=-5cm,yshift=-5.5cm,scale=1.33] 
%\draw [double] (3,0)--(4,0);
\draw (3,0)--(4,0); \draw (3.5,0.3) node {\small $4$};
\fill (0,0.5) circle (2.3pt); 
\fill (0,-0.5) circle (2.3pt); 
\fill (1,0) circle (2.3pt); 
\fill (3,0) circle (2.3pt); 
\fill (4,0) circle (2.3pt); 
\draw (0,0.5)--(1,0)--(0,-0.5);
\draw (1,0)--(1.5,0);
\draw [dashed] (1.5,0)--(2.5,0);
\draw (2.5,0)--(3,0);
\draw (-1.8,0) node {$\widetilde B_n$, $(n\ge4)$:};	
\end{scope}

% ~B_3
\begin{scope}[xshift=7cm,yshift=-5.5cm,scale=1.33] 
%\draw [double] (1,0)--(2,0);
\draw (1,0)--(2,0); \draw (1.5,0.3) node {\small $4$};
\fill (0,0.5) circle (2.3pt); 
\fill (0,-0.5) circle (2.3pt); 
\fill (1,0) circle (2.3pt); 
\fill (2,0) circle (2.3pt); 
\draw (0,0.5)--(1,0)--(0,-0.5);
\draw (-1,0) node {$\widetilde B_3$:};	
\end{scope}

% ~C_n
\begin{scope}[xshift=-5cm,yshift=-7.5cm,scale=1.33] 
%\draw [double] (0,0)--(1,0);
%\draw [double] (3,0)--(4,0);
\draw (0,0)--(1,0); \draw (0.5,0.3) node {\small $4$};
\draw (3,0)--(4,0); \draw (3.5,0.3) node {\small $4$};
\fill (0,0) circle (2.3pt); 
\fill (1,0) circle (2.3pt); 
\fill (3,0) circle (2.3pt); 
\fill (4,0) circle (2.3pt); 
\draw (1,0)--(1.5,0);
\draw [dashed] (1.5,0)--(2.5,0);
\draw (2.5,0)--(3,0);
\draw (-1.8,0) node {$\widetilde C_n$, $(n\ge3)$:};	
\end{scope}

% ~C_2
\begin{scope}[xshift=7cm,yshift=-7.5cm,scale=1.33] 
%\draw [double] (0,0)--(1,0)--(2,0);
\draw (0,0)--(1,0)--(2,0);
\draw (0.5,0.3) node {\small $4$}; \draw (1.5,0.3) node {\small $4$};
\fill (0,0) circle (2.3pt); 
\fill (1,0) circle (2.3pt); 
\fill (2,0) circle (2.3pt); 
\draw (-1,0) node {$\widetilde C_2$:};
\end{scope}

% ~D_n
\begin{scope}[xshift=-5cm,yshift=-9.5cm,scale=1.33] 
\fill (0,0.5) circle (2.3pt); 
\fill (0,-0.5) circle (2.3pt); 
\fill (1,0) circle (2.3pt); 
\fill (3,0) circle (2.3pt); 
\fill (4,0.5) circle (2.3pt); 
\fill (4,-0.5) circle (2.3pt); 
\draw (0,0.5)--(1,0)--(0,-0.5);
\draw (4,0.5)--(3,0)--(4,-0.5);
\draw (1,0)--(1.5,0);
\draw [dashed] (1.5,0)--(2.5,0);
\draw (2.5,0)--(3,0);
\draw (-1.8,0) node {$\widetilde D_n$, $(n\ge5)$:};	
\end{scope}

% ~D_4
\begin{scope}[xshift=7cm,yshift=-9.5cm,scale=1.33] 
\fill (0,0.5) circle (2.3pt); 
\fill (0,-0.5) circle (2.3pt); 
\fill (1,0) circle (2.3pt); 
\fill (2,0.5) circle (2.3pt); 
\fill (2,-0.5) circle (2.3pt); 
\draw (0,0.5)--(1,0)--(0,-0.5);
\draw (2,0.5)--(1,0)--(2,-0.5);
\draw (-1,0) node {$\widetilde D_4$:};	
\end{scope}

% ~E_6
\begin{scope}[xshift=-7.25cm,yshift=-12cm,scale=1.33] 
\fill (0,0) circle (2.3pt); 
\fill (1,0) circle (2.3pt); 
\fill (2,0) circle (2.3pt); 
\fill (3,0) circle (2.3pt); 
\fill (4,0) circle (2.3pt); 
\fill (2,-1) circle (2.3pt); 
\fill (2,-2) circle (2.3pt); 
\draw (0,0)--(1,0)--(2,0)--(3,0)--(4,0);
\draw (2,0)--(2,-1)--(2,-2);
\draw (-1,0) node {$\widetilde E_6$:};	
\end{scope}

% ~E_7
\begin{scope}[xshift=3.35cm,yshift=-12cm,scale=1.33] 
\fill (0,0) circle (2.3pt); 
\fill (1,0) circle (2.3pt); 
\fill (2,0) circle (2.3pt); 
\fill (3,0) circle (2.3pt); 
\fill (4,0) circle (2.3pt); 
\fill (5,0) circle (2.3pt); 
\fill (6,0) circle (2.3pt);
\fill (3,-1) circle (2.3pt); 
\draw (0,0)--(1,0)--(2,0)--(3,0)--(4,0)--(5,0)--(6,0);
\draw (3,0)--(3,-1);
\draw (-1,0) node {$\widetilde E_7$:};	
\end{scope}

% ~F_4
\begin{scope}[xshift=6cm,yshift=-15cm,scale=1.33] 
%\draw [double] (1,0)--(2,0);
\draw (1,0)--(2,0);
\draw (1.5,0.3) node {\small $4$};
\fill (0,0) circle (2.3pt); 
\fill (1,0) circle (2.3pt); 
\fill (2,0) circle (2.3pt); 
\fill (3,0) circle (2.3pt); 
\fill (4,0) circle (2.3pt); 
\draw (0,0)--(1,0);
\draw (2,0)--(3,0)--(4,0);
\draw (-1,0) node {$\widetilde F_4$:};	
\end{scope}

% ~E_8
\begin{scope}[xshift=-7.25cm,yshift=-16cm,scale=1.33] 
\fill (0,0) circle (2.3pt); 
\fill (1,0) circle (2.3pt); 
\fill (2,0) circle (2.3pt); 
\fill (3,0) circle (2.3pt); 
\fill (4,0) circle (2.3pt); 
\fill (5,0) circle (2.3pt); 
\fill (6,0) circle (2.3pt); 
\fill (7,0) circle (2.3pt); 
\fill (2,-1) circle (2.3pt);
\draw (0,0)--(1,0)--(2,0)--(3,0)--(4,0)--(5,0)--(6,0)--(7,0);
\draw (2,0)--(2,-1);
\draw (-1,0) node {$\widetilde E_8$:};	
\end{scope}

% ~G_2
\begin{scope}[xshift=8.65cm,yshift=-17cm,scale=1.33] 
\fill (0,0) circle (2.3pt); 
\fill (1,0) circle (2.3pt); 
\fill (2,0) circle (2.3pt); 
\draw (0,0)--(1,0)--(2,0);
\draw (0.5,0.3) node {\small $6$};
\draw (-1,0) node {$\widetilde G_2$:};	
\end{scope}

\end{tikzpicture}
\end{center}
\caption{\small{Dynkin diagrams for the irreducible affine Coxeter systems. Note that the rank of an affine system is one more than the subscript of the system's name.}\label{fig:affine}}
\end{figure}
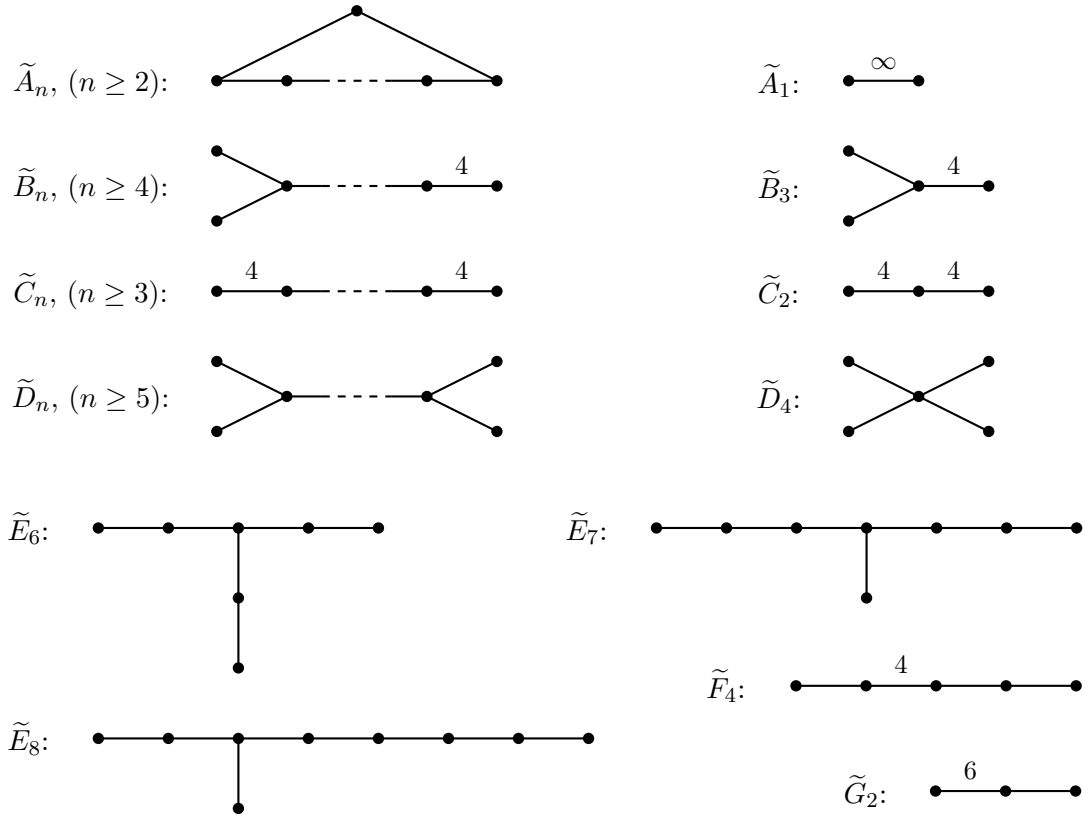

Given $T \subseteq S$, we denote by $W_T$ the subgroup of $W$ generated by $T$, and call $W_T$ a \emph{special subgroup}.  The pair $(W_T,T)$ is a Coxeter system. We say $T$ is \emph{spherical} if $W_T$ is finite.  We will sometimes abuse terminology by saying that a subset $T \subseteq S$ is \emph{irreducible} if $(W_T,T)$ is irreducible, and \emph{affine} if $(W_T,T)$ is affine.  For $A, B \subseteq S$, we denote by $A \times B$ the set $A\cup B$ where $A$ and $B$ are disjoint commuting subsets of $S$. 
Given $T \subseteq S$ we write $T^\perp$ for the set of elements of $S \setminus T$ which commute with all elements of $T$.  To simplify notation, if $s \in S$ we write $s^\perp$ instead of $\{s\}^\perp$.

A subset $A \subseteq S$ is \emph{minimal nonspherical} if $A$ is nonspherical but every proper subset of $A$ is spherical. Notice that if $A\subseteq S$ is minimal nonspherical then any subset $B\subseteq S$ commuting with~$A$ is disjoint from $A$.
(Indeed, if the decomposition $A=(A\setminus B)\sqcup(A\cap B)$ is nontrivial, then due to minimality of $A$, both subsets $A\setminus B$ and $A\cap B$ are spherical and commuting, and  hence their union $A$ must be spherical, which yields a contradiction.)

\begin{remark}\label{rem:minimalRA} In a right-angled system the minimal nonspherical subsets are exactly the pairs of nonadjacent vertices in the defining graph, and a nonempty subset is spherical exactly when it is the vertex set of a clique in the defining graph. 
\end{remark}

\begin{remark}\label{rem:minimal}
For general $(W,S)$, if $A \subseteq S$ is minimal nonspherical then $W_A$ is an (infinite)  simplicial Coxeter group in the sense discussed in~\cite[Section~6.9 and Example~14.2.3]{davis-book}, meaning that~$W_A$ is generated by the set $A$ of reflections in the faces of a compact simplex in either Euclidean or hyperbolic space.  Hence $(W_A,A)$ is either irreducible affine (see Figure~\ref{fig:affine}) or irreducible hyperbolic as classified by Lann\'er (see~\cite[Table 6.2]{davis-book}, which we reproduce here in Figure~\ref{fig:lanner}). We also note that all these Coxeter systems are $1$-ended, except that of type $\widetilde A_1$, which is virtually $\Z$ and hence is $2$-ended.
\end{remark}

\begin{figure}
\begin{center}
\begin{tikzpicture}[thick, scale=0.7, double distance=2.3pt]
\begin{scope}[xshift=-6.25cm]
\node[draw=none,minimum size=1cm,regular polygon,regular polygon sides=3] (a) {};
\draw (a.corner 1)--(a.corner 2);
\draw (a.corner 2)--(a.corner 3);
\draw (a.corner 3)--(a.corner 1);
\foreach \x in {1,2,3}
  \fill (a.corner \x) circle[radius=3pt];
\draw (-0.7,0.3) node {$p$};
\draw (0.7,0.3) node {$q$};
\draw (0,-0.7) node {$r$};
\draw (6.5,0) node {with\,\, $\dfrac1{\,p\,}+\dfrac1{\,q\,}+\dfrac1{\,r\,}<1$,\,\, including};
\begin{scope}[xshift=10.5cm, scale=1.33]
\fill (1,0) circle (2.3pt);
\fill (2,0) circle (2.3pt);
\fill (3,0) circle (2.3pt);
\draw (1,0)--(2,0)--(3,0);
\draw (1.5,-0.3) node {$p$};
\draw (2.5,-0.3) node {$q$};
\draw (5.5,0) node {with\,\, $\dfrac1{\,p\,}+\dfrac1{\,q\,}<\dfrac1{\,2\,}$};
\end{scope}
\end{scope}

\begin{scope}[xshift=-2cm,yshift=-1cm]
% o----o--5--o----o
\begin{scope}[xshift=-5cm,yshift=-2cm,scale=1.33] 
\fill (0,0) circle (2.3pt);
\fill (1,0) circle (2.3pt);
\fill (2,0) circle (2.3pt);
\fill (3,0) circle (2.3pt);
\draw (1.5,0.3) node {5};
\draw (0,0)--(1,0)--(2,0)--(3,0);
\end{scope}

% o--5--o----o--4--o
\begin{scope}[xshift=-5cm,yshift=-3.5cm,scale=1.33] 
%\draw [double] (2,0)--(3,0);
\draw (2,0)--(3,0); \draw (2.5,0.3) node {4};
\fill (0,0) circle (2.3pt);
\fill (1,0) circle (2.3pt);
\fill (2,0) circle (2.3pt);
\fill (3,0) circle (2.3pt);
\draw (0.5,0.3) node {5};
						 
\draw (0,0)--(1,0)--(2,0);
\end{scope}

% o--5--o----o--5--o
\begin{scope}[xshift=-5cm,yshift=-5cm,scale=1.33] 
\fill (0,0) circle (2.3pt);
\fill (1,0) circle (2.3pt);
\fill (2,0) circle (2.3pt);
\fill (3,0) circle (2.3pt);
\draw (0.5,0.3) node {5};
\draw (2.5,0.3) node {5};
\draw [thick] (0,0)--(1,0)--(2,0)--(3,0);
\end{scope}

%       o 
%       |
% o--5--o---o
\begin{scope}[xshift=-5cm,yshift=-7cm,scale=1.33] 
\fill (0,0) circle (2.3pt);
\fill (1,0) circle (2.3pt);
\fill (1.866,0.5) circle (2.3pt);
\fill (1.866,-0.5) circle (2.3pt);
\draw (0.5,0.3) node {5};
\draw [thick] (0,0)--(1,0)--(1.866,0.5);
\draw [thick] (1,0)--(1.866,-0.5);
\end{scope}
\end{scope}

\begin{scope}[xshift=2cm, yshift=-4.25cm]
% 3-4 square
\begin{scope}[xshift=1cm, scale=1.33]
%\draw [double] (0,0)--(1,0);
\draw (0,0)--(1,0); \draw (0.5,0.3) node {4};
\fill (0,0) circle (2.3pt);
\fill (1,0) circle (2.3pt);
\fill (0,1) circle (2.3pt);
\fill (1,1) circle (2.3pt);
\draw (1,0)--(1,1)--(0,1)--(0,0);
\end{scope}

% 3-5 square
\begin{scope}[xshift=3.5cm,scale=1.33]
\fill (0,0) circle (2.3pt);
\fill (1,0) circle (2.3pt);
\fill (0,1) circle (2.3pt);
\fill (1,1) circle (2.3pt);
\draw [thick] (0,0)--(1,0)--(1,1)--(0,1)--(0,0);
\draw (0.5,0.3) node {5};
\end{scope}

% 4-4 square
\begin{scope}[xshift=6cm,scale=1.33]
%\draw [double] (0,0)--(1,0);
%\draw [double] (0,1)--(1,1);
\draw (0,0)--(1,0); \draw (0.5,0.3) node {4};
\draw (0,1)--(1,1); \draw (0.5,1.3) node {4};
\fill (0,0) circle (2.3pt);
\fill (1,0) circle (2.3pt);
\fill (0,1) circle (2.3pt);
\fill (1,1) circle (2.3pt);
\draw (1,0)--(1,1); \draw (0,1)--(0,0);
\end{scope}

% 5-5 square
\begin{scope}[xshift=2.25cm,yshift=-3cm,scale=1.33]
\fill (0,0) circle (2.3pt);
\fill (1,0) circle (2.3pt);
\fill (0,1) circle (2.3pt);
\fill (1,1) circle (2.3pt);
\draw [thick] (0,0)--(1,0)--(1,1)--(0,1)--(0,0);
\draw (0.5,0.3) node {5};
\draw (0.5,1.3) node {5};
\end{scope}

% 4-5 square
\begin{scope}[xshift=4.75cm,yshift=-3cm,scale=1.33]
%\draw [double] (0,1)--(1,1);
\draw (0,1)--(1,1); \draw (0.5,1.3) node {4};
\fill (0,0) circle (2.3pt);
\fill (1,0) circle (2.3pt);
\fill (0,1) circle (2.3pt);
\fill (1,1) circle (2.3pt);
\draw (0,1)--(0,0)--(1,0)--(1,1);
\draw (0.5,0.3) node {5};
\end{scope}
\end{scope}

\begin{scope}[xshift=-7cm,yshift=-11cm]
% o--5--o----o----o----o
\begin{scope}[scale=1.33] 
\fill (0,0) circle (2.3pt);
\fill (1,0) circle (2.3pt);
\fill (2,0) circle (2.3pt);
\fill (3,0) circle (2.3pt);
\fill (4,0) circle (2.3pt);
\draw (0.5,0.3) node {5};
\draw [thick] (0,0)--(1,0)--(2,0)--(3,0)--(4,0);
\end{scope}

% o--5--o----o----o--4--o
\begin{scope}[yshift=-1.5cm,scale=1.33] 
%\draw [double] (3,0)--(4,0);
\draw (3,0)--(4,0); \draw (3.5,0.3) node {4};
\fill (0,0) circle (2.3pt);
\fill (1,0) circle (2.3pt);
\fill (2,0) circle (2.3pt);
\fill (3,0) circle (2.3pt);
\fill (4,0) circle (2.3pt);
\draw (0.5,0.3) node {5};
\draw (0,0)--(1,0)--(2,0)--(3,0);
\end{scope}

% o--5--o----o----o--5--o
\begin{scope}[yshift=-3cm,scale=1.33] 
\fill (0,0) circle (2.3pt);
\fill (1,0) circle (2.3pt);
\fill (2,0) circle (2.3pt);
\fill (3,0) circle (2.3pt);
\fill (4,0) circle (2.3pt);
\draw (0.5,0.3) node {5};
\draw (3.5,0.3) node {5};
\draw [thick] (0,0)--(1,0)--(2,0)--(3,0)--(4,0);
\end{scope}

% pentagon
\begin{scope}[xshift=10cm,yshift=-1.5cm]
\node[draw=none,minimum size=1.752cm,regular polygon,regular polygon sides=5] (a) {};
\draw (a.corner 1)--(a.corner 2);
\draw (a.corner 2)--(a.corner 3);
\draw (a.corner 3)--(a.corner 4); \draw (a.corner 3) node [shift={(3*360/5+175:0.6)}] {$4$};
\draw (a.corner 4)--(a.corner 5);
\draw (a.corner 5)--(a.corner 1);
\foreach \x in {1,2,...,5}
  \fill (a.corner \x) circle[radius=3pt];
\end{scope}

%            o 
%            |
% o--5--o----o---o
\begin{scope}[xshift=13.5cm,yshift=-1.5cm,scale=1.33] 
\fill (0,0) circle (2.3pt);
\fill (1,0) circle (2.3pt);
\fill (2,0) circle (2.3pt);
\fill (2.866,0.5) circle (2.3pt);
\fill (2.866,-0.5) circle (2.3pt);
\draw (0.5,0.3) node {5};
\draw [thick] (0,0)--(1,0)--(2,0)--(2.866,0.5);
\draw [thick] (2,0)--(2.866,-0.5);
\end{scope}
\end{scope}
\end{tikzpicture}
\end{center}
\caption{\small{Dynkin diagrams of Lann\'er hyperbolic Coxeter systems. Together with the irreducible affine Coxeter systems (Figure~\ref{fig:affine}), they form the class of minimal nonspherical systems; see Remark~\ref{rem:minimal}.}
\label{fig:lanner}}
\end{figure}

Any Coxeter system $(W,S)$ has an  associated Davis complex $\Sigma = \Sigma(W,S)$.  The $1$-skeleton of $\Sigma$ is the Cayley graph $\cC = \cC(W,S)$ of $W$ with respect to the generating set $S$, with a single unoriented edge connecting vertices $w$ and $ws$ for each $w \in W$ and $s \in S$.  The cells of~$\Sigma$ then have vertex sets $w W_T$ where $w \in W$ and $W_T$ is a spherical special subgroup of $W$.  We now metrise $\Sigma$ so that each edge has length $1$, and each cell of $\Sigma$ is a compact Euclidean polytope (with all edges of length $1$).  In particular, each $2$-cell of $\Sigma$ is a regular Euclidean $2m$-gon with $2 \leq m < \infty$ and edges labelled alternately by a pair $s,t \in S$ with $m_{st} = m$.
 Equipped with this metric, $\Sigma$ is a $\CAT(0)$ space (see~\cite[Theorem~12.3.3]{davis-book}), and the group $W$ is quasi-isometric to $\Sigma$ (see~\cite[Proposition~B.5.2]{Kra}).

 A  \emph{reflection} in $W$ is a conjugate of an element of $S$, and a \emph{wall} in $\Sigma$ is the fixed set of a reflection. We say that an edge in $\cC$ is \emph{dual} to a wall $H$ if its midpoint is contained in $H$, and that a wall in $\Sigma$ is \emph{dual} to an edge $e$ if it contains the midpoint of $e$.
 We then say that a path in $\cC$ \emph{crosses}~$H$ if it contains at least one edge which is dual to $H$. Each wall~$H$ separates $\Sigma$ into exactly two components, and the two \emph{half-spaces} defined by $H$ are the closures of the components of $\Sigma \setminus H$ (see Section~3.2 of~\cite{davis-book}). For any two subsets $A,B\subset \Sigma$, we say that a wall $H$ \emph{separates $A$ from $B$} if $A$ and $B$ lie in different components of $\Sigma\setminus H$.

\subsection{Hyperbolicity and thickness for Coxeter groups}\label{sec:hypCoxeter}

We now recall the characterisations of hyperbolicity, relative hyperbolicity and strong algebraic thickness for Coxeter groups.  

For hyperbolicity we have the following theorem of Moussong.

\begin{thm}[Moussong, see Corollary~12.6.3 of \cite{davis-book}]\label{thm:Moussong}
Let $(W,S)$ be a Coxeter system.  Then~$W$ is hyperbolic if and only if there is no subset $T \subseteq S$ satisfying either of the following two conditions:
\begin{enumerate}
    \item $(W_T,T)$ is an irreducible affine Coxeter system of rank $\geq 3$.
    \item $(W_T,T) = (W_{T_1},T_1) \times (W_{T_2},T_2)$ where both $W_{T_1}$ and $W_{T_2}$ are infinite.\qed
\end{enumerate}
\end{thm}

Behrstock, Hagen, Sisto and Caprace prove in the Appendix to~\cite{behrstock-hagen-sisto-caprace}
that every Coxeter group admits a canonical minimal  relatively hyperbolic structure, whose peripheral subgroups are special subgroups.  Moreover, the families of special subgroups which can serve as peripheral subgroups for a relatively hyperbolic structure are characterised by Caprace in~\cite{caprace,caprace-erratum}~as follows. 

\begin{thm}[Theorem~A$'$ of~\cite{caprace-erratum}]\label{thm:relHypCoxeter}
Let $(W,S)$ be a Coxeter system and $\cT$ be a collection of proper subsets of $S$. Then $W$ is hyperbolic relative to $\{W_T\mid T\in\cT\}$ if and only if the following conditions hold:
\begin{itemize}
\item[(RH1)] For each subset $T\subseteq S$ such that $T$ is irreducible affine of cardinality $\geq 3$, there exists $T'\in\cT$ such that $T \subseteq T'$. Given any pair of irreducible nonspherical subsets $T_1$, $T_2\subseteq S$ with $[T_1,T_2]=1$, there exists $T\in\cT$ such that $T_1\cup T_2\subseteq T$.
\item[(RH2)] For all $T_1, T_2\in\cT$ with $T_1\ne T_2$, the intersection $T_1\cap T_2$ is spherical.
\item[(RH3)] For each $T\in\cT$ and each irreducible nonspherical $U\subseteq T$, we have $U^\perp\subseteq T$.\qed
\end{itemize}
\end{thm}

We will also need the classes $\mathbb{T}_0$ and $\mathbb{T}$ of Coxeter systems defined by Behrstock, Hagen, Sisto and Caprace in the Appendix to~\cite{behrstock-hagen-sisto-caprace}.

\begin{definition}[Section~A.1 of~\cite{behrstock-hagen-sisto-caprace}]\label{defn:T}
Define $\bT_0$ to be the class of all Coxeter systems $(W,S)$ such that either:
\begin{enumerate}
    \item[(a)] $(W,S)$ is irreducible affine of rank $\geq 3$; or
    \item[(b)] $(W,S) = (W_{S_1},S_1) \times (W_{S_2},S_2)$ with $(W_{S_i},S_i)$ irreducible and nonspherical for $i = 1,2$.
\end{enumerate}
The class $\bT$ of Coxeter systems $(W,S)$ is then defined inductively as follows.
\begin{enumerate}
    \item $\bT$ contains $\bT_0$.
    \item If $S = S_0 \sqcup \{ s \}$ such that $(W_{S_0},S_0) \in \bT$ and $s^\perp \subset S$ is nonspherical, then $(W_S,S) \in \bT$.
    \item If $S = S_1 \cup S_2$ such that $(W_{S_i},S_i) \in \bT$ for $i = 1,2$ and $S_1 \cap S_2$ is nonspherical, then $(W_S,S) \in \bT$.
\end{enumerate}
\end{definition}

Both strong algebraic thickness and not being relatively hyperbolic can then be characterised via the class $\mathbb{T}$:

\begin{thm}[See Corollary~A.10 of~\cite{behrstock-hagen-sisto-caprace}]\label{thm:thickNotRelHyp}
Let $(W,S)$ be a Coxeter system.  The following are equivalent:
\begin{enumerate}
    \item $(W,S)$ is in $\bT$;
    \item $W$ is strongly algebraically thick;
    \item $W$ is not relatively hyperbolic with respect to any family of proper subgroups.\qed
\end{enumerate}
\end{thm}

\section{Linear and quadratic divergence}\label{sec:linearQuadratic}

In this short section we prove Theorem~\ref{thm:linearQuadraticIntro} of the introduction, which provides a quadratic lower bound on divergence for irreducible nonaffine Coxeter groups.  We also prove Corollary~\ref{cor:linearIntro} 
of the introduction, which characterises linear divergence for Coxeter groups.   
 
Let $(W,S)$ be a Coxeter system.  Recall that a \emph{rank one geodesic} in a $\CAT(0)$ space is a geodesic line which does not bound a flat half-plane.  An element of $W$ is then said to be \emph{rank one} if it is hyperbolic in its action on the Davis complex $\Sigma = \Sigma(W,S)$ and some (hence any) of its axes in $\Sigma$ are rank one geodesics.  We emphasise that we are considering geodesic lines in $\Sigma$ in this definition, not combinatorial geodesics in the Cayley graph $\cC = \cC(W,S)$ (since these are not in general geodesic in the $\CAT(0)$ metric on $\Sigma$).  Recall also that a \emph{Coxeter element} of $W$ is a product of the elements of~$S$, in any order.  A  \emph{parabolic subgroup} of $W$ is a conjugate of a special subgroup, and the \emph{parabolic closure} of an element $w \in W$ is the smallest parabolic subgroup of $W$ which contains $w$.  The next result follows from work of Caprace--Fujiwara~\cite{caprace-fujiwara}.

\begin{prop}\label{prop:CoxeterRank1}  If $(W,S)$ is irreducible, nonspherical and nonaffine, then 
any Coxeter element of $W$ is a rank one element.
\end{prop}

\begin{proof}
By, for example, Corollary~4.3 of~\cite{caprace-fujiwara}, the parabolic closure of a Coxeter element of $W$ is equal to $W$.  Now by Proposition~4.5 of~\cite{caprace-fujiwara}, for general Coxeter groups $W$, an element $w \in W$ is \emph{not} of  rank one if and only if $w \in P$, where $P$ is parabolic with either (i) $P$ finite, or (ii) $P = P_1 \times P_2$ with $P_i$ both infinite parabolic, or (iii) $P = K \times P_{\mbox{\tiny{aff}}}$ where $K$ is a finite parabolic (possibly trivial) and $P_{\mbox{\tiny{aff}}}$ is  affine parabolic of rank at least $3$.  So if we take~$w$ to be a Coxeter element of $W$ with $(W,S)$ infinite, irreducible and nonaffine, then $w$ is a rank one element.  
\end{proof}

We now use the following general result due to Kapovich and Leeb~\cite{kapovich-leeb} (see also~\cite{kapovich-kleiner-leeb}).  We note that since $W$ is finitely generated, the Davis complex $\Sigma$ is locally compact.

\begin{prop}[Proposition~3.3 of~\cite{kapovich-leeb} or Proposition~4.5 of~\cite{kapovich-kleiner-leeb}]\label{prop-kap-leeb}  Let $X$ be a locally compact $\CAT(0)$ space and let $\gamma$ be a complete geodesic which is invariant under a cyclic group of hyperbolic isometries.  If $\gamma$ has subquadratic divergence (in the sense of Remark~\ref{rem:subquadratic}), then it bounds a flat half-plane and hence has linear divergence.\qed
\end{prop}

\begin{remark}\label{rem:subquadratic}
The authors of \cite{kapovich-leeb,kapovich-kleiner-leeb} use a slightly different notion of the subquadratic divergence of a geodesic than we do. Namely, they say that a geodesic $\gamma$ has subquadratic divergence if $\lim_{r\to\infty}\alpha(r)/r^2=0$, where $\alpha(r)$ is the infimum of lengths of $r$-avoidant paths connecting $\gamma(r)$ and $\gamma(-r)$. However, a careful analysis of the proofs of Proposition~3.3 of~\cite{kapovich-leeb} and Proposition~4.5 of~\cite{kapovich-kleiner-leeb} shows that an even weaker condition on $\gamma$ would yield the same conclusion, namely, that
\begin{quote}{\it 
there exists a sequence $r_k\to\infty$ such that $\lim_{k\to \infty}\alpha(r_k)/r_k^2=0$.}
\end{quote}
The negation of this condition will then imply that there exist real numbers $C,N>0$ such that for all $r>N$ we have $\alpha(r)/r^2>C$. For the purposes of Proposition~\ref{prop-kap-leeb} and the proof of Theorem~\ref{thm:linearQuadraticIntro} we will understand the term `subquadratic divergence' in the above italicised sense.
\end{remark}

Now we can prove Theorem~\ref{thm:linearQuadraticIntro}.

\begin{proof}[Proof of Theorem~\ref{thm:linearQuadraticIntro}]
Suppose the Coxeter system $(W,S)$ is irreducible and nonaffine and the group $W$ is $1$-ended. Choose a Coxeter element $w\in W$ and let $\gamma$ be an axis for $w$ (recall that by Proposition~\ref{prop:CoxeterRank1}, $w$ is a rank one element). Then $\gamma$ is invariant under the infinite cyclic group $\langle w \rangle$ of hyperbolic isometries. If the geodesic $\gamma$ has subquadratic divergence (in the sense of Remark~\ref{rem:subquadratic}) then by Proposition~\ref{prop-kap-leeb}, $\gamma$ bounds a flat half-plane, which contradicts $w$ being rank one. Hence~$\gamma$ does not have subquadratic divergence, and by Remark~\ref{rem:subquadratic} we conclude that $\div{\gamma}(r)\succeq r^2$.  
This proves that the divergence of the Davis complex $\Sigma$ is at least quadratic as well. Since $W$ is quasi-isometric to $\Sigma$ (see~\cite[Proposition~B.5.2]{Kra}), and since the divergence function is invariant under quasi-isometries \cite[Proposition~2.1]{gersten-quadratic}, we conclude that $W$ also has at least quadratic divergence, as required.
\end{proof}

\medskip
The following result establishes Corollary~\ref{cor:linearIntro} 
of the introduction.

\begin{cor}\label{cor:linear} Let $(W,S)$ be a Coxeter system such that $W$ is $1$-ended.  Then $W$ has linear divergence if and only if 
$(W,S)=(W_1,S_1) \times (W_2,S_2)$ where either:
\begin{enumerate}
\item both $W_1$ and $W_2$ are infinite; or
\item $W_1$ is finite (possibly trivial) and $(W_2,S_2)$ is irreducible affine of rank $\ge 3$.
\end{enumerate}
\end{cor}

\begin{proof}  We first note that $W$ itself must be infinite since finite groups are $0$-ended.  
Since the divergence of $W$ is invariant under passing to a subgroup of finite index, for any decomposition $W=W_1\times W_2$ with $W_1$ finite, the divergence of $W$ and of $W_2$ will be the same. Thus without loss of generality we may assume that either:
\begin{enumerate}
    \item $(W,S)=(W_1,S_1)\times (W_2,S_2)$ with both $W_1$ and $W_2$ infinite; or
    \item $W$ is irreducible.
\end{enumerate}

In the first case, write $\ell_i$ for the word length on $W_i$ with respect to $S_i$, for $i = 1,2$.  Then since~$W_i$ is infinite, for all $w_i \in W_i$ there is at least one $s \in S_i$ such that $\ell_i(w_is) > \ell_i(w_i)$ (this follows from Lemmas 4.7.2 and 4.7.3 of~\cite{davis-book}).  Hence any $w_i \in W_i$ lies on an infinite geodesic ray based at the identity vertex in the Cayley graph of $W_i$ with respect to $S_i$.  We can then apply the proof of Lemma~7.2 of~\cite{abddy} to conclude that $W = W_1 \times W_2$ has linear divergence.

In the second case, i.e.\ if $W$ is irreducible, $W$ may be either affine of type $\widetilde A_1$, or affine of rank~$\ge 3$, or nonaffine. If $W$ is of type $\widetilde A_1$ then $W$ is $2$-ended and must be discarded. If $W$ is affine of rank $n+1\ge 3$ then $W$ is virtually $\Z^{n}$ with $n \geq 2$, and hence $W$ has linear divergence.

In the remaining case when $W$ is irreducible, nonspherical and nonaffine, Theorem~\ref{thm:linearQuadraticIntro} implies that the divergence of $W$ is at least quadratic.
\end{proof}

\section{Hypergraph index}\label{sec:hypergraph}

In Section~\ref{sec:motiv} we provide motivation and an outline of our construction with a few technicalities omitted.
 In Section~\ref{sec:defn} we define hypergraph index for general Coxeter systems and prove some easy properties.  Our main results for hypergraph index are contained in Section~\ref{sec:hyp_thick_div}, where we prove Theorem~\ref{thm:hyp_thick_div_Intro}.  
Throughout this section, $(W,S)$ is a Coxeter system with $S$ finite. 

\subsection{Motivation and an outline of the construction}\label{sec:motiv}

To motivate our definition of hypergraph index, we first recall that, as was proven in~\cite[Appendix]{behrstock-hagen-sisto-caprace}, every Coxeter group is either hyperbolic relative to a system of proper special peripheral subgroups (and hence has exponential  divergence), or it is thick (and hence has divergence bounded above by a polynomial of integer degree). Moreover,  the inductive procedure of Definition~\ref{defn:T} allows one to either build a minimal system of such peripheral subgroups, or to establish that the entire Coxeter group is thick. An upper bound on the order of thickness (respectively, degree of divergence) can then be calculated, using results of~\cite{behrstock-drutu-mosher}, as the total number of operations (2) and (3) from Definition~\ref{defn:T} which have been applied (respectively, this number plus one).

However not all operations (2) and (3) actually increase the order of thickness (respectively, degree of divergence), and it is  desirable to group several of them together to achieve a better upper bound. In Section~\ref{sec:defn}
we describe a procedure which does exactly that: compacts multiple applications of operations~(2) and~(3) into fewer steps if they do not increase the order of thickness (respectively, degree of divergence). The total number of steps until stabilisation is our invariant, \emph{the hypergraph index}, which we show is a nonnegative integer if the Coxeter group is thick, and equals infinity if the Coxeter group is relatively hyperbolic. 
The construction  
was inspired by the work of Levcovitz~\cite{levcovitz-RACG} who introduced it for the case of right-angled Coxeter groups.

In more detail, the motivation for hypergraph index (as well as for the rules (RH1)--(RH3) of Caprace's Theorem~\ref{thm:relHypCoxeter} and the operations (1)--(3) of Definition~\ref{defn:T}) comes from the following facts about relatively hyperbolic groups:
\begin{itemize}
\item[(i)] If a relatively hyperbolic group has a subgroup isomorphic to $\Z\times\Z$, then it should be contained in a peripheral subgroup (otherwise, \emph{the coned-off Cayley graph} will not be \emph{fine}, in the terminology of~\cite[Definition~2.43]{groves-manning-dehn-rel-hyp}).
\item[(ii)] In a relatively hyperbolic group, the intersection of any two peripheral subgroups or their conjugates should be finite (see e.g.~\cite[Lemma~4.20]{drutu-sapir-tree-graded}).
\end{itemize}

The goal is, for any Coxeter system $(W,S)$, to build a minimal system of peripheral special subgroups of $W$ which satisfies conditions (i) and (ii) above. We start with the smallest collection of subsets of $S$ such that the corresponding special subgroups contain $\Z\times\Z$ in an obvious way. These are the maximal subsets of the form $A\times B\subseteq S$ with $A$ and $B$ both nonspherical (which makes $W_A$ and $W_B$ both infinite), or $A$ irreducible affine of cardinality $\ge3$ and $B$ spherical. We denote this family $\Omega(S)$ and call it the family of \emph{wide subsets}.

After that we need to ensure that condition (ii) is satisfied. We achieve this in an inductive process, by modifying the current family of subsets of $S$ corresponding to prospective peripheral subgroups.  Since the intersection of special subgroups is special, it is easy to see when two members of the current family of subsets of $S$ have nonspherical intersection; in this case, the corresponding special subgroups must be absorbed into a bigger peripheral subgroup.
However it is not always straightforward to decide when, for two subsets of $S$, their corresponding special subgroups have \emph{conjugates} with infinite intersection. Caprace's rule (RH3) from Theorem~\ref{thm:relHypCoxeter} is the key to detecting such situations. Namely, if a subset  $T\subset S$ contains a nonspherical subset~$U$ whose elements commute with some $s\in S \setminus T$, then the subgroups $W_T$ and $sW_Ts^{-1}$ contain the infinite special subgroup $W_U=sW_Us^{-1}$ in their intersection. 
So there are subsets of the form $U\times K\subset S$ with $U$ minimal nonspherical and $K$ maximal spherical commuting with $U$ such that $W_{U}\times W_K$ is not allowed to have nonspherical intersection with any of the peripheral subgroups.
We denote the family of all such $U \times K$ by $\Psi(S)$ and call its elements the \emph{slab subsets}. 

These considerations allow us to keep track of a prospective family of peripheral special subgroups, as follows. 
We start with $\Lambda_0=\Omega(S)\cup\Psi(S)$ and construct a sequence $\Lambda_0$, $\Lambda_1$, $\Lambda_2$, \dots\ of sets of subsets of $S$, where each $\Lambda_i$ consists of unions of those members of $\Lambda_{i-1}$ which are forced to be united due to violation of condition (ii), as described above. We achieve this by introducing an equivalence relation on each of the $\Lambda_i$ generated by the condition that two subsets have nonspherical intersection.

If, for the resulting sequence of subsets $\Lambda_i$, there is an index $h$ such that $\Lambda_h$ contains $S$, then the Coxeter system $(W,S)$ is thick and the minimal such number $h$ is the hypergraph index.  We prove in Theorem~\ref{thm:hyp_thick_div_Intro} that $h$ is an upper bound on the order of (strong) thickness, and hence obtain an upper bound on the degree of divergence. 
However, if the sets $\Lambda_i$ stabilise without reaching~$S$, then it follows from Caprace's Theorem~\ref{thm:relHypCoxeter} that the non-slab elements of $\Lambda_i$ 
form a family of minimal peripheral special subgroups for a relatively hyperbolic structure on $(W,S)$, and~$W$ then has exponential divergence. 

We remark that each set $\Lambda_i$ in the construction can be thought of as a hypergraph with vertex set $S$ and hyperedges corresponding to the subsets in $\Lambda_i$, although we do not use this terminology. This motivates the name   hypergraph index, which was introduced by Levcovitz~\cite{levcovitz-RACG}.

\subsection{Definitions and first properties}\label{sec:defn}

In this section we recall Levcovitz's definition of hypergraph index for right-angled Coxeter groups~\cite[Section~3]{levcovitz-RACG}, in parallel with formulating our generalisation of hypergraph index to arbitrary Coxeter systems (see Definition~\ref{defn:hypergraphIndex}).  We also establish some basic properties of hypergraph index.

First, we define a set of building blocks, called \emph{wide subsystems} and \emph{slab subsystems}, out of which all thick subsystems of $(W,S)$ will be constructed.   
These will respectively generalise the following notions, which were introduced in the right-angled case by Levcovitz~\cite[Definition~3.1]{levcovitz-RACG}.  Let~$\G$ be the defining graph of a right-angled Coxeter group.  The  \emph{wide subgraphs} $\Omega = \Omega(\G)$ are the  maximal induced subgraphs of $\G$ such that every element of $\Omega$ is the join of two induced subgraphs, each of which contains a pair of nonadjacent vertices (equivalently, the vertex sets of these two subgraphs commute and are both nonspherical).  The \emph{strip subgraphs} $\Psi = \Psi(\G)$ are the maximal induced subgraphs of $\G$ such that every element of $\Psi$ is the join of a pair of nonadjacent vertices with a nonempty clique (equivalently, the vertex sets of strip subgraphs generate $2$-ended special subgroups of the form $W_A \times C_2^{|K|}$, where $W_A \cong D_\infty$ and $K$ is the vertex set of a maximal nonempty clique).

We frame our generalisation of these notions in terms of subsets of $S$ rather than induced subgraphs of the defining graph $\G$.
Let $\cP(S)$ be the power set of $S$.

\begin{definition}\label{defn:wide}
Consider the collection of subsets of $S$ of the form $A \times B$ such that either:
\begin{enumerate}
    \item $A$ and $B$ are both nonspherical; or
    \item $A$ is irreducible affine of cardinality $\ge3$ and $B$ is spherical (possibly empty).
\end{enumerate}
The \emph{wide subsets} $\Omega(S) \subseteq \cP(S)$ are the maximal elements of this collection, with respect to the partial order induced by inclusion.  
A \emph{wide subsystem} is a pair $(W_T,T)$ where $T \in \Omega(S)$.
\end{definition}

\begin{remark} Since there are no right-angled irreducible affine Coxeter systems of rank $\geq 3$, when $(W,S)$ is right-angled the elements of $\Omega(S)$ from Definition~\ref{defn:wide} are exactly the vertex sets of the wide subgraphs from~\cite[Definition~3.1]{levcovitz-RACG}.
\end{remark}

The following observation is immediate from Theorem~\ref{thm:Moussong}.

\begin{lemma}\label{lem:hypOmega}
Let $(W,S)$ be a Coxeter system.  Then $W$ is hyperbolic if and only if $\Omega(S) = \varnothing$.
\end{lemma}

Part (1) of the next result generalises~\cite[Remark~3.2]{levcovitz-RACG}, and justifies the use of the term ``wide".  We note that if $T \in \Omega(S)$ then $W_T$ is $1$-ended.

\begin{lemma}\label{lem:linearOmega}  Let $(W,S)$ be a Coxeter system.
\begin{enumerate}
\item The elements $T \in \Omega(S)$ correspond bijectively to the maximal special subgroups of $W$ which are wide. 
\item The group $W$ has linear divergence if and only if $\Omega(S) = \{ S \}$.
\end{enumerate}
\end{lemma}

\begin{proof}  This follows from Proposition~\ref{prop:wide}, Corollary~\ref{cor:linearIntro} and the maximality of elements of $\Omega(S)$.
\end{proof}

In our generalisation of the set $\Psi$ from~\cite[Definition~3.1]{levcovitz-RACG}, we replace the term ``strip" by ``slab", since for us the special subgroups generated by elements of $\Psi(S)$ will not always be $2$-ended.  

\begin{definition}\label{defn:slab}
Consider the collection of subsets of $S$ of the form $A\times K$ such that all of the following hold:
\begin{enumerate}
\item $A$ is minimal nonspherical; 
\item $K$ is a nonempty spherical subset commuting with $A$; and 
\item there does not exist $T \in \Omega(S)$ such that $A\times K \subseteq T$.
\end{enumerate}
The \emph{slab subsets} $\Psi(S) \subseteq \cP(S)$ are the maximal elements of this collection, with respect to the partial order induced by inclusion.  A \emph{slab subsystem} is a pair $(W_T,T)$ where $T \in \Psi(S)$.
\end{definition}

\begin{remark}
It follows from Remark~\ref{rem:minimalRA} that Definition~\ref{defn:slab} also specialises to Levcovitz's in the right-angled case. 
\end{remark}

\begin{remark}\label{rem:slab}
By Remark~\ref{rem:minimal} and the definition of wide subsets, if $A\times K$ is slab, then $(W_A,A)$ is either  irreducible affine of type $\widetilde A_1$, or is irreducible hyperbolic as classified by Lann\'er (and depicted in Figure~\ref{fig:lanner}).
\end{remark}

\begin{remark}\label{rem:irredAffine} 
If $(W,S)$ is irreducible affine or a Lann\'er hyperbolic system, then $\Psi(S) = \varnothing$, since  $S$ is minimal nonspherical, so $S$ cannot have disjoint commuting subsets $A$ and $K$ with $A$ nonspherical and $K \neq \varnothing$.
\end{remark}

Next, we generalise the Lambda hypergraphs introduced by Levcovitz in the right-angled case \cite[Definition~3.5]{levcovitz-RACG}.  The result will be a sequence of subsets $\Lambda_i(S)$ of $\cP(S)$ defined inductively as follows.

\begin{definition}\label{defn:Lambda}
Define $\Lambda_{-1}(S):=\varnothing$ and $\Lambda_0(S):=\Omega(S)\cup\Psi(S)$.
     For $i \geq 0$, if $T,T'\in \Lambda_i(S)$, define $T\equiv_i T'$ if there are subsets
\[
T=T_0,\, T_1,\, \dots,\, T_n=T' \in \Lambda_i(S),
\]
such that $T_j\cap T_{j+1}$ is nonspherical for each $0\le j<n$. Then $\Lambda_{i+1}(S)$ is defined as the set of all unions of elements of the $\equiv_i$-equivalence classes.  That is, if $\cC_i$ denotes the set of all $\equiv_i$-equivalence classes in $\Lambda_i(S)$, then $\Lambda_{i+1}(S):=\{\bigcup_{T\in C} T\mid C\in\cC_i\}\subseteq\cP(S)$.
\end{definition}

\noindent We will sometimes simplify notation by writing $\Lambda_i$ rather than $\Lambda_i(S)$, when the set $S$ is clear.

We now define hypergraph index for general Coxeter systems.  This generalises~\cite[Definition~3.9]{levcovitz-RACG}.

\begin{definition}\label{defn:hypergraphIndex} 
A Coxeter system $(W,S)$ has \emph{hypergraph index} $h = h(W,S)\in \N\cup\{0\}$ if $S\in\Lambda_h(S)\setminus\Lambda_{h-1}(S)$ and $\Omega(S)\ne\varnothing$.  If such an $h$ does not exist, or if $\Omega(S)=\varnothing$, then $(W,S)$ is said to have \emph{infinite hypergraph index} ($h=h(W,S) = \infty$). 
\end{definition}

{\small
\begin{figure}%[htb!]
\begin{tikzpicture}[double distance=3pt, thick, scale=0.75]

% nonagon
\begin{scope}[xshift=-3cm]
\node[draw=none,minimum size=4cm,regular polygon,regular polygon sides=9] (a) {};
\draw (a.corner 1)--(a.corner 2);
\draw (a.corner 2)--(a.corner 3);
\draw (a.corner 3)--(a.corner 4);
\draw (a.corner 4)--(a.corner 5);
%\draw[line width=3pt] (a.corner 5)--(a.corner 6);
\draw (a.corner 5)--(a.corner 6);
\draw (a.corner 6)--(a.corner 7);
\draw (a.corner 7)--(a.corner 8);
\draw (a.corner 8)--(a.corner 9);
\draw (a.corner 9)--(a.corner 1);

\begin{scope}[semithick]
\draw (a.corner 1)--(a.corner 3);
\draw (a.corner 2)--(a.corner 4);
\draw (a.corner 1)--(a.corner 8);
\draw (a.corner 9)--(a.corner 7);
\end{scope}

\foreach \x in {1,2,...,9}
  \fill (a.corner \x) circle[radius=3pt] node[shift={(\x*360/9+35:0.4)}] {{\small$s_\x$}};
\draw (a.corner 5) node[shift={(7*360/9+60+40:0.77)}] {\footnotesize$5$};
\end{scope} % nonagon

\begin{scope}[xshift=5.5cm, yshift=4cm]
\draw (-0.2,0) node {\underline{Wide subsets $\Omega(S)$}\,:};

% T_1 = S\{s_1,s_5}
\begin{scope}[yshift=-1.5cm] 
\draw (-0.358,0) node {$T_1=S\setminus\{s_1,s_5\}$\,:};
% triangle |>
\begin{scope}[xshift=1.25cm,yshift=0,scale=1.0] 
\fill (1.866,0.5) circle (2.3pt);
\fill (1.866,-0.5) circle (2.3pt);
\fill (2.732,0) circle (2.3pt);
\draw [thick] (2.732,0)--(1.866,-0.5)--(1.866,0.5)--(2.732,0);
\draw (1.866,-0.8) node {\tiny$s_3$}; 
\draw (1.866,0.8) node {\tiny$s_2$}; 
\draw (2.9,-0.35) node {\tiny$s_4$};
\end{scope}
\draw (5,0) node {$\times$};
% kite --<|
\begin{scope}[xshift=6cm,yshift=0cm,scale=1.0] 
\fill (0,0) circle (2.3pt);
\fill (1,0) circle (2.3pt);
\fill (1.866,0.5) circle (2.3pt);
\fill (1.866,-0.5) circle (2.3pt);
\draw [thick] (0,0)--(1,0)--(1.866,0.5);
\draw [thick] (1,0)--(1.866,-0.5)--(1.866,0.5);
\draw (0,-0.35) node {\tiny$s_6$};
\draw (0.85,-0.35) node {\tiny$s_7$};
\draw (1.866,-0.8) node {\tiny$s_8$};
\draw (1.866,0.8) node {\tiny$s_9$};
\end{scope}
\end{scope}

% T_2 = S\{s_1,s_6}
\begin{scope}[yshift=-3.75cm] 
\draw (-0.358,0) node {$T_2=S\setminus\{s_1,s_6\}$\,:};
% triangle <|
\begin{scope}[xshift=6cm,yshift=0,scale=1.0] 
\fill (1,0) circle (2.3pt);
\fill (1.866,0.5) circle (2.3pt);
\fill (1.866,-0.5) circle (2.3pt);
\draw [thick] (1,0)--(1.866,-0.5)--(1.866,0.5)--(1,0);
\draw (1.866,-0.8) node {\tiny$s_8$}; 
\draw (1.866,0.8) node {\tiny$s_9$}; 
\draw (1,-0.35) node {\tiny$s_7$};
\end{scope}
\draw (6,0) node {$\times$};
% kite |>--
\begin{scope}[xshift=1.25cm,yshift=0,scale=1.0] 
\fill (1.866,0.5) circle (2.3pt);
\fill (1.866,-0.5) circle (2.3pt);
\fill (2.732,0) circle (2.3pt);
\fill (3.732,0) circle (2.3pt);
\draw [thick] (2.732,0)--(1.866,-0.5)--(1.866,0.5)--(2.732,0);
\draw [thick] (2.732,0)--(3.732,0);
\draw (1.866,-0.8) node {\tiny$s_3$}; 
\draw (1.866,0.8) node {\tiny$s_2$}; 
\draw (2.9,-0.35) node {\tiny$s_4$};
\draw (3.8,-0.35) node {\tiny$s_5$};
\end{scope}
\end{scope}

% T_3 
\begin{scope}[yshift=-6cm] 
\draw (0,0) node {$T_3=S\setminus\{s_4,s_8,s_9\}$\,:};
% triangle |>
\begin{scope}[xshift=1.25cm,yshift=0,scale=1.0] 
\fill (1.866,0.5) circle (2.3pt);
\fill (1.866,-0.5) circle (2.3pt);
\fill (2.732,0) circle (2.3pt);
\draw [thick] (2.732,0)--(1.866,-0.5)--(1.866,0.5)--(2.732,0);
\draw (1.866,-0.8) node {\tiny$s_2$}; 
\draw (1.866,0.8) node {\tiny$s_1$}; 
\draw (2.9,-0.35) node {\tiny$s_3$};
\end{scope}
\draw (5,0) node {$\times$};
% o--(5)--o----o
\begin{scope}[xshift=6cm,yshift=0cm,scale=1.0] 
\fill (0,0) circle (2.3pt);
\fill (1,0) circle (2.3pt);
\fill (2,0) circle (2.3pt);
%\draw [line width=2.6pt] (0,0)--(1,0);
\draw [thick] (0,0)--(1,0);
\draw [thick] (1,0)--(2,0);
\draw (0,-0.35) node {\tiny$s_5$};
\draw (1,-0.35) node {\tiny$s_6$};
\draw (2,-0.35) node {\tiny$s_7$};
\draw (0.5,0.25) node {\tiny$5$};
\end{scope}
\end{scope}

% T_4 
\begin{scope}[yshift=-8.25cm] 
\draw (0,0) node {$T_4=S\setminus\{s_2,s_3,s_7\}$\,:};
% triangle <|
\begin{scope}[xshift=6cm,yshift=0,scale=1.0] 
\fill (1,0) circle (2.3pt);
\fill (1.866,0.5) circle (2.3pt);
\fill (1.866,-0.5) circle (2.3pt);
\draw [thick] (1,0)--(1.866,-0.5)--(1.866,0.5)--(1,0);
\draw (1.866,-0.8) node {\tiny$s_9$}; 
\draw (1.866,0.8) node {\tiny$s_1$}; 
\draw (1,-0.35) node {\tiny$s_8$};
\end{scope}
\draw (6.1,0) node {$\times$};
% o----o--(5)--o
\begin{scope}[xshift=3.25cm,yshift=0cm,scale=1.0] 
\fill (0,0) circle (2.3pt);
\fill (1,0) circle (2.3pt);
\fill (2,0) circle (2.3pt);
\draw [thick] (0,0)--(1,0);
%\draw [line width=2.6pt] (1,0)--(2,0);
\draw [thick] (1,0)--(2,0);
\draw (0,-0.35) node {\tiny$s_4$};
\draw (1,-0.35) node {\tiny$s_5$};
\draw (2,-0.35) node {\tiny$s_6$};
\draw (1.5,0.25) node {\tiny$5$};
\end{scope}
\end{scope}

\end{scope} % wide subsets

\begin{scope}[xshift=-3cm, yshift=-6cm]
\draw (-1,0) node {\underline{Slab subsets $\Psi(S)$}\,:};
\draw (5,0) node {$T_5=S\setminus\{s_2,s_3,s_8,s_9\}$\,:};
\begin{scope}[xshift=9cm]
\fill (0,0) circle (2.3pt);
\fill (1,0) circle (2.3pt);
\fill (2,0) circle (2.3pt);
\fill (3,0) circle (2.3pt);
\draw [thick] (0,0)--(1,0);
%\draw [line width=2.6pt] (1,0)--(2,0);
\draw [thick] (1,0)--(2,0);
\draw [thick] (2,0)--(3,0);
\draw (0,-0.35) node {\tiny$s_4$};
\draw (1,-0.35) node {\tiny$s_5$};
\draw (2,-0.35) node {\tiny$s_6$};
\draw (3,-0.35) node {\tiny$s_7$};
\draw (1.5,0.25) node {\tiny$5$};
\draw (4,0) node {$\times$};
\fill (5,0) circle (2.3pt);
\draw (5,-0.35) node {\tiny$s_1$};
\end{scope}
\end{scope} % slab subsets

\begin{scope}[xshift=5cm, yshift=-8cm]
\draw (-3.75,0) node {$\Lambda_0=\Omega(S)\cup\Psi(S) = \{\,T_1,\,\,T_2,\,\,T_3,\,\,T_4,\,\,T_5\,\},\quad T_1\equiv_0 T_2$\quad as\quad$T_1\cap T_2\,=\,\,$};

% triangle |>
\begin{scope}[xshift=2.75cm,yshift=0,scale=1.0] 
\fill (1.866,0.5) circle (2.3pt);
\fill (1.866,-0.5) circle (2.3pt);
\fill (2.732,0) circle (2.3pt);
\draw [thick] (2.732,0)--(1.866,-0.5)--(1.866,0.5)--(2.732,0);
\draw (1.866,-0.8) node {\tiny$s_3$}; 
\draw (1.866,0.8) node {\tiny$s_2$}; 
\draw (2.9,-0.35) node {\tiny$s_4$};
\end{scope}
\draw (6.25,0) node {$\times$};
% triangle <|
\begin{scope}[xshift=6cm,yshift=0,scale=1.0] 
\fill (1,0) circle (2.3pt);
\fill (1.866,0.5) circle (2.3pt);
\fill (1.866,-0.5) circle (2.3pt);
\draw [thick] (1,0)--(1.866,-0.5)--(1.866,0.5)--(1,0);
\draw (1.866,-0.8) node {\tiny$s_8$}; 
\draw (1.866,0.8) node {\tiny$s_9$}; 
\draw (1,-0.35) node {\tiny$s_7$};
\end{scope}

\draw (-3.75,-1.75) node {$\Lambda_1 = \{\,T_6 := T_1\cup T_2,\,\,T_3,\,\,T_4,\,\,T_5\,\},\quad T_6\equiv_1 T_5$\quad as\quad $T_6\cap T_5\,=\,\,$};
\begin{scope}[xshift=4cm,yshift=-1.75cm]
\fill (0,0) circle (2.3pt);
\fill (1,0) circle (2.3pt);
\fill (2,0) circle (2.3pt);
\fill (3,0) circle (2.3pt);
\draw [thick] (0,0)--(1,0);
%\draw [line width=2.6pt] (1,0)--(2,0);
\draw [thick] (1,0)--(2,0);
\draw [thick] (2,0)--(3,0);
\draw (0,-0.35) node {\tiny$s_4$};
\draw (1,-0.35) node {\tiny$s_5$};
\draw (2,-0.35) node {\tiny$s_6$};
\draw (3,-0.35) node {\tiny$s_7$};
\draw (1.5,0.25) node {\tiny$5$};
\end{scope}

\draw (-1,-3) node {$\Lambda_2=\{\,S = T_6\cup T_5,\,\,T_3,\,\,T_4\,\}.$};
\end{scope}

\end{tikzpicture}
\caption{\small{The Dynkin diagram for a Coxeter system $(W,S)$ with hypergraph index $2$ is shown in the top left, and the remainder of this figure shows the wide subsets, slab subsets, and the sets $\Lambda_0$, $\Lambda_1$ and $\Lambda_2$, with $S \in \Lambda_2 \setminus \Lambda_1$. As per our convention, a non-edge corresponds to $m_{s_is_j}=2$ and an edge without a label corresponds to $m_{s_is_j}=3$.  
The label $m_{s_5s_6}=5$ is written explicitly.}
\label{fig:hi2_9}}
\end{figure}
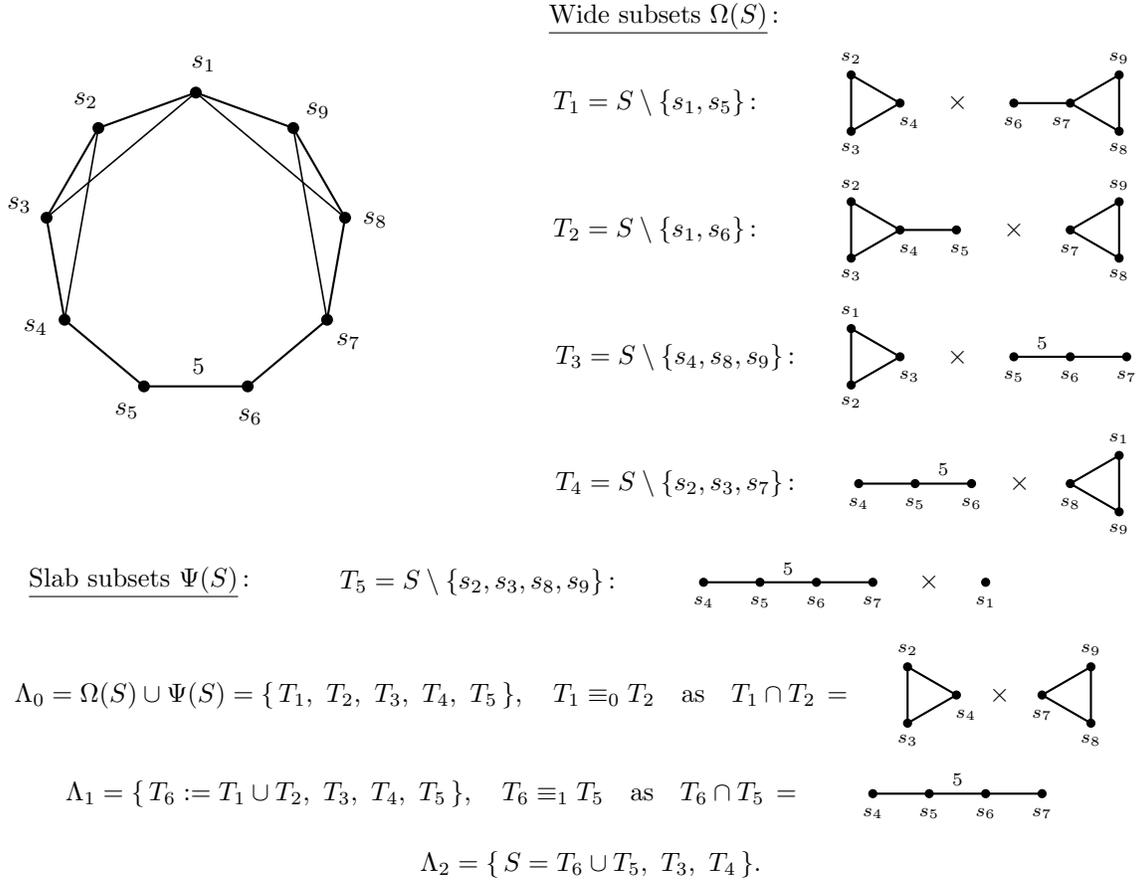
} % \small

\begin{remark}\label{rem:irredHypergraph}
If $(W,S)$ is irreducible affine of rank $\geq 3$, then $\Omega(S) = \{ S \}$ (and $\Psi(S) = \varnothing$), so $\Lambda_i(S) = \{S\}$ for all $i \geq 0$, hence $h(W,S) = 0$.
 However if $(W,S)$ is of type $\widetilde{A}_1$ or is a Lann\'er hyperbolic system, then $h(W,S)=\infty$, since $\Omega(S) = \varnothing$ in these  cases.
\end{remark}

\begin{example}\label{ex:Hypergraph-example}
Figure~\ref{fig:hi2_9} gives an example of a Coxeter system $(W,S)$ with hypergraph index $2$. Here, the set of wide subsets $\Omega(S)$ has four elements: $T_1$ and $T_2$ have the form $A\times B$ with $A$ and $B$ both nonspherical (and $A$ of type $\widetilde A_2$), while $T_3$ and $T_4$ both have the form (irreducible affine)$\times$(spherical), of type $\widetilde A_2\times H_3$. The set of slab subsets $\Psi(S)$ has a single element $A\times K$, where $(W_A,A)$ is a Lann\'er hyperbolic system of type\,\ 
\begin{tikzpicture}[thick,scale=0.75]
\fill (0,0) circle (2.3pt);
\fill (1,0) circle (2.3pt);
\fill (2,0) circle (2.3pt);
\fill (3,0) circle (2.3pt);
\draw (0,0)--(1,0)--(2,0)--(3,0);
\draw (1.5,0.25) node {\tiny$5$};
\end{tikzpicture}
\ and $K =\{s_1\}$.
This example shows that an element of $\Psi(S)$ can participate in forming the hypergraph index in an essential way.
\end{example}

We next establish some technical results concerning the sets  $\Lambda_i(S)$.  
\begin{lemma}\label{lem:widenotinslab}
A slab subset cannot contain a wide subset.
\end{lemma}
\begin{proof}
Let $ A\times K \in \Psi(S)$, where $A$ is minimal nonspherical and $K$ is spherical and 
nonempty. Note that $A \times K$ is not  itself a wide subset, and so in particular $(W_A,A)$ is not irreducible affine of rank $\ge 3$.
Now suppose $B\times C \subseteq A \times K$ for some  $B\times C \in \Omega(S)$, 
where (without loss of generality) $B$ is nonspherical.  Since $A$ 
 is minimal nonspherical,   $A \cap B$ is either spherical or equal to $A$.  If  $A \cap B$ is  spherical then, since $K \cap B$ is spherical, we have that $B  = ( A\cap B) \cup (K \cap B)$ is spherical, a contradiction.  
So $A \cap B = A$, that is, $A \subseteq B$.  If $(W_B,B)$ is irreducible affine of rank $\geq 3$, then $B$ is minimal nonspherical, so $A=B$, 
which contradicts $(W_A,A)$ being not irreducible affine of rank $\geq 3$.  Thus $B$ and $C$ are both nonspherical, and since $A \subseteq B$, it follows that $C \subseteq K$, which contradicts $K$ being spherical. 
\end{proof}

The following lemma restricts the form of the sequences $T_0,T_1,\dots,T_n$ which can appear in Definition~\ref{defn:Lambda}, and is related to~\cite[Remark 3.4]{levcovitz-RACG}.

\begin{lemma}\label{lem:intersection}  Let $(W,S)$ be a Coxeter system.  Let $T$ and $T'$ be elements 
of the form $T = A \times K$ and $T' = A' \times K'$, where $A$ and $A'$ are minimal nonspherical (possibly irreducible affine of rank $\geq 3$, as in Definition~\ref{defn:wide}), and $K$ and $K'$ are spherical.  If $T \cap T'$ is nonspherical then $T = T'$.
\end{lemma}
\begin{proof}  Since $T = A \sqcup K$ and $T' = A' \sqcup K'$, we have
\[
T \cap T' = (A \cap A') \sqcup (A \cap K') \sqcup (K \cap A') \sqcup (K \cap K').
\]
Now $K$ and $K'$ are spherical, so if $T \cap T'$ is nonspherical then $A \cap A'$ must be nonspherical.  But since $A$ and $A'$ are both minimal nonspherical, this implies $A \cap A' = A = A'$.  
Now $A = A'$ commutes with $K \cup K'$.
If $K \cup K'$ is nonspherical, then $A \times K \subsetneq A \times (K \cup K') \subseteq T''$ for some $T'' \in \Omega(S)$.  This violates the maximality in Definition~\ref{defn:wide} if 
 $A\times K$ is wide, and violates condition (3) in Definition~\ref{defn:slab} if it is slab. 
  Thus $K \cup K'$ is spherical, and $A \times (K \cup K')$ satisfies conditions (1)--(3) in Definition~\ref{defn:slab} (since $A \times K$ does). By maximality of $A \times K$, we have 
$A \times K = A\times (K \cup K')$, hence 	
$K' \subseteq K$.  Similarly, $K \subseteq K'$. It follows that $K = K'$ and $T = T'$ as required. 
\end{proof}

We will need the following corollary in Section~\ref{sec:hgiduplex}.  

\begin{cor}\label{cor:non-slab} Assume that $\Omega(S) \neq \varnothing$.
For all $i\ge 0$, let $X \in \Lambda_{i+1}(S)$ and let $\mathcal C  $ be an $\equiv_{i}$-equivalence class  whose union is $X$.  Then $X\in \Lambda_{i+1}(S)\setminus \Psi(S) $ if and only if $\mathcal C$ contains an element of $\Lambda_{i}(S) \setminus \Psi(S)$. 
\end{cor}
\begin{proof}
We first prove that for each $i\ge 0$, if $X \in \Lambda_{i+1}(S)\setminus \Psi(S)$, then $\mathcal C$ contains an element of $\Lambda_{i}(S) \setminus \Psi(S)$.  If not, then $\mathcal C$ consists only of elements from $\Psi(S)$.  Now Lemma~\ref{lem:intersection} implies that $\mathcal C$ has a unique element, which is then necessarily $X$. Thus $X \in \Psi(S)$, a contradiction. 

To establish the other direction, we first observe that an easy inductive proof using the statement proved in the previous paragraph and the fact that $\Omega(S)$ is nonempty shows that for all $j\ge 0$, every element of $\Lambda_{j}(S) \setminus \Psi(S)$ contains a wide subset.  
Thus if $\mathcal C$ contains an element $T$ of $\Lambda_{i}(S) \setminus \Psi(S)$, then $X$ contains a wide subset (since $T$ does) and therefore cannot be in $\Psi(S)$ by Lemma~\ref{lem:widenotinslab}.
\end{proof}

It will sometimes be convenient to think of the elements of $\cup_i\Lambda_i(T)$, where $T \subseteq S$, as the vertices of a directed forest, as follows.  For $i \geq 0$, there is a directed edge from $A \in \Lambda_{i}(T)$ to $B \in \Lambda_{i+1}(T)$ if $A$ is one of the elements in an $\equiv_i$-equivalence class whose union is $B$.
Then for $i < j$ we say that $A \in \Lambda_i(T)$ \emph{feeds into} $B \in \Lambda_j(T)$ if there is a directed path (i.e., a path in which the edges are coherently oriented) from $A$ to $B$.  
We caution that even though $A \subseteq B$ may hold set-theoretically for some 
$A \in \Lambda_i(T)$ and $B \in \Lambda_j(T)$, it may not be true that $A$ feeds into $B$ (see for example~\cite[Remark~3.3.1]{levcovitz-RACG}).

  We now record for later use some monotonicity results.

\begin{lemma}\label{lem:monotone} Let $(W,S)$ be a Coxeter system and let $T\subset T'$ be subsets of $S$.  Let $i \in \N$.
\begin{enumerate}
\item If $T \in \Lambda_i(S)$ then $T$ feeds into some element of $\Lambda_{i+1}(S)$.
\item Given $Y \in \Lambda_i(T)$, there exists an element $Y' \in \Lambda_i (T')$ such that the following conditions hold:
\begin{enumerate}
\item $Y'$ contains $Y$; and 
\item given $X \in \Lambda_j(T)$, with $j<i$, such that $X$ feeds into $Y$, every $X' \in \Lambda_j(T')$ containing $X$ feeds into $Y'$. 
\end{enumerate}
\end{enumerate}
\end{lemma}
\begin{proof}
The proof of (1) follows from the definition of the sets $\Lambda_i(S)$.  

For (2), the proof is by induction on $i$.  In the base case $i=0$, (a) follows easily from the definition of $\Lambda_0$, and (b) is vacuously true.   Assume the conclusion is true for $i-1$ and let $Y \in \Lambda_i(T)$.  Let $ Y_1, \dots, Y_n$ be the elements of an $\equiv_{i-1}$-equivalence class in $\Lambda_{i-1}(T)$ whose union is $Y$ (so in particular, each $Y_j$ feeds into $Y$).  Consider the set  $\mathcal Y \subset \Lambda_{i-1}(T')$ defined~by 
\[
\mathcal Y = \{Z \mid Z \in \Lambda_{i-1}(T'), \;Y_j \subseteq  Z\text{ for some }1 \le j \le n\}.
\]
Note that $\mathcal Y$ is nonempty by the induction hypothesis. 

We claim that the elements of $\mathcal Y$ are pairwise equivalent in $\Lambda_{i-1}(T')$.  To see this, let $Z_1, Z_2 \in \mathcal Y$, and suppose $Y_j \subseteq Z_1$ and $Y_k \subseteq Z_2$.  If $j=k$, then since $Y_j$ is nonspherical, $Z_1$ and $Z_2$ are equivalent.  Otherwise, since $Y_j$ and $Y_k$ are equivalent in $\Lambda_{i-1}(T)$, there is a sequence $Y_j = X_1, \dots, X_l = Y_k$ in $\Lambda_{i-1}(T)$ such that 
$\{X_1, \dots, X_l\} \subset \{Y_1, \dots, Y_n\}$  and 
$X_s \cap X_{s+1}$ is nonspherical for $1 \le s < l$.  By the induction hypothesis, there is a sequence $X_1', \dots, X_l'$ in $\Lambda_{i-1}(T')$ such that $X_s\subseteq X_s'$  for $1 \le s \le l$.  
Now for $1 \le s< l$, the intersection $X'_s \cap X'_{s+1}$ contains $X_s \cap X_{s+1}$, and so 
$X_s'$ is equivalent to $X_{s+1}'$ in $\Lambda_{i-1}(T')$.  Furthermore, $Z_1$ is equivalent to $X_1'$ and 
$Z_2$ is equivalent to $X_l'$.  Thus $Z_1$ and $Z_2$ are equivalent in $\Lambda_{i-1}(T')$.

We have shown that $\mathcal Y$ is a subset of an equivalence class of $\Lambda_{i-1}(T')$.  We now define $Y'$
to be the union of all the sets in this equivalence class.  
Then  $Y' \in \Lambda_i(T')$ and $Y \subseteq Y'$ by construction, i.e.~(a) holds. 
To prove (b), let 
$X\in \Lambda_j(T)$, with $j<i$, be an element which feeds into $Y$, and let $X'$ be any element of $\Lambda_j(T')$ containing $X$.  
 If $j=i-1$, the claim is true by construction.  If $j<i-1$, then $X$ feeds into $Y_k$ for some $1 \le k \le n$. By the induction hypothesis, $X'$ feeds into an element $Y_k'$ containing $Y_k$, and by construction, $Y_k'$ feeds into $Y'$.  Thus $X'$ feeds into $Y'$. 
\end{proof}

Next, we make some first observations relating to hypergraph index itself (Definition~\ref{defn:hypergraphIndex}).  The following three results will allow us to focus on infinite irreducible Coxeter systems when considering hypergraph index below. 

\begin{lemma}\label{lem:finite} Let $(W,S)$ be a Coxeter system.  If $W$ is finite then $h(W,S) = \infty$.
\end{lemma}

\begin{proof}
In this case $\Omega(S) = \varnothing$ and so $h(W,S) = \infty$.
\end{proof}

\begin{lemma}\label{lem:reducible1} Let $(W,S)$ be a Coxeter system.  If $(W,S) = (W_1,S_1) \times (W_2,S_2)$ with $W_1$ and $W_2$ both infinite, then $h(W,S) = 0$.
\end{lemma}

\begin{proof}
The set $S = S_1 \times S_2$ is in $\Omega(S)$, so $h(W,S) = 0$. 
\end{proof}

\begin{prop}\label{prop:reducible2}  Let $(W,S)$ be a Coxeter system.  If $(W,S) = (W_1,S_1) \times (W_2,S_2)$ with $(W_1,S_1)$ irreducible and infinite and $W_2$ finite, then $h(W,S) = h(W_1,S_1)$.
\end{prop}

\begin{remark} If Conjecture~\ref{conj:lower} of the introduction holds, then hypergraph index is a quasi-isometry invariant, and Proposition~\ref{prop:reducible2} is immediate.
\end{remark}

\begin{remark}\label{rem:slab1}
Our current proof of Proposition~\ref{prop:reducible2} is longer and more technical than might be expected  because, although every element $A \times K$ of $\Psi(S_1)$ yields an element $A \times (K \cup S_2)$ of $\Psi(S)$, the following example shows that $\Psi(S)$ may also contain elements which are not of this form.  Let $S_1 = \{ a_1, b_1, c_1, d_1\}$ and let $(W_1,S_1)$ have Dynkin diagram as shown in Figure~\ref{fig:slab}.  Then it is not hard to check that $\Psi(S_1) = \varnothing$.  However for any nontrivial spherical Coxeter system $(W_2,S_2)$, if $(W,S) = (W_1,S_1) \times (W_2,S_2)$ then $\{b_1, c_1, d_1\} \times S_2$ is in $\Psi(S)$. 

The existence of such elements of $\Psi(S)$ complicates the proof because such a set could potentially act as a bridge between subsets of  $\Lambda_i(S)$ that ``come from'' distinct equivalence classes in $\Lambda_i(S_1)$, and so create an equivalence class in $\Lambda_i(S)$ which is not induced just by the relation $\equiv_i$ in $\Lambda_i(S_1)$.  This might then feed into an element of $\Lambda_{i+1}(S)$ which acts as a bridge between subsets which ``come from" distinct equivalence classes in $\Lambda_{i+1}(S_1)$, and so on.  Ultimately, this could lead to 
$S$ appearing sooner in the directed forest for $(W,S)$, yielding a lower hypergraph index. For example, in the Coxeter system given in Example~\ref{ex:Hypergraph-example} and Figure~\ref{fig:hi2_9}, the element $T_5\in\Psi(S)$ participates essentially in forming an equivalence class that turns out to be the whole of $S$, thus reducing hypergraph index from $\infty$ to $2$. 
Below we give a proof that, in the situation of $(W,S) = (W_1,S_1) \times (W_2,S_2)$, such behaviour does not in fact occur.
\begin{figure}
    \centering
\begin{tikzpicture}[thick]
	\fill (-1.73,0) circle (2.3pt) node [inner sep=5pt, anchor=east] {$a_1$};
	\fill (1.73,0) circle (2.3pt) node [inner sep=5pt, anchor=west] {$c_1$};
	\fill (0,1) circle (2.3pt) node [inner sep=5pt, anchor=south] {$b_1$};
	\fill (0,-1) circle (2.3pt) node [inner sep=5pt, anchor=north] {$d_1$};
	\draw (-1.73,0)--(0,1)--(1.73,0)--(0,-1)--(-1.73,0);
	\draw (0,1)--(0,-1);
	\draw (1.0,0.7) node {$4$};
	\draw (1.0,-0.7) node {$4$};
	\end{tikzpicture}
  \caption{The Coxeter system $S_1$ from Remark~\ref{rem:slab1}. Edges without labels correspond to $m_{st}=3$.\label{fig:slab}}
\end{figure}
\end{remark}

\begin{proof}[Proof of Proposition~\ref{prop:reducible2}]
Consider the function $\theta\colon \cP(S_1) \to \cP(S)$ defined by $\theta(Y) = Y \cup S_2$.  Since images and preimages of nonspherical sets under $\theta$ are nonspherical, 
it is easy to see from the definitions that  $\theta$ induces a map $\Lambda_0(S_1) \to \Lambda_0 (S)$ which takes 
$\Omega(S_1)$ bijectively to $\Omega(S)$ and  takes 
 $\Psi(S_1)$ into $\Psi(S)$.   Let $\Sigma$ denote the (possibly empty) set of 
 elements of $\Psi(S)$ which are not in the image of $\theta$.  Then every element of $\Sigma$ is of the form 
 $A \times S_2$ where $A$ is minimal nonspherical in $S_1$.  

 We now state a lemma, show how it easily implies Proposition~\ref{prop:reducible2}, and finally prove this lemma.

 \begin{lemma}\label{lem:claim_reducible2}
 Let $(W,S)= (W_1,S_1) \times (W_2,S_2)$ be as in the statement of Proposition~\ref{prop:reducible2}, and let $\theta:Y \mapsto Y \cup S_2$ and $\Sigma = \Psi(S) \setminus \theta(\Psi(S_1))$ be as defined above.  Then for all $i \geq 0$:
 \begin{enumerate}
     \item the map $\theta$ induces a map $\Lambda_{i}(S_1) \to \Lambda_i (S)$; and 
     \item every element of $\Lambda_i(S) \setminus \Sigma$ is contained in the image of $\theta$.
 \end{enumerate}
 \end{lemma}

 To complete the proof of Proposition~\ref{prop:reducible2} using Lemma~\ref{lem:claim_reducible2}, we first note that $\Omega(S_1)$ is empty if and only if $\Omega(S)$ is empty, and that in this case $h(W_1, S_1)= h(W, S)=\infty$.  So we may assume $\Omega(S)$ is nonempty. 
 Now if $S_1 \in \Lambda_i(S_1)$ for some $i$, then $S= \theta(S_1)\in  \Lambda_i(S)$ by Lemma~\ref{lem:claim_reducible2}(1).  On the other hand, if $S \in \Lambda_i(S)$, then since $\Omega(S)$ is nonempty, Lemma~\ref{lem:widenotinslab} implies that $S$ is not in $\Psi(S)$, and so not in~$\Sigma$.  Then by Lemma~\ref{lem:claim_reducible2}(2), we have that $S$ is in the image of $\theta\colon \Lambda_i(S_1) \to  \Lambda_i(S)$, and it follows that $\theta^{-1}(S) = S_1 \in \Lambda_i(S_1)$.  From this we conclude that $h(W_1, S_1) = h(W, S)$. 
 \end{proof}
 
\begin{proof}[Proof of Lemma~\ref{lem:claim_reducible2}]
  We prove this by induction on $i \geq 0$.  Observe that the first paragraph of the proof of Proposition~\ref{prop:reducible2} provides the base case for both (1) and (2). 
  For the inductive step, assume that both (1) and (2) hold for $k$, and let $X\in \Lambda_{k+1}(S_1)$.    First, we show that $\theta(X) = X \cup S_2$ is an element of $\Lambda_{k+1}(S)$, thus establishing (1). 
 
 Let $T\in \Lambda_k(S_1)$ be an element feeding into $X$, then, by the induction hypothesis, $\theta(T)$ belongs to $\Lambda_k(S)$.  Let  $\mathcal{C}_1$ and $\mathcal C$ be the $\equiv_k$-equivalence classes containing $T$ and $\theta(T)$ in $\Lambda_k(S_1)$ and $\Lambda_k(S)$, respectively, and let $\theta(\mathcal{C}_1)= \{\theta(Y) \mid Y \in \mathcal{C}_1\}$.

Observe that for any $A, B \in \Lambda_k(S_1)$, we have  that  $A \cap B$ is nonspherical if and only if $\theta(A)\cap \theta(B)$ is.   It easily follows that if $T_1 \in \mathcal{C}_1$, then $\theta(T) \equiv_k \theta(T_1)$ in 
 $\Lambda_k(S)$, and so $\theta(\mathcal{C}_1)\subset \mathcal C$.   
 
 \emph{Claim. Every element of $\mathcal C$ is a subset of an element of $\theta(\mathcal{C}_1)$.}
  
To prove the Claim, let 
$T'' \in \mathcal C$. Then 
$\theta(T) \equiv_k T''$ in $\Lambda_k(S)$, and  there is a 
  sequence $\theta(T) = T_1'$, $T_2'$, \dots, $T_m' =T''$ of sets in $
 \Lambda_k(S)$, with $T'_i \cap T'_{i+1}$ nonspherical for $1 \le i <m$.  Let $j$ be the smallest index (if any) such that $T_j'$ is not in the image of $\theta$, and note that $j>1$. Then by the induction hypothesis for (2),  $T_j' \in \Sigma$ and so $T_j' = A\times S_2$ for some minimal nonspherical $ A$.    Now since $T_{j-1}' \cap (A \times S_2)$ is  nonspherical, it follows that this intersection contains $A$, and since $T_{j-1}'$ is in the image of $\theta$, this intersection also contains $S_2$.  Thus $T_j'$ is a subset of $T_{j-1}'$.  Consequently, if $j<m$, then 
  $T_{j-1}' \cap T_{j+1}'$ is nonspherical,  and so $T_j'$ can be removed from the sequence.  
  Thus we may assume that $T_i'$ is in the image of $\theta$  for all $i<m$.  
  
  Now if $T'' = T_m'$ is also in the image, i.e.\ 
  $T_m' = \theta(T_m)$ for some $T_m \in \Lambda_k(S)$, then the preimage of the sequence $\theta(T)=T_1'$, $T_2'$, \dots, $T_m' =T''$
  establishes that 
$T \equiv_k T_m$, i.e. $T_m \in \mathcal{C}_1$.  On the other hand, if $T'' \in \Sigma$, we know (from the previous paragraph) that it is already contained in an  element of  $\theta(\mathcal{C}_1)$, namely $T_{m-1}'$.   
This completes the proof of the Claim.

From the Claim, we conclude that if $X' \in \Lambda_{k+1}(S)$ denotes the union of the sets in $\mathcal C$, then 
\[  
X' =\bigcup_{Y' \in \mathcal C} Y'  = \bigcup_{Y' \in\theta( \mathcal{C}_1)} Y' =\bigcup_{Y \in \mathcal{C}_1} \theta (Y)= \bigcup_{Y \in \mathcal{C}_1}  Y\cup S_2
= \theta (X) .
\] This completes the inductive proof of (1).
 
To finish the proof of (2), we need to show that every  $X' \in \Lambda_{k+1}(S) \setminus \Sigma$ is in the image of $\theta$.   If $\mathcal C$ denotes an equivalence class of sets in $\Lambda_k(S)$ feeding into $X'$, then $\mathcal C$ contains an element $T'$ of $\Lambda_k(S) \setminus \Sigma$, for otherwise, Lemma~\ref{lem:intersection} would imply that $\mathcal C$ consists of a single element of $\Sigma$, forcing $X'$ to be in $\Sigma$.  By the induction hypothesis, $T' = \theta(T)$ for some $T \in \Lambda_k(S_1)$, and if $T$ feeds into $X \in \Lambda_{k+1}(S)$, it follows from the above discussion that 
 $\theta(X) = X'$.   This completes the inductive proof of (2). 
 \end{proof}

We next prove that for the computation of hypergraph index, finite edge labels in the set $\{7,8,\dots\}$ are indistinguishable.  

\begin{prop}\label{prop:labels7up}
Let $I = \{1,\dots,n\}$ and let $(W,S)$ and $(W',S')$ be Coxeter systems with $S = \{s_1,\dots,s_n\}$ and $S' = \{ s_1', \dots, s_n'\}$, so that 
\[
W = \langle s_i  \mid (s_i s_j)^{m_{ij}} = 1 \rangle \quad \mbox{and} \quad W' = \langle s'_i  \mid (s'_i s'_j)^{m'_{ij}} = 1 \rangle, \quad i,j \in I.
\]  
Suppose that the matrices $(m_{ij})$ and $(m'_{ij})$ differ only in one pair of positions corresponding to indices $(p,q)$ and $(q,p)$, with both $m_{pq}$ and $m'_{pq}$ finite numbers $\ge7$.
Let $\iota\colon S\to S'$, $s_i \mapsto s_i'$ be the natural bijection.
Then $\iota$ induces bijections $\Omega(S) \to \Omega(S')$, $\Psi(S) \to \Psi(S')$, and $\Lambda_k(S) \to \Lambda_k(S')$ for all $k \geq 0$. In particular, the hypergraph indexes of $(W,S)$ and $(W',S')$ are equal.
\end{prop}
\begin{proof}
We notice that in the construction of the sets  $\Omega(S)$, $\Psi(S)$ and $\Lambda_k(S)$ from Definitions~\ref{defn:wide}, \ref{defn:slab} and \ref{defn:Lambda}, respectively, we consider subsystems of the following classes: 
\begin{enumerate}
\item spherical; 
\item nonspherical;
\item irreducible affine of rank $\ge3$;
\item minimal nonspherical:
\begin{enumerate}
\item irreducible affine of rank $\ge 2$; and
\item Lann\'er hyperbolic (see Figure~\ref{fig:lanner}).
\end{enumerate}
\end{enumerate}

By analysing the Dynkin diagrams for irreducible affine systems and for Lann\'er hyperbolic ones, we observe that none of them has finite edge labels $\ge 7$.  Also, among Dynkin diagrams for spherical Coxeter systems, finite edge labels $\ge 7$ appear only in the rank $2$ systems corresponding to finite dihedral groups.
It follows that for an arbitrary subset $T \subseteq S$, $T$ is of any type listed in (1)--(4) above if and only if $\iota(T)$ is of that same type. Thus $\iota$ induces a bijection between the sets $\Omega(S)$ and $\Omega(S')$ and also between the sets $\Psi(S)$ and $\Psi(S')$. 

Now observe that, by similar arguments, the sets $T_j$ appearing in Definition~\ref{defn:Lambda} have the property: $T_j\cap T_{j+1}$ is nonspherical if and only if $\iota(T_j)\cap\iota(T_{j+1})$ is nonspherical. Hence by induction, $\iota$ induces a bijection between $\Lambda_k(S)$ and $\Lambda_k(S')$ for all $k \geq 0$.
\end{proof}

\begin{remark}
Proposition~\ref{prop:labels7up} tells us that for all purposes related to hypergraph index, we can work with a set of just seven edge labels, namely  $m_{ij}\in \{2,3,4,5,6,7,\infty\}$. This makes the problem of exploration of possible hypergraph indexes (and possible configurations of the sets $\Lambda_i(S)$) on a fixed generating set $S$ a finite computational problem.
\end{remark}

We conclude this section by considering the relationship with tight (algebraic) networks (see Definitions~\ref{defn:tight} and~\ref{defn:tightAlgebraic}).  The next two results do not directly generalise any results from~\cite{levcovitz-RACG}, although they follow easily from the definitions.  

\begin{lemma}\label{lem:tightAlgebraic}  Let $(W,S)$ be a Coxeter system.  For $i \geq 0$, let $U$ be an element of $\Lambda_{i+1}(S)$, with~$U$ the union of the elements $T \in C_U$ for $C_U$ an $\equiv_{i}$-equivalence class in $\Lambda_{i}(S)$.  Then $W_U$ is a tight algebraic network with respect to the collection of special subgroups \[ \cH_U = \{ W_T \mid T \in C_U \}. \]
\end{lemma}

\begin{proof}  Special subgroups of Coxeter groups are convex, so as each $T \in C_U$ is a subset of $U$, each $W_T \in \cH_U$ is convex in $W_U$.  As $U$ is the union of the elements $T \in C_U$, the union of the subgroups in $\cH_U$ generates $W_U$ itself.  If $T, T' \in C_U$ then by definition of $\equiv_i$-equivalence, there are subsets $T=T_0, T_1, \dots, T_n=T'$ in $C_U$ such that $W_{T_j} \cap W_{T_{j+1}} = W_{T_j \cap T_{j+1}}$ is infinite for $0 \leq j < n$.  Since each $W_{T_j \cap T_{j+1}}$ is convex in $W_U$, it is path-connected.  
\end{proof}

\begin{cor}\label{cor:tight}  With all notation as in the statement of Lemma~\ref{lem:tightAlgebraic}, $W_U$ is a tight network with respect to the collection of left cosets
\[
\cL_U = \{ w W_T \mid w \in W_U, T \in C_U \}.
\]
\end{cor}

\begin{proof}  This is immediate from Lemma~\ref{lem:tightAlgebraic} and Proposition~\ref{prop:tight}.
\end{proof} 

\subsection{Hypergraph index, thickness and divergence}\label{sec:hyp_thick_div}

In this section we establish our main results relating  hypergraph index to thickness and divergence,  which are stated as Theorem~\ref{thm:hyp_thick_div_Intro} of the introduction. 

As mentioned in the introduction, we will use strong algebraic thickness (see Definition~\ref{defn:strongAlgThick}). We will need the following result. 

\begin{prop}\label{prop:thickFinite}  Let $(W,S)$ be a Coxeter system.  If $W$ is strongly algebraically thick then $(W,S)$ has finite hypergraph index.
\end{prop}

\begin{proof}  By Theorem~\ref{thm:thickNotRelHyp}, we have $(W,S) \in \bT$. Thus by the inductive definition of $\bT$ (see Definition~\ref{defn:T}) it suffices to prove the following three claims:
\begin{enumerate}
\item If $T \subseteq S$ and $(W_T,T) \in \bT_0$ then $(W_T,T)$ has finite hypergraph index.
\item If $S_0 \subseteq S$ and $s \in S \setminus S_0$ are such that $(W_{S_0},S_0) \in \bT$, $s^\perp \cap S_0$ is nonspherical and $(W_{S_0},S_0)$ has finite hypergraph index $h$, then $(W_{S_0 \cup \{s\}}, S_0 \cup \{s \})$ has finite hypergraph index.
\item If $S_1, S_2 \subseteq S$ are such that $(W_{S_i},S_i) \in \bT$ for $i = 1,2$ and $S_1 \cap S_2$ is nonspherical, and $(W_{S_i},S_i)$ has finite hypergraph index $h_i$ for $i = 1,2$, then $(W_{S_1 \cup S_2}, S_1 \cup S_2)$ has finite hypergraph index.
\end{enumerate}

For (1), by the definitions we have $T  \in \Omega(T)$, and so $T  \in \Lambda_0(T) = \Lambda_0(T) \setminus \Lambda_{-1}(T)$, and $(W_T,T)$ has hypergraph index $0$.

For (2), observe first that $\Omega(S_0) \neq \varnothing$ by the definition of finite hypergraph index, hence $\Omega(S_0 \cup \{s\}) \neq \varnothing$.  
Furthermore, $S_0 \in \Lambda_h(S_0)$, so by Lemma~\ref{lem:monotone}(2a), there exists $T \in \Lambda_h(S_0 \cup\{s\})$ such that $S_0 \subseteq T$.  
Let $A$ be a minimal nonspherical subset of $s^\perp \cap S_0$.  Then $A \times \{s\}$ is contained in some wide or slab subset of 
$S_0 \cup \{s \}$, which feeds into some set $T' \in \Lambda_h(S_0 \cup\{s\})$, by Lemma~\ref{lem:monotone}(1). 
Now $T \cap T'$ contains $A$, so is nonspherical.  Thus $T \cup T'$ is a subset of an element of $\Lambda_{h+1}(S_0 \cup \{s\})$. 
Since $S_0 \cup \{s\} \subseteq T \cup T'$, we conclude that $S_0 \cup \{s\}$ has hypergraph index $\le h+1$. 

For (3), we have that $\Omega(S_i) \neq \varnothing$ for $i = 1,2$, hence $\Omega(S_1 \cup S_2) \neq \varnothing$.  Now let $h = \max\{ h_1, h_2 \}$.  Then since by assumption $S_i \in \Lambda_{h_i}(S_i)$, by Lemma~\ref{lem:monotone}(2a) $S_i$ is contained in an element of $\Lambda_h(S_1 \cup S_2)$, say $T_i$.   Now $T_1 \cap T_2$ is nonspherical, since it contains $S_1 \cap S_2$.  So $T_1 \cup T_2$ is a subset of an element of $\Lambda_{h+1}(S_1 \cup S_2)$.  Then since $(S_1 \cup S_2) \subseteq (T_1 \cup T_2)$, we get $S_1 \cup S_2 \in \Lambda_{h+1}(S_1 \cup S_2)$.  This completes the proof. 
\end{proof}

We can now prove one implication of Theorem~\ref{thm:hyp_thick_div_Intro}(4), without requiring $W$ to be $1$-ended (since we have not yet invoked the divergence of $W$).

\begin{cor}  Let $(W,S)$ be a Coxeter system.  If $(W,S)$ has infinite hypergraph index, then $W$ is relatively hyperbolic.
\end{cor}

\begin{proof}
This is immediate from the contrapositive to Proposition~\ref{prop:thickFinite} together with Theorem~\ref{thm:thickNotRelHyp}.
\end{proof}

We next show that hypergraph index $0$ characterises linear divergence, proving Theorem~\ref{thm:hyp_thick_div_Intro}(1).

\begin{thm}\label{thm:h0Linear} Let $(W,S)$ be a Coxeter system such that $W$ is $1$-ended.  Then $(W,S)$ has hypergraph index $0$ if and only if $W$ has linear divergence.
\end{thm}

\begin{proof}  Suppose first that $W$ has linear divergence.  Then by Lemma~\ref{lem:linearOmega}(2), $\Omega(S) = \{ S \}$, hence $S \in \Lambda_0 = \Lambda_0 \setminus \Lambda_{-1}$.  Thus $(W,S)$ has hypergraph index $0$.

Now suppose that $(W,S)$ has hypergraph index $0$.  Then $S$ contains a wide subset, since $\Omega(S) \neq \varnothing$, and $S  \in  \Lambda_0 = \Omega(S) \cup \Psi(S)$.  Lemma~\ref{lem:widenotinslab} implies that $S \notin \Psi(S)$, and so we have $S \in \Omega(S)$.  Hence $\Omega(S) = \{ S \}$, and so by Lemma~\ref{lem:linearOmega}(2) again, $W$ has linear divergence.
\end{proof}

The next two results prove Theorem~\ref{thm:hyp_thick_div_Intro}(3).

\begin{thm}\label{thm:finiteThick} Let $(W,S)$ be a Coxeter system such that $W$ is $1$-ended. Suppose $(W,S)$ has finite hypergraph index $h$.  Then $W$ is strongly thick of order at most $h$.
\end{thm}

\begin{proof}  The proof is by induction on $h$.  If $h = 0$ then by Theorem~\ref{thm:h0Linear}, $W$ has linear divergence.  Thus $W$ is strongly thick of order $0$, by Proposition~\ref{prop:wide}.

For the inductive step, since $S \in \Lambda_h$, by definition there is an $\equiv_{h-1}$-equivalence class $C_S$ in $\Lambda_{h-1}$ such that $S$ is the union of the elements $T$ of $C_S$.  Hence by Corollary~\ref{cor:tight}, $W = W_S$ is a tight network with respect to the collection of left cosets $\cL_S = \{ w W_T \mid w \in W, T \in C_S \}$.  Now for each $T \in \Lambda_{h-1}$, the Coxeter system $(W_T,T)$ has hypergraph index at most $h-1$.  Hence by inductive assumption, for each $T \in \Lambda_{h-1}$, the group $W_T$  is strongly thick of order at most $h-1$.  It follows that each coset $w W_T \in \cL_S$ is strongly thick of order at most $h-1$.  Therefore $W$ is strongly thick of order at most $h$, as required.
\end{proof}

\begin{cor}\label{cor:hDiv} Let $(W,S)$ be a Coxeter system such that $W$ is $1$-ended.  If $(W,S)$ has finite hypergraph index $h$, then the divergence of $W$ is bounded above by a polynomial of degree $h + 1$.
\end{cor}

\begin{proof}  This follows from Theorem~\ref{cor:thickDiv} and Theorem~\ref{thm:finiteThick}.
\end{proof}

We now prove the other direction of Theorem~\ref{thm:hyp_thick_div_Intro}(4).

\begin{cor}
\label{cor:hInfiniteHyp} Let $(W,S)$ be a Coxeter system such that $W$ is $1$-ended. If $W$ is relatively hyperbolic, then $(W,S)$ has infinite hypergraph index.
\end{cor}

\begin{proof}  Suppose $(W,S)$ has finite hypergraph index~$h$.  Then by Corollary~\ref{cor:hDiv}, the divergence of $W$ is bounded above by a polynomial of degree at most $h+1$. However, this is a contradiction, as finitely presented relatively hyperbolic groups have exponential divergence~\cite[Theorem~1.3]{Sisto12}.   So $(W,S)$ has infinite hypergraph index.
\end{proof}

Finally, we prove Theorem~\ref{thm:hyp_thick_div_Intro}(2).

\begin{cor}\label{cor:h1Quadratic} Let $(W,S)$ be a Coxeter system such that $W$ is $1$-ended.  Suppose  $(W,S)$ has hypergraph index $h = 1$.  Then $W$ has quadratic divergence.
\end{cor}

\begin{proof}  By Theorem~\ref{thm:h0Linear}, $W$ does not have linear divergence.  Hence if $(W,S)$ is reducible, by Corollary~\ref{cor:linearIntro} $(W,S)$ is the direct product of an irreducible nonspherical nonaffine Coxeter system with some finite (possibly reducible) Coxeter system.   As divergence is a quasi-isometry invariant, we may thus assume without loss of generality that $(W,S)$ is irreducible and nonaffine.  Now by Theorem~\ref{thm:linearQuadraticIntro}, the divergence of $W$ is at least quadratic, while by Corollary~\ref{cor:hDiv}, the divergence of $W$ is at most quadratic.  Thus $W$ has quadratic divergence as required.
\end{proof}

\begin{remark}\label{rem:strongAlg1}
In the right-angled case, by~\cite[Theorem~6.2]{levcovitz-RACG} the Coxeter system $(W,S)$ having hypergraph index~$h=1$ is equivalent to $W$ being strongly algebraically thick of order~$1$; the proof relies on the characterisation of quadratic divergence for right-angled Coxeter groups.  Using this equivalence as the base case of an induction, Theorem~6.5 of~\cite{levcovitz-RACG} then shows that in right-angled systems, $(W,S)$ having finite hypergraph index $h > 0$ implies that $W$ is strongly algebraically thick of order at most $2h-1$.  The inductive step in the proof of~\cite[Theorem~6.5]{levcovitz-RACG} goes through in the general setting, using Lemma~\ref{lem:intersection}.  However, for general $(W,S)$, we do not know if $h(W,S)=1$ implies that $W$ is strongly algebraically thick of order $1$, so we are missing the base case for a similar inductive proof of the corresponding statement in the general setting.
\end{remark}

\section{The duplex construction}\label{sec:duplex}

We now introduce the procedure we call the \emph{duplex construction} which, given any right-angled Coxeter system, produces infinitely many non-right-angled Coxeter systems of the same hypergraph index.
In Section~\ref{sec:hgiduplex} we define the duplex construction and prove that it preserves hypergraph index.  We briefly discuss the relationships between the original right-angled Coxeter group and its various duplexes in Section~\ref{sec:relationships}.  Then in Section~\ref{sec:DTexamples}, we apply the duplex construction to a family of examples from~\cite{dani-thomas-div} to obtain, for every integer $d \geq 1$, infinitely many pairwise nonisomorphic Coxeter groups of hypergraph index $d-1$.  
In Section~\ref{sec:examplesDiv}, we
show that certain duplexes with hypergraph index $d-1$ have divergence~$\simeq r^d$ and are strongly thick of order $d-1$; thus Conjecture~\ref{conj:lower} holds for these examples.  Since we are working with right-angled Coxeter systems, we will mostly use defining graphs rather than Dynkin diagrams in this section.

\subsection{Construction of the duplex}\label{sec:hgiduplex}

The duplex construction is the construction given by Definition~\ref{defn:duplex} below, and the main result of this section is stated as Theorem~\ref{thm:duplexhi} immediately afterwards. 

\begin{definition}\label{defn:duplex}
Let $(W,S)$ be a right-angled Coxeter system.  Let $m$ be an integer $\geq 2$ and let $n$ be either an integer $\geq 5$ or equal to $\infty$.  The \emph{$(m,n)$-duplex of $(W,S)$}, denoted $(W_2(m,n),S_2)$, is the Coxeter system constructed as follows.  The generating set $S_2$ is the disjoint union of two copies of $S$, with each $u\in S$ giving rise to the two elements $u',u''\in S_2$.  The edges of the defining graph of $(W_2(m,n),S_2)$ are then labelled as follows:
\begin{enumerate}
\item for all $u \in S$, the label on the edge $[u',u'']$ is $m$; 
\item if $m_{uv} = 2$ in $(W,S)$, then the four edges $[u',v']$, $[u',v'']$, $[u'',v']$ and $[u'',v'']$ all have label $2$; and
\item if $m_{uv}=\infty$ in $(W,S)$, then the four edges $[u',v']$, $[u',v'']$, $[u'',v']$ and $[u'',v'']$ all have label $n$.
\end{enumerate}
\end{definition}

\begin{thm}\label{thm:duplexhi}  Let $(W,S)$ be a right-angled Coxeter system.  For all $m \geq 2$ and $n \geq 5$ (including $n = \infty$), the hypergraph index of $(W_2(m,n),S_2)$ is the same as the hypergraph index of $(W,S)$.
\end{thm}

Theorem~\ref{thm:duplexhi} will be proved in the end of this subsection.

\begin{remark} \label{rem:duplex} The intuition for Definition~\ref{defn:duplex} and the proof of Theorem~\ref{thm:duplexhi} is as follows.  The requirement $m<\infty$ is necessary so that $\{ u',u''\}$ is a spherical subset (since it corresponds to a single generator of a right-angled Coxeter group).  The requirement $n\ge5$ is necessary so that 
no nonspherical subsystem of rank $3$ in $S_2$ is affine (see Lemma~\ref{lem:minDuplex}). 
We also note that it is possible to generalise the construction of a duplex by allowing edge labels (in case $m_{uv}=\infty$) be varying labels $\ge5$ (or $\infty$) instead of all the same label $n\ge5$.
\end{remark}

The next example will be followed throughout this section.

\begin{example} Figure~\ref{fig:g3} illustrates the construction of a duplex. On the left is the defining graph $\G_3$ from~\cite[Section~5]{dani-thomas-div}.  In this graph, all edges are labelled with $2$ and all ``missing" edges are labelled with $\infty$.  To keep notation simple, let $S$ be the vertex set of $\G_3$ and let $W = W_{\G_3}$ be the associated right-angled Coxeter group. On the right we give a schematic for the defining graph, denoted $(\G_3)_2$, of the duplex $(W_2(m,n),S_2)$ of $(W,S)$.  Each vertex here represents two generators $u',u'' \in S_2$ such that $m_{u',u''}=m$, each thick black edge represents four edges all labelled by $2$ and each thick light grey edge represents four edges all labelled by $n$.  The Coxeter systems $(W,S)$ and $(W_2(m,n),S_2)$ both have hypergraph index 2, as we will prove.

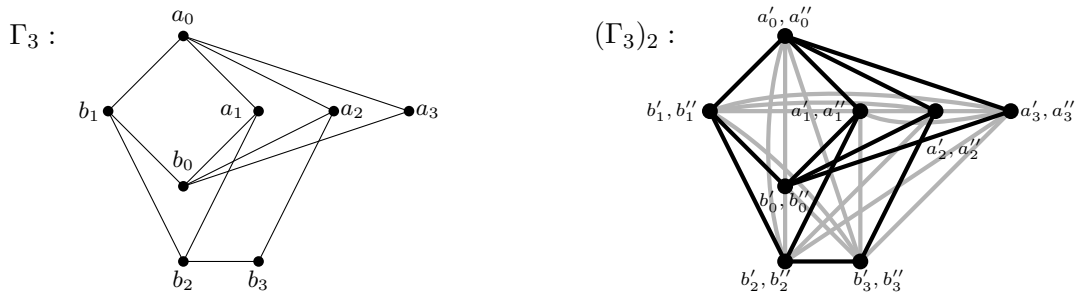
\begin{figure}[htb!] 
\begin{tikzpicture}[scale=1]
\fill (0,1) circle (2pt); \draw (0,1.25) node {\footnotesize $a_0$};
\fill (0,-1) circle (2pt); \draw (0,-0.65) node {\footnotesize $b_0$};
\fill (0,-2) circle (2pt); \draw (0,-2.25) node {\footnotesize $b_2$};
\fill (1,-2) circle (2pt); \draw (1,-2.25) node {\footnotesize $b_3$};
\fill (-1,0) circle (2pt); \draw (-1.25,0) node {\footnotesize $b_1$};
\fill (1,0) circle (2pt); \draw (0.65,0) node {\footnotesize $a_1$};
\fill (2,0) circle (2pt); \draw (2.25,0) node {\footnotesize $a_2$};
\fill (3,0) circle (2pt); \draw (3.25,0) node {\footnotesize $a_3$};
\draw (-1,0)--(0,1);
\draw (-1,0)--(0,-1);
\draw (-1,0)--(0,-2);
\draw (1,0)--(0,1);
\draw (1,0)--(0,-1);
\draw (1,0)--(0,-2);
\draw (2,0)--(0,1);
\draw (2,0)--(0,-1);
\draw (3,0)--(0,1);
\draw (3,0)--(0,-1);
\draw (1,-2)--(0,-2);
\draw (1,-2)--(2,0);
\draw (-2,1) node {$\Gamma_3:$};
\begin{scope}[xshift=8cm,ultra thick]
\begin{scope}[color=mygrey]
\draw (0,1)--(0,-1)--(0,-2);
\draw (-1,0)--(1,0)--(2,0)--(3,0);
\draw (0,1)--(1,-2);
\draw (0,-2)--(2,0);
\draw (0,-2)--(3,0);
\draw (1,-2)--(3,0);
\draw (1,-2)--(1,0);
\draw (0,-1)--(1,-2);
\draw (-1,0) to [out=-30, in=120,looseness=0.75] (1,-2);
\draw (0,1) to [out=-110, in=110,looseness=0.75] (0,-2);
\draw (-1,0) to [out=10, in=170,looseness=0.75] (2,0);
\draw (-1,0) to [out=20, in=170,looseness=0.75] (3,0);
\draw (1,0) to [out=-30, in=-170,looseness=0.75] (3,0);
\end{scope}
\fill (0,1) circle (3pt); \draw (0,1.25) node {\tiny $a'_0,a''_0$};
\fill (0,-1) circle (3pt); \draw (0,-1.2) node {\tiny $b'_0,b''_0$};
\fill (0,-2) circle (3pt); \draw (-0.25,-2.25) node {\tiny $b'_2,b''_2$};
\fill (1,-2) circle (3pt); \draw (1.25,-2.25) node {\tiny $b'_3,b''_3$};
\fill (-1,0) circle (3pt); \draw (-1.5,0) node {\tiny $b'_1,b''_1$};
\fill (1,0) circle (3pt); \draw (0.35,0) node {\tiny $a'_1,a''_1$};
\fill (2,0) circle (3pt); \draw (2.35,-0.5) node {\tiny $a'_2, a''_2$};
\fill (3,0) circle (3pt); \draw (3.75,0) node {\tiny $a'_3,a''_3$};
\draw (-1,0)--(0,1);
\draw (-1,0)--(0,-1);
\draw (-1,0)--(0,-2);
\draw (1,0)--(0,1);
\draw (1,0)--(0,-1);
\draw (1,0)--(0,-2);
\draw (2,0)--(0,1);
\draw (2,0)--(0,-1);
\draw (3,0)--(0,1);
\draw (3,0)--(0,-1);
\draw (1,-2)--(0,-2);
\draw (1,-2)--(2,0);
\draw (-2,1) node {$(\Gamma_3)_2:$};
\end{scope}
\end{tikzpicture} 
\caption{\small{The duplex construction. In the picture on the right 
every bold vertex corresponds to a pair of vertices of $(\Gamma_3)_2$ connected by an edge labelled $m$, every thick black edge corresponds to four edges of $(\Gamma_3)_2$ labelled $2$, and every thick grey edge corresponds to four edges of $(\Gamma_3)_2$ labelled $n$.}}
\label{fig:g3}
\end{figure}
\label{ex:duplex}
\end{example}

Given a defining graph $\G$ for a right-angled Coxeter system $(W,S)$, we write $\G_2 = \G_2(m,n)$ for the defining graph of the duplex $(W_2(m,n),S_2)$. Given $T\subseteq S$, we refer to the set   
$\widetilde{T}:=\left\{u',u''  \mid u \in T \right\} \subseteq S_2$ as the \emph{lift} of $T$. Given $T\subseteq S_2$, we refer to the set 
$T_0:=\{u\in S \mid \{u',u''\} \cap T \neq \varnothing \} \subseteq S$ as the \emph{projection} of $T$.  In order to prove Theorem~\ref{thm:duplexhi}, we first characterise the minimal nonspherical subsets of $S_2$.

\begin{lemma}\label{lem:minDuplex} Let $A$ be a minimal nonspherical subset of $S_2$.  Then either:
\begin{enumerate}
    \item $|A| = 2$ and the subgraph of $\G_2$ induced by $A$ is an edge labelled by $n = \infty$; or
    \item $|A| = 3$ and the subgraph of $\G_2$ induced by $A$ is a triangle with edge labels $(p,n,n)$, where $p \in \{2,m,n\}$ and $n$ is finite.
\end{enumerate}
In particular, $A$ is not irreducible affine of cardinality $\ge3$.
\end{lemma}

\begin{proof} By Remark~\ref{rem:minimal}, $(W_A,A)$ is either irreducible affine, or has Dynkin diagram appearing in Figure~\ref{fig:lanner}. We then note that in the defining graph $\G_2$ of $(W_2(m,n),S_2)$, the edges labelled by $m$ are just those which connect lifts of single vertices.  Therefore if two adjacent edges of $\G_2$ both have labels greater than $2$, then either they must have labels $m \geq 3$ and $n$, or they are both labelled by $n$.  Since $m$ is finite and $n \geq 5$, it follows from the classification of irreducible affine Coxeter systems and Figure~\ref{fig:lanner}, respectively, that the only possibilities for~$A$ are as in the statement.
The last sentence follows easily from the description of $A$.
\end{proof}

\begin{example} Consider the duplex system in Example~\ref{ex:duplex} (see also Figure~\ref{fig:g3}), and assume~$n$ is finite. In this case, minimal nonspherical sets in $S_2$ include $\{a'_1, b'_3,b''_3\}$, $\{a'_1, a'_3,b''_3\}$ and $\{a'_1, a'_3,b''_2\}$. The subgraphs of $(\G_3)_2$ induced by these sets are all triangles, with labels $(m,n,n)$, $(n,n,n)$ and $(2,n,n)$, respectively. 
\label{ex:dupnonsph}
\end{example}

We now consider the relationships between lifts and projections of various subsets of $S$ and of $S_2$.

\begin{lemma}\label{lem:dupCommute}
Lifts and projections of commuting subsets also commute.
\end{lemma}  
\begin{proof}
If $m_{uv}=2$ for $u,v \in S$, then by construction the lifts of $\{u\}$ and $\{v\}$ also commute. Conversely, if $m_{st}=2$ for $s,t\in S_2$, then the projection of $\{s,t\}$ can only be either two distinct vertices in $\Gamma$ connected by an edge, or (in the case that $m=2$) a single vertex. 
\end{proof}

\begin{lemma} 
Lifts and projections of nonspherical subsets are also nonspherical.
\label{lem:dupNonsph}
\end{lemma}

\begin{proof} 
Given a nonspherical subset $A$  of $S$, the subgraph of $\G$ induced by $A$ contains a pair of nonadjacent vertices.  By Conditions (1) and (3) of Definition~\ref{defn:duplex}, this means that the defining graph of $(W_{\widetilde{A}},\widetilde{A})$ contains a triangle with edge labels $(m, n, n)$.
Thus as $n \geq 5$, the set $\widetilde{A}$ is also nonspherical.
Now if $T \subseteq S_2$ is nonspherical, it contains a minimal nonspherical subset. Then by Lemma~\ref{lem:minDuplex} the subgraph of $\G_2$ induced by $T$ contains an edge labelled $n$ that corresponds to a pair of  disconnected vertices in the subgraph of $\G$ induced by $T_0$, so~$T_0$ is also nonspherical.
\end{proof}

The next result says that lifts and projections of wide subsets are wide (see Definition~\ref{defn:wide}).
\begin{lemma} Let $(W,S)$ be a right-angled Coxeter system with  duplex $(W_2(m,n),S_2)$. 
\begin{enumerate}
    \item Every set in $\Omega(S_2)$ is the lift of a set in $\Omega(S)$.
    \item Every set in $\Omega(S)$ lifts to a set in $\Omega(S_2)$.
\end{enumerate}
\label{lem:dupWide}
\end{lemma}

\begin{proof} For (1), suppose $T \in \Omega(S_2)$. Then $T=A\times B$ for some $A,B \subseteq S_2$.  We may assume without loss of generality that $A$ is nonspherical. 
Let $A'$ be a minimal nonspherical subset of~$A$.  Then $A'$, and therefore $A$, cannot be irreducible affine of cardinality $\geq 3$, by Lemma~\ref{lem:minDuplex}. 
From the definition of wide subsets, we conclude that $B$ is also nonspherical.

Thus by Lemma~\ref{lem:dupNonsph}, the projections $A_0$ and $B_0$ are also nonspherical, and by Lemma~\ref{lem:dupCommute}, they must also commute. If $m\geq 3$, we also know that $A_0$ and $B_0$ are disjoint, since $A$ and $B$ are disjoint, and $m_{uv}=2$ for all $u\in A$ and $v\in B$. If $m=2$, it is possible for $A_0\cap B_0 \neq \varnothing$, but we observe that if $u \in A_0\cap B_0$, then $m_{uv}=2$ for all $v\in B_0$.  Thus, $A_0 \cap B_0$ is spherical and the relative complement $C:=B_0\backslash A_0$ is nonspherical. Therefore for any finite $m \geq 2$, we have that $A_0$ and $C$ are disjoint nonspherical commuting sets. 

This implies that there exists a wide subset $A'\times B'$ of $S$ that contains $A_0 \times C$, where $A'$ and $B'$ are disjoint nonspherical commuting sets. Furthermore, their lifts $\widetilde{(A')}$ and $\widetilde{(B')}$ are also disjoint nonspherical commuting sets (by construction, Lemma~\ref{lem:dupNonsph} and Lemma~\ref{lem:dupCommute}). Since $\widetilde{(A')}\times \widetilde{(B')}$ contains $A\times B=T$, by the maximality of $T$, we must have  $T = \widetilde{(A')}\times \widetilde{(B')}=\widetilde{A'\times B'}$. Thus $T$ is the lift of a wide subset of $S$.

For (2), suppose $T\in \Omega(S)$.  Then $T = A\times B$ for some $A,B \subseteq S$. Therefore $\widetilde{T}=\widetilde{A}\times \widetilde{B}$, by Condition~(2) of Definition~\ref{defn:duplex}. Further, since $(W,S)$ is a right-angled system, the sets $A$ and $B$ are both nonspherical 
(since there are no right-angled irreducible affine systems of rank $\ge3$).
Thus $\widetilde{A}$ and $\widetilde{B}$ are also nonspherical by Lemma~\ref{lem:dupNonsph}. 

To show maximality of $\widetilde{T}$, suppose that $\widetilde{T}$ is contained in a wide subset $T' \in \Omega(S_2).$ By 
part~(1), we know that $T'$ must be the lift of a wide subset $U \in \Omega(S)$, that is, $T' = \widetilde U$. Since $\widetilde{T} \subset T'$, we have $T \subset U \in \Omega(S)$. By the maximality of $T$, this means $T=U$, so $\widetilde{T} = T'$. Thus $\widetilde{T}$ is a wide subset of $S_2$. 
\end{proof}

\begin{cor}\label{cor:wideNonempty}
Let $(W,S)$ be a right-angled Coxeter system with duplex $(W_2(m,n),S_2)$.  Then $\Omega(S) \neq \varnothing$ if and only if $\Omega(S_2) \neq \varnothing$.\qed
\end{cor}

\begin{example} 
It is not difficult to verify that the wide subsets of $S$ from Example~\ref{ex:duplex} are the sets $\{a_0,b_0\}\times \{b_1,a_1,a_2,a_3\}$ and $\{a_1,b_1\}\times\{a_0,b_0,b_2\}$.  By Lemma~\ref{lem:dupWide}, the wide subsets of $S_2$ are exactly the lifts of these sets. 
\label{ex:dupwide}
\end{example}

We now consider the slab subsets of $S$ and of $S_2$ (see Definition~\ref{defn:slab}).  Here, the situation is more subtle.  Observe that the lift of an element of $\Psi(S)$ is not an element of $\Psi(S_2)$, since a minimal nonspherical set $A \subseteq S$ is a pair of disconnected vertices, but by Lemma~\ref{lem:minDuplex}, its lift  $\widetilde{A} \subseteq S_2$ (which has four elements) is not minimal nonspherical. The next lemma gives us a partial remedy.

 \begin{lemma} Let $T=A\times K \in \Psi(S)$, where $A$ is minimal nonspherical and $K$ is a maximal nonempty spherical subset commuting with $A$. Then there exists a set $\overline{T}=\overline{A}\times \overline{K} \in \Psi(S_2)$ such that the projection of $\overline{T}$ is $T$, $\overline{A}$ is a minimal nonspherical subset of $\widetilde{A}$ and $\overline{K}=\widetilde{K}.$
 \label{lem:slabcover}
 \end{lemma}

\begin{proof}  Since $S$ is right-angled, the set $A$ consists of a pair of disconnected vertices.  If $n = \infty$, choose $\overline{A}$ to be a pair of vertices in $\widetilde A$ that are connected in $\Gamma_2$ by an edge labelled $\infty$.  If~$n$ is finite, choose $\overline{A}$ to be any three vertices in $\widetilde A$, so that the subgraph induced by $\overline{A}$ is a triangle with finite edge labels $(m,n,n)$.  We observe that $\overline{A} \subset \widetilde A$ is minimal nonspherical, and that the projection of $\overline{A}$ is~$A$. 

We know that the lift $\widetilde{K}$ is spherical, or else its projection $K$ would not be spherical by Lemma~\ref{lem:dupNonsph}. Furthermore, $\widetilde{K}$ must be a maximal spherical subset of $S_2$ commuting with~$\overline{A}$, because any larger spherical set $L \supsetneq \widetilde{K}$ which commutes with $\overline A$ would have spherical projection $L_0 \supsetneq K$ that commutes with $A$ (by Lemmas~\ref{lem:dupCommute} and~\ref{lem:dupNonsph}), contradicting the maximality of $K$. 

Now let $\overline{T}=\overline{A} \times \widetilde{K}$. Then the projection of $\overline{T}$ is $T$.  The set $\overline{T}$ is not wide, as $\overline A$ is not irreducible affine of rank $\geq 3$ (see Remark~\ref{rem:duplex}).  Also, $\overline{T}$ cannot be properly contained in a wide subset of $S_2$, since by Lemma~\ref{lem:dupWide}, its projection $T = A \times K$ would then also be contained in a wide subset of $S$, contradicting the hypothesis that $T$ belongs to $\Psi(S)$. Therefore $\overline{T}$ is in $\Psi(S_2)$. Put $\overline{K} = \widetilde K$ and we are done.
 \end{proof}

\begin{definition}\label{def:slabcover} Let $T \in \Psi(S).$ Then if $\overline{T} \in \Psi(S_2)$ is as in the statement of Lemma~\ref{lem:slabcover}, we call $\overline{T}$ a \emph{covering slab}  of $T$.
\end{definition}

\noindent
From the proof of Lemma~\ref{lem:slabcover} we obtain a complete description of possible components $\overline A$ for a covering slab $\overline T$: if $n=\infty$, then $\overline A$ is a rank $2$ subsystem with the edge label $n=\infty$, and if $n$ is finite, then $\overline A$ is a rank $3$ subsystem with finite edge labels $(m,n,n)$. Clearly, covering slabs are not unique.

\begin{example} Continuing Example~\ref{ex:duplex}, consider $A=\{a_1, b_3\}$ and $K=\{b_2\}$. Then $A$ is minimal nonspherical, $K$ is the maximal spherical subset commuting with $A$, and $T=A\times K$ is not contained in any wide subset (see Example~\ref{ex:dupwide}). So $T$ is a slab subset of $S$.

Now assume that $n\ne\infty$ and let $\overline A=\{a'_1, b'_3,b''_3\} \subset S_2$. Then $\overline A$ is a minimal nonspherical subset of $S_2$, and the maximal spherical subset $\overline K \subset S_2$ which commutes with $\overline A$ is $\{b'_2,b''_2\}$. From Example~\ref{ex:dupwide}, we can see that the set $\overline{T} = \overline A \times \overline K=\{a'_1, b'_3,b''_3\}\times\{b'_2,b''_2\}$ is not contained in any wide subset, and hence is a covering slab of $T$.  Other covering slabs of $T$ are $\{a''_1, b'_3,b''_3\}\times\{b'_2,b''_2\}$, $\{a'_1,a''_1, b'_3\}\times\{b'_2,b''_2\}$ and $\{a'_1,a''_1, b''_3\}\times\{b'_2,b''_2\}$.
\label{ex:dupslab}
\end{example}

For convenience we will call the elements of $\Lambda_i(S) \setminus \Psi(S)$ the \emph{non-slab elements} of $\Lambda_i(S)$.  
The last result we will need before proving Theorem~\ref{thm:duplexhi} extends the lifting property of wide subsets to all non-slab elements of $\Lambda_i(S)$.

 \begin{prop} Assume that $\Omega(S) \neq \varnothing$, equivalently by Corollary~\ref{cor:wideNonempty} that $\Omega(S_2) \neq \varnothing$.  Then for $i \geq 0$, every non-slab element of $\Lambda_i(S)$ lifts to a non-slab element of $\Lambda_i(S_2)$, and every non-slab element of $\Lambda_i(S_2)$ is the lift of a non-slab element of $\Lambda_i(S)$.
\label{prop:nsarelifts}
\end{prop}

\begin{proof} We proceed by induction on $i$. As the wide subsets of $S_2$ are exactly the lifts of the wide subsets of $S$, by Lemma~\ref{lem:dupWide}, the result holds for $i=0$.

Now assume the result holds for some fixed $i$.  If $X\in \Lambda_{i+1}(S) \setminus \Psi(S)$, then Corollary~\ref{cor:non-slab} implies that any  $\equiv_i$-equivalence class whose union is $X$ has at least one non-slab element, say $T$. Then, by assumption, $\widetilde{T}$ is a non-slab element of $\Lambda_i(S_2)$. Let $Y$ be the union of the elements of the $\equiv_i$-equivalence class of $\widetilde{T}$.  Then  $Y \in \Lambda_{i+1}(S_2) \setminus \Psi(S_2)$  by Corollary~\ref{cor:non-slab}.   
Moreover, Corollary~\ref{cor:non-slab} and the induction hypothesis imply that every non-slab element of $\Lambda_{i+1}(S_2)$ has this form, 
i.e.\ is the union of the elements of the  $\equiv_i$-equivalence class of a non-slab element which is the lift of a non-slab element of $\Lambda_i(S)$.  Thus, both statements of the proposition will be proved if we show that $\widetilde{X}=Y$.  We achieve this by first showing  that $\widetilde X\subseteq Y$  and $Y_0 \subseteq X$.  This will complete the proof, since the latter statement implies that $Y \subseteq \widetilde{Y_0} \subseteq \widetilde X$.

To prove that $\widetilde X \subseteq Y$, fix $x \in X$. 
We 
need to show that the lift $\{x',x''\}$ is contained in $Y$. Let $U \in \Lambda_{i}(S)$ be a subset of $S$ such that $x\in U$ and $T\equiv_i U$. Thus there exists a sequence $T=T_0$, $T_1$, \ldots, $T_\ell=U$ of distinct sets in $\Lambda_{i}(S)$ such that $T_j \cap T_{j+1}$ is nonspherical for $0\leq j < \ell$. We now construct a sequence of sets $\widehat{T}_j \in \Lambda_i(S_2)$ with similar properties.

For $0 \leq j \leq \ell$, if $T_j$ is a non-slab element, then take $\widehat T_j=\widetilde T_j$, and note that $\widehat T_j \in \Lambda_i(S_2)$ by the inductive hypothesis. 
Otherwise, let $ \overline T_j$ be any covering slab of $T_j$, and let $\widehat T_j$ be the element of $\Lambda_i(S_2)$ such that $\overline T_j$ feeds into $\widehat T_j$, which exists by Lemma~\ref{lem:monotone}(1).

We claim that $\widehat T_j\cap\widehat T_{j+1}$ is nonspherical for each $j$.  
This is obvious for non-slab $T_j$ and $T_{j+1}$, since then $\widehat T_j\cap \widehat T_{j+1}=\widetilde T_j\cap\widetilde T_{j+1}=\widetilde {T_j\cap T_{j+1}}$, which is nonspherical by Lemma~\ref{lem:dupNonsph}. If, say, $T_j$ is slab and $T_{j+1}$ is non-slab, then their intersection contains a pair of vertices connected by an edge with label infinity.  From the construction of 
$\overline T_j$, we see that $\overline T_j\cap \widetilde T_{j+1}$ is nonspherical.  Thus 
 $\widehat T_j\cap \widehat T_{j+1}$ is nonspherical, as it contains $\overline T_j\cap \widetilde T_{j+1}$.  Finally, Lemma~\ref{lem:intersection} implies that $T_{j}$ and $T_{j+1}$ cannot both be slab. 
This proves the claim.  
It follows that $\widetilde{T} \equiv_i \widehat{T_{\ell}}.$

If $U$ is non-slab then $\widehat{T_{\ell}} = \widetilde U$ contains the lift $\{x',x''\}$ of $\{x\}$.  If $U$ is slab, then there exist covering slabs $\overline U'$ and $\overline U''$ that contain $x'$ and $x''$ 
and feed into $\widehat T_{\ell}'$ and $\widehat T_{\ell}''$
respectively. 
Then from the previous paragraph we know that $T \equiv_i \widehat T_{\ell}'$  and $T \equiv_i \widehat T_{\ell}''$. It follows that 
both $x'$ and $x''$  belong to $Y$ and hence $\widetilde{X} \subseteq Y.$

It remains to show that $Y_0 \subseteq X$.  Let $y\in Y.$  Then there is a $U \in \Lambda_{i}(S_2)$ such that $y\in U$ and $\widetilde{T} \equiv_i U$. So there exists a sequence $\widetilde{T}=U_1$, $U_2$, \ldots, $U_\ell=U$ of distinct sets in $\Lambda_{i}(S_2)$ such that $U_j \cap U_{j+1}$ is nonspherical for $1\leq j < \ell$. We will  construct a sequence of subsets $V_1$, \dots, $V_\ell$ in $\Lambda_i(S)$ such that $V_1=T$, the set  $V_\ell$ contains the projection of $y$, and $V_j\cap V_{j+1}$ is nonspherical for all $0\le j<\ell$. 

If $U_j$ is non-slab, let $V_j=(U_j)_0$, and note that $V_j \in \Lambda_i(S)$ by the inductive hypothesis. If $U_j$ is slab, then $U_j=A\times K$ where $A$ is a minimal nonspherical subset of $S_2$ and $K$ is a maximal nonempty spherical subset of $S_2$ commuting with $A$. For $j<\ell$, let $A'$ be any minimal nonspherical subset of the projection $A_0$ (which is nonspherical by Lemma~~\ref{lem:dupNonsph}), and for $j=\ell$ we additionally require that $A'$ contains the projection $\{y\}_0$ of $y$, if $y\in A$ (we appeal here to Lemma~\ref{lem:minDuplex}). Note that $K_0$ commutes with $A'$ by Lemma~\ref{lem:dupCommute}.  
Then we  know from Defintions~\ref{defn:wide} and~\ref{defn:slab} that $A \times K_0$ is contained in either a wide or slab subset in $\Lambda_0(S)$, and thus feeds into some element of $\Lambda_i(S)$ by Lemma~\ref{lem:monotone}(1).
Denote this subset $V_j$. 

Note that by Lemma~\ref{lem:dupNonsph}, each projection $(U_j \cap U_{j+1})_0$ is nonspherical. 
If $U_j$ and $U_{j+1}$ are both non-slab, then $V_j$ and $V_{j+1}$ are non-slab as well by the inductive hypothesis, and $V_j \cap V_{j+1}=(U_j)_0 \cap (U_{j+1})_0$ is nonspherical, since this intersection contains $(U_j \cap U_{j+1})_0$. 
If, say, $U_j$ is non-slab but $U_{j+1}$ is slab, then $U_{j+1}=A\times K$ as described above, and we observe that since $A$ is minimal nonspherical, it is contained in the nonspherical intersection $U_j\cap U_{j+1}$. In particular, we see that $A'\subseteq (U_j)_0 \cap V_{j+1}$, so $V_j\cap V_{j+1}$ is nonspherical in this case  as well. Finally, Lemma~\ref{lem:intersection} implies that no two consecutive elements of the sequence $U_1$,\dots, $U_\ell$ are slab. We have shown that every pair of consecutive elements of  the sequence $T=V_1$, $V_2$, \ldots, $V_\ell$ has nonspherical intersection. 
Therefore $T \equiv_i V_\ell$, and by construction, the projection $\{y\}_0$ is contained in $V_\ell$. Thus $\{y\}_0 \subset X$ and we conclude that $Y_0 \subseteq X.$
\end{proof}

We can now prove Theorem~\ref{thm:duplexhi}. 

\begin{proof}[Proof of Theorem~\ref{thm:duplexhi}] By Corollary~\ref{cor:wideNonempty}, we know that $\Omega(S)\neq \varnothing$ if and only if $\Omega(S_2) \neq \varnothing$.  If $\Omega(S)$ and $\Omega(S_2)$ are both empty then $(W,S)$ and its duplex have infinite hypergraph index and we are done.  So we suppose that $\Omega(S)\neq \varnothing$ and $\Omega(S_2) \neq \varnothing$.  
Then $S$ contains a wide subset, and so by Lemma~\ref{lem:widenotinslab}, 
$S \not \in \Psi(S)$. Similarly, $S_2 \not \in \Psi(S_2)$.

Assume that $S_2$ is a set in $\Lambda_i(S_2)$ for some $i\geq 0$.  Then it is a non-slab subset, that is, $S_2 \in \Lambda_i(S_2) \setminus \Psi(S_2)$. By Proposition~\ref{prop:nsarelifts}, $S_2 \in \Lambda_i(S_2) \setminus \Psi(S_2)$ if and only if $S \in \Lambda_i(S) \setminus \Psi(S)$, since~$S_2$ is the lift of $S$. Therefore, $S_2$ is in $\Lambda_h(S_2)\backslash\Lambda_{h-1}(S_2)$ if and only if $S$ is in $\Lambda_h(S)\backslash\Lambda_{h-1}(S)$, and so the two Coxeter systems must have the same hypergraph index.
\end{proof}

\subsection{Relationships between  duplexes}\label{sec:relationships}

In Proposition~\ref{prop:distinct} we  show that as we vary the parameters $m$ and $n$ in the duplex construction, we obtain infinitely many pairwise nonisomorphic Coxeter groups.  We then state some open questions concerning the commensurability and quasi-isometry classification of duplex groups.

\begin{prop}\label{prop:distinct}  Let $(W,S)$ be a nonspherical right-angled Coxeter system.  Let $m$ and~$m'$ be integers~$\geq 2$ and let $n$ and $n'$ be either integers $\geq 5$ or equal to $\infty$.  Then the Coxeter groups $W_2(m,n)$ and $W_2(m',n')$ are abstractly isomorphic if and only if $(m,n) = (m',n')$.  
\end{prop}

\begin{proof} Assume $W_2(m,n)$ and $W_2(m',n')$ are isomorphic.  We wish to show that $(m,n) = (m',n')$. Since $(W,S)$ is nonspherical, $(W_2(m,n),S_2)$ is nonspherical for all $m$ and $n$ by Lemma~\ref{lem:dupNonsph}. In particular, there exist edges labelled $n$ in the Dynkin diagram of $(W_2(m,n),S_2)$ (and similarly for $n'$ and $W_2(m',n')$). 
Up to symmetry, there are three cases possible:
\begin{enumerate}
\item both $n$ and $n'$ are $\infty$;
\item $n$ is finite, but $n'$ is $\infty$;
\item both $n$ and $n'$ are finite.
\end{enumerate}

We will make use of the following facts: 
\begin{itemize}
\item[(i)] every finite subgroup of a Coxeter group is conjugate into a finite special subgroup (see e.g.\  \cite[Proposition\,1.3]{Brink-Howlett} or \cite[Theorem~4.5.3]{Bjorner-Brenti}); and, moreover,
\item[(ii)]
 the conjugacy classes of the maximal finite subgroups of a Coxeter group are in one-to-one correspondence with the maximal spherical subsets of its Coxeter generating set (see~\cite[Theorem~1.9]{brady-mccammond-muhlherr-neumann}).
\end{itemize}

In case (1) the finite subgroups of maximal order in $W_2(m,\infty)$ are all isomorphic to the direct product of $N$ dihedral groups of order $2m$, where $N$ is the maximum number of vertices in a clique in the defining graph of $(W,S)$.  (These dihedral subgroups of $W_2(m,\infty)$ are generated by $\cup_{i=1}^N \{ u_i', u_i''\}$, where $\{u_1,\dots,u_N\} \subseteq S$ is the vertex set of a maximal clique in the defining graph for $(W,S)$.)
Thus, the maximum order of a finite subgroup of $W_2(m,\infty)$ is $(2m)^N = 2^N m^N$, 
and this maximum is achieved. Similarly, the maximum order of a finite subgroup of $W_2(m',\infty)$ is $2^N(m')^N$, and this maximum is achieved as well. It follows that $m = m'$.

In case (2) we observe that, as in case (1) above, the maximal spherical subsets of $(W_2(m',\infty),S_2)$ bijectively correspond to cliques in the defining graph of $(W,S)$. We notice that each such clique gives rise to a maximal spherical subset in the system $(W_2(m,n),S_2)$ as well. However, in the system $(W_2(m,n),S_2)$ there are also $2$-element spherical subsets consisting of edges labelled $n$ and they are not contained in any of the maximal spherical subsets coming from cliques. By (ii) above, we conclude that $W_2(m,n)$ contains more conjugacy classes of maximal finite subgroups than $W_2(m',\infty)$, and therefore these Coxeter groups cannot be isomorphic. (Notice that our reasoning allows the situation when $n=m$.)

In case (3) we appeal to Theorem~1(b) of~\cite{franzsen-howlett-muhlherr}, which says that a nonspherical Coxeter system with all edge labels finite is rigid, meaning that the corresponding Dynkin diagram is unique up to the isomorphism of edge-labelled graphs. In our setting, this means there is a label preserving graph isomorphism $f\colon \Delta \to \Delta'$, where $\Delta$ and $\Delta'$ are the Dynkin diagrams corresponding to $(W_2(m,n),S_2)$ and $(W_2(m',n'),S_2)$, respectively.   Note that each vertex of $\Delta$ has at most one incident edge with label $m$ (one if $m\neq 2$ and none if $m=2$), and
an even number of incident edges with label $n$. 
In particular, every vertex has even valence if and only if $m=2$.  
The
 analogous statements hold for $\Delta'$. 
 It follows that $m=2$ if and only if $m'=2$.  Now if $m=m'=2$, then $n=n'$, as $n$ and $n'$ are the only labels. Otherwise, if $m, m' \neq 2$, the equality of edge labels of $\Delta$ and $\Delta'$ gives
$\{m,n\}=\{m',n'\}$.  Now note that 
since $(W,S)$ is nonspherical, its Dynkin diagram has at least one edge with label $\infty$, and it follows that there is a vertex $v$ in $\Delta$ with at least two incident edges labelled $n$.  Since $f(v)$ is incident to just one edge labelled $m'$, at least one of these $n$-labelled edges is mapped to an edge incident to $f(v)$ labelled $n'$.  Thus $n=n'$, and so $m=m'$. 
\end{proof}

The following open questions concern the coarse geometric relationships between duplexes.  We say  that two groups $G$ and $H$ are \emph{commensurable} if they have isomorphic finite index subgroups.

\begin{questions}
Let $W$ be a nonspherical right-angled Coxeter group, let $m$ and $m'$ be integers~$\geq 2$ and let $n$ and $n'$ be either integers $\geq 5$ or equal to $\infty$. 
\begin{enumerate}
\item
\begin{enumerate}
    \item Is $W$ commensurable to the duplex $W_2(m,n)$? 
    \item If $(m,n) \neq (m',n')$, are $W_2(m,n)$ and $W_2(m',n')$ commensurable?
\end{enumerate}
\item 
\begin{enumerate}
    \item Is $W$ quasi-isometric to the duplex $W_2(m,n)$?
    \item If $(m,n) \neq (m',n')$, are $W_2(m,n)$ and $W_2(m',n')$ quasi-isometric?
\end{enumerate}
\end{enumerate}
\end{questions}

We have included 1(a) and 2(a) here since a positive answer to these questions immediately implies a positive answer to 1(b) and 2(b), respectively.  Note also that positive answers to 1(a) and 1(b) imply positive answers to 2(a) and 2(b), respectively.
Concerning 2(a), we note that the ``obvious" map from $W_2(m,n)$ back to $W$, taking $u',u'' \in S_2$ to $u \in S$, is not a quasi-isometry.  For example, if $m_{uv} = \infty$ then $\langle u',u'', v',v''\rangle$ is a convex  subgroup of $W_2(m,n)$, and is not $2$-ended, but its image under this map is the convex  subgroup $\langle u,v\rangle \cong D_\infty$ of $W$, which is $2$-ended. Similar reasoning shows that a map from $W$ to $W_2(m,n)$ induced by sending each $u \in S$ to some element of the finite dihedral group $\langle u',u'' \rangle$ is not a quasi-isometry.

\subsection{Examples with arbitrary hypergraph index}\label{sec:DTexamples}

We now show how to obtain examples with hypergraph index any non-negative integer, and thus prove part of Corollary~\ref{cor:duplexIntro}. 

For $d \geq 1$, let $\G_d$ be the defining graph from Figure~\ref{fig:dtgammas}; these examples first appeared in~\cite[Section~5]{dani-thomas-div}.  
While it is possible to show directly that the associated right-angled Coxeter system has hypergraph index $d-1$, we use the following two results instead.

\begin{prop}[Propositions 5.1 and 5.3 of~\cite{dani-thomas-div}]  \label{prop:dtdivd}
For $d \geq 1$, the right-angled Coxeter group defined by $\Gamma_d$ has divergence polynomial of degree $d$.
\end{prop}

\begin{thm}[Theorem~A of~\cite{levcovitz-RACG-2020}] \label{thm:levcohidiv}
For $d \geq 1$, a right-angled Coxeter group has divergence polynomial of degree $d$ if and only if its defining graph has hypergraph index $d-1$.
\end{thm}

\begin{figure}
\begin{center}
\begin{tikzpicture}[scale=0.85,semithick]
% Gamma_2
\begin{scope}[xshift=-5cm]
\fill (0,1) circle (2.3pt); 
\fill (0,-1) circle (2.3pt);
\fill (-1,0) circle (2.3pt);
\fill (1,0) circle (2.3pt); 
\draw (-1,0)--(0,1);
\draw (-1,0)--(0,-1);
\draw (1,0)--(0,1);
\draw (1,0)--(0,-1);
\draw (0,1.75) node {$\Gamma_1$};
\end{scope}

% Gamma_2
\begin{scope}[xshift=-1.5cm]
\fill (0,1) circle (2.3pt); 
\fill (0,-1) circle (2.3pt); 
\fill (0,-2) circle (2.3pt); 
\fill (-1,0) circle (2.3pt); 
\fill (1,0) circle (2.3pt); 
\fill (2,0) circle (2.3pt); 
\draw (-1,0)--(0,1);
\draw (-1,0)--(0,-1);
\draw (-1,0)--(0,-2);
\draw (1,0)--(0,1);
\draw (1,0)--(0,-1);
\draw (1,0)--(0,-2);
\draw (2,0)--(0,1);
\draw (2,0)--(0,-1);
\draw (0,1.75) node {$\Gamma_2$};
\end{scope}

\draw (2.5,0) node {\Huge\dots};

% Gamma_d
\begin{scope}[xshift=5.5cm]
\fill (0,1) circle (2.3pt); \draw (0,0.65) node {\footnotesize $a_0$};
\fill (0,-1) circle (2.3pt); \draw (0,-0.65) node {\footnotesize $b_0$};
\fill (0,-2) circle (2.3pt); \draw (0,-2.3) node {\footnotesize $b_2$};
\fill (1,-2) circle (2.3pt); \draw (1,-2.3) node {\footnotesize $b_3$};
\fill (2,-2) circle (2.3pt); \draw (2,-2.3) node {\footnotesize $b_4$};
\fill (4,-2) circle (2.3pt); \draw (4,-2.3) node {\footnotesize $b_d$};
\fill (-1,0) circle (2.3pt); \draw (-1.25,0) node {\footnotesize $b_1$};
\fill (1,0) circle (2.3pt); \draw (0.65,0) node {\footnotesize $a_1$};
\fill (2,0) circle (2.3pt); \draw (1.6,0) node {\footnotesize $a_2$};
\fill (3,0) circle (2.3pt); \draw (2.5,0) node {\footnotesize $a_3$};
\fill (5,0) circle (2.3pt); \draw (5.4,-0.35) node {\footnotesize $a_{d-1}$};
\fill (6,0) circle (2.3pt); \draw (6.45,0) node {\footnotesize $a_{d}$};
\draw (3.75,0) node {\dots};
\draw (-1,0)--(0,1);
\draw (-1,0)--(0,-1);
\draw (-1,0)--(0,-2);
\draw (1,0)--(0,1);
\draw (1,0)--(0,-1);
\draw (1,0)--(0,-2);
\draw (2,0)--(0,1);
\draw (2,0)--(0,-1);
\draw (3,0)--(0,1);
\draw (3,0)--(0,-1);
\draw (1,-2)--(0,-2);
\draw (1,-2)--(2,0);
\draw (2,-2)--(1,-2);
\draw (2,-2)--(3,0);
\draw (5,0)--(0,1);
\draw (5,0)--(0,-1);
\draw (6,0)--(0,1);
\draw (6,0)--(0,-1);
\draw (4,-2)--(5,0);
\draw (2,-2)--(2.5,-2);
\draw (3.5,-2)--(4,-2);
\draw (3,-2) node {\dots};
\draw (0,1.75) node {$\Gamma_d$};
\end{scope}
\end{tikzpicture}
\end{center}
\caption{\small{A sequence of defining graphs $\Gamma_d$ of right-angled Coxeter systems from~\cite{dani-thomas-div}.} }
\label{fig:dtgammas}
\end{figure}
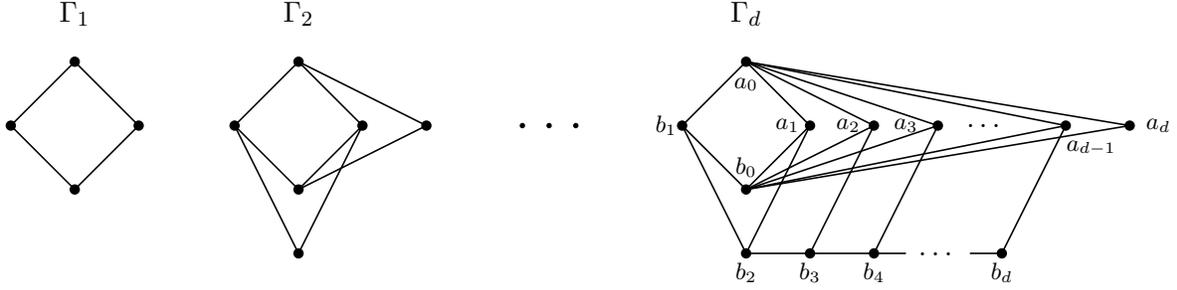

By combining these results with   Theorem~\ref{thm:duplexhi}, we obtain:

\begin{cor}\label{cor:hypDuplex}  Let $\G_d$ be the $d$-th graph in the sequence shown in Figure~\ref{fig:dtgammas}, with associated right-angled Coxeter group $W_{\G_d}$.   For all $m \ge 2$ and $n \geq 5$ (including $n = \infty$), the duplex group $(W_{\G_d})_2(m, n)$ has hypergraph index $d-1$. \qed
\end{cor}

Now Corollary~\ref{cor:hypDuplex} and Proposition~\ref{prop:distinct} together establish the existence of infinitely many pairwise nonisomorphic examples of non-right-angled Coxeter groups with hypergraph index~$d -1$.  Thus  the first sentence in Corollary~\ref{cor:duplexIntro} of the introduction holds.

\subsection{Examples with arbitrary divergence}\label{sec:examplesDiv}

We now prove that for all positive integers $d$, certain of the duplex groups from Section~\ref{sec:DTexamples} have divergence polynomial of degree~$d$ and are strongly thick of order~$d-1$.  This completes the proof of Corollary~\ref{cor:duplexIntro} of the introduction.  Specifically, we prove:
  
\begin{prop}\label{prop:divd}
Let $\G_d$ be the $d$-th graph in the sequence shown in Figure~\ref{fig:dtgammas}, with associated right-angled Coxeter group $W_{\G_d}$.   Let $m \ge 2$ be an even number.  Then the duplex group $(W_{\G_d})_2(m, \infty)$ has divergence $\simeq r^d$. 
\end{prop}

\begin{cor}\label{cor:thickness}  For $m \geq 2$ even, $(W_{\G_d})_2(m, \infty)$ is strongly thick of order $d-1$.
\end{cor}
\begin{proof}[Proof of Corollary~\ref{cor:thickness}]  By Corollary~\ref{cor:hypDuplex}, the group  $(W_{\G_d})_2(m, \infty)$ has hypergraph index~$d-1$.  Hence by Theorem~\ref{thm:finiteThick}, $(W_{\G_d})_2(m, \infty)$ is strongly thick of order at most $d-1$.  If it is strongly thick of order $n < d-1$, then by Theorem~\ref{cor:thickDiv} its divergence is bounded above by $r^{n+1}$ with $n +1 < d$.  This  contradicts Proposition~\ref{prop:divd}.
\end{proof}

We note that 
the duplex Coxeter groups in Proposition~\ref{prop:divd} are right-angled 
only when $m=2$.  Thus by Proposition~\ref{prop:distinct} and Corollary~\ref{cor:thickness}, for every $d \geq 1$, as $m$ varies we obtain infinitely many non-right-angled Coxeter groups which have degree $d$ polynomial divergence, and for which Conjecture~\ref{conj:lower} holds.  As discussed in Section~\ref{sec:relationships}, we do not know whether these duplex groups are quasi-isometric to 
 the right-angled Coxeter group defined by $\G_d$. 

In order to prove Proposition~\ref{prop:divd}, we first introduce some background, fix some notation, and state a technical lemma (Lemma~\ref{lem:divd-lower} below). We then give a short proof of Proposition~\ref{prop:divd}; the upper bound in the proposition is obtained from  the hypergraph index, and the lower bound follows easily from Lemma~\ref{lem:divd-lower}.  
We defer the proof of Lemma~\ref{lem:divd-lower}  till the end of the section.  Its proof follows the same general outline as the analogous result in the right-angled case~\cite[Proposition~5.3]{dani-thomas-div}, but the details are considerably more involved. 

For the remainder of this section, we fix an even integer $m \ge 2$ and
denote the duplex Coxeter group $(W_{\G_d})_2(m, \infty)$ simply by $W_d$.
 We denote the  defining graph  $(\G_d)_2$ of $W_d$ by $\Phi_d$ and  its Cayley graph by 
 $\Cayley{d}$.  We write $\Sigma_d$ for the Davis complex of $W_d$ (with its duplex generating set), and recall that $\Sigma_d$ has $1$-skeleton $\Cayley{d}$.  (See the end of Section~\ref{sec:CoxeterDavis} for background on the structure of the Davis complex.)

 Since $m$ is even, every $m_{st}$ in the presentation of $W_d$ is either even or infinity.  It follows that  given any wall $H$ of $\Sigma_d$, 
 each edge of  $\Cayley{d}$ dual to $H$ is labelled by the same generator; this generator is called 
the \emph{type} of $H$.
Note that unlike in the right-angled case, walls of the same type can intersect each other.  
If two walls intersect, then either they have the same type, or there is an edge between their types in $\Phi_d$ (which means that the corresponding label $m_{st}$ is not $\infty$).

Given a wall $H$ of $\Sigma_d$, we define $N_H$ to be the intersection of all (closed) half-spaces of~$\Sigma_d$ which contain~$H$ in their interior; if there are no such half-spaces (equivalently, if every wall of $\Sigma_d$ intersects $H$), we define $N_H = \Sigma_d$. Let $N_H^1$ denote the subgraph of~$\cC_d$ induced by the vertices of~$\cC_d$ which lie in $N_H$.  
 We say that a path $\gamma$ in $\cC_d$ \emph{travels along $H$} if $\gamma$ lies in~$N_H^1$ but does not cross~$H$.  Observe that these definitions make sense for arbitrary Coxeter systems, not just the duplex $W_d$.
 
 \begin{lemma}\label{lem:travel}
 Let $H$ be a wall of $\Sigma_d$ and let $N_H^1$ be as defined above.  Then:
 \begin{enumerate}
 \item If $v$ is a vertex of $\cC_d$ which is the endpoint of an edge dual to $H$, then $v \in N_H^1$.     In particular, $N_H^1$ is nonempty.
 \item Let $v$ and $w$ be any two vertices of $N_H^1$, and let $\gamma$ be any geodesic in $\cC_d$ from $v$ to $w$.  Then $\gamma$ is contained in $N_H^1$, and if in addition $v$ and $w$ are in the same component of the complement of $H$, then $\gamma$ travels along $H$.
 \item If a path $\gamma$ in $\cC_d$ travels along $H$, then the label of any of its edges lies in the star of the type of $H$ in $\Phi_d$. 
 \end{enumerate}
  \end{lemma}

\begin{proof}
If $N_H = \Sigma_d$ then $N_H^1 = \cC_d$ and (1) is obvious.  Otherwise, let $H'$ be any wall of $\Sigma_d$ which is disjoint from $H$.  Write $H'_v$ for the half-space of $\Sigma_d$ which is bounded by~$H'$ and contains $v$.  Then since no wall separates~$v$ from $H$, and $H$ and $H'$ are disjoint, the wall~$H$ is contained in the interior of $H'_v$.  That is, for every wall $H'$ which is disjoint from $H$, the half-space bounded by $H'$ which contains~$v$ equals the half-space bounded by $H'$ which contains $H$ in its interior.  Hence $v \in N_H$, and so $v \in N_H^1$. 

Now by \cite[Lemma~4.5.7]{davis-book}, the subgraph $N_H^1$ of $\cC_d$ is convex in the word metric, so the first claim in (2) holds.  If $v$ and $w$ are in the same component of the complement of $H$, the path $\gamma$ contains~$2k$ edges which are dual to $H$, for some integer $k \geq 0$.  But since $\gamma$ is a geodesic, it crosses each wall at most once, so $k = 0$.  Therefore $\gamma$ travels along~$H$, as required.

For (3), as $\gamma$ travels along $H$, each of its edges is in $N_H^1$.  By the definition of $N_H^1$, any wall~$K$ dual to such an edge necessarily intersects  $H$.
Then, as noted above, the types of $K$ and $H$ are either equal or adjacent in $\Phi_d$.  
\end{proof}

\noindent We note that the proofs of (1) and (2) in Lemma~\ref{lem:travel} are valid for all Coxeter systems, while~(3) holds whenever all finite $m_{st}$ in the presentation are even, so that walls have well-defined types.

Returning to our setting of duplex groups, the following lemma gives a lower bound on the divergence of pairs of geodesics that travel along certain types of walls in $\Sigma_d$. Recall that $W_{\Gamma_d}$ is generated by the elements $a_i,b_i$ for $i=0,\dots,d$, and that its duplex $W_d$ has corresponding generators $a_i',a_i'',b_i',b_i''$.  The $(p,r)$-avoidant distance $d^{\mathrm{av}}_{p,r}$ in the following statement is defined in Section~\ref{sec:div}.

\begin{lemma}\label{lem:divd-lower}
For each $1 \le  k \le d$, there exists a constant $C_k$  with the following property:\\
Given any 
$n$ with
$k \le n \le d$, and any geodesic rays $\alpha$ and $\beta$  in $\Cayley{{d+3}}$ satisfying (1)--(3) below:
\begin{enumerate}
\item $\alpha$ and $\beta$ have the same initial vertex, say $p$;
\item $\alpha$ travels along a wall of type $b_{n}'$ or $b_{n}''$; and
\item $\beta$ travels along a wall of type  $b_{n+1}'$ or $b_{n+1}''$;
\end{enumerate}
the following inequality holds for all
$r\ge 3^{k+1}M$ (where $M = 4m+1$), 
and for all $r_1, r_2 \ge r$:
\[
d^{\mathrm{av}}_{p,r}(\alpha(r_1), \beta(r_2)) \ge C_k r^k.
\]
\end{lemma}

In practice, for a fixed $n$, we will use the best lower bound provided by this lemma, which occurs when $k=n$. 
However, the above formulation is needed for a proof by induction (on~$k$). As will be evident from the proof strategy, it is not easy to induct on $n$ directly; this has to do with the fact that walls of type $b_{n}'$ or $b_{n}''$ can cross walls of type $b_{n+1}'$ or $b_{n+1}''$.  

\bigskip

Before proving Lemma~\ref{lem:divd-lower}, we explain how this immediately gives Proposition~\ref{prop:divd}.
\begin{proof}[Proof of Proposition~\ref{prop:divd}]
By Proposition~\ref{prop:dtdivd}, Theorem~\ref{thm:levcohidiv} and Theorem~\ref{thm:duplexhi}, the hypergraph index of $W_d$ is $d-1$.  Then it follows from Corollary~\ref{cor:hDiv}
 that the divergence of $W_d$ is $\preceq r^d$.  
To obtain the  lower bound on divergence, we
define 
$\alpha$ (respectively,~$\beta$) to be the geodesic ray  based at $e$  with label $b'_{d-1}a'_{d-1}b'_{d-1}a'_{d-1}\dots$
 (respectively,~$b_d'a_d'b_d'a_d' \dots$), which travels along a wall of type $b'_{d}$ (respectively,~$b'_{d+1}$)
in $\Cayley{{d+3}}$.  Thus, 
we can apply Lemma~\ref{lem:divd-lower} with $k=n=d$ to conclude that  
$d^{\mathrm{av}}_{e,r}(\alpha(r), \beta(r)) \ge C_d r^d,$ 
for all $r> 3^{d+1}(4m+1)$,
 where the avoidant distance is being computed in $\Cayley{{d+3}}$.    Any avoidant path 
 between $\alpha(r)$ and $\beta(r)$
in $\Cayley{{d}}$ remains avoidant under the isometric inclusion $\Cayley{{d}} \hookrightarrow \Cayley{{d+3}}$, and therefore also 
has length bounded below by $C_d r^{d}.$ This proves the desired lower bound on the  divergence of $W_d$. 
\end{proof}

We now introduce some terminology and a lemma to be used in the proof of Lemma~\ref{lem:divd-lower}.
For any vertex $x$ of $\G_d$, we use $x^*$ to denote one or both of the lifts of $x$, i.e.,~ $x'$ and/or  $x''$ (depending on the context).
We sometimes refer to $i$ as the \emph{subscript} of $b_i^*$.

The \emph{support of a word} is the set of letters that appear in it.  
We define the \emph{support of a path} (finite or infinite) in $\Cayley{d}$ to be the support of its label. 
Recall that the label of any subsegment of a geodesic is a reduced word.
The following lemma 
will be applied to labels of geodesic segments a number of times  in the proof of Lemma~\ref{lem:divd-lower}. 

\begin{lemma}\label{lem:geod-label}
Let $w$ be a reduced word in $W_k$, with $k$ arbitrary. 
\begin{enumerate}
\item
If the support of $w$ is contained in the star of a vertex $x'$ (or $x''$) in $\Phi_k$, then $w$ contains at most $m$ instances of $x^*$.
\item 
If the support of $w$ is contained in the lift of an adjacent pair of vertices $\{x, y\}$ of $\G_k$, then 
$|w| \le 2m$.
 \end{enumerate}

\end{lemma}
\begin{proof}
For (1), note that the stars of $x'$ and $x''$ in $\Phi_k$ are equal, and  all edges in this star, except possibly the edge between $x'$ and $x''$ (which is labelled $m$), are labelled $2$. Thus, $x^*$  commutes with every 
letter of $w$ that is not in $\{x', x''\}$, and so 
$w$ can be rearranged, by commuting letters, to a  word of the same length which has  the form $w_xw'$, where $w_x$ has support contained in $\{x', x''\}$, and the support of $w'$ does not contain $x^*$.  Since $w$ is reduced, so is $w_x$.    Now since $(x'x'')^m =1$, a word of length greater than $m$ in these letters is not reduced. Thus,  $|w_x| \le m$, and (1) follows.

For (2), we note that the lift of $\{x, y\}$ is contained in the star of $x^*$ as well as the star of  $y^*$ in $\Phi_k$, and we apply  (1) to both $x^*$ and $y^*$. 
\end{proof}

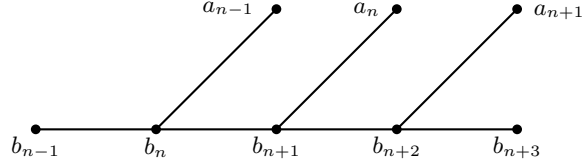
\begin{figure}
\begin{center}
\begin{tikzpicture}[scale=0.8,thick]
\fill (0,-2) circle (2.3pt); \draw (0,-2.3) node {\footnotesize $b_{n-1}$};
\fill (2,-2) circle (2.3pt); \draw (2,-2.3) node {\footnotesize $b_n$};
\fill (4,-2) circle (2.3pt); \draw (4,-2.3) node {\footnotesize $b_{n+1}$};
\fill (6,-2) circle (2.3pt); \draw (6,-2.3) node {\footnotesize $b_{n+2}$};
\fill (8,-2) circle (2.3pt); \draw (8,-2.3) node {\footnotesize $b_{n+3}$};

\fill (4,0) circle (2.3pt); \draw (3.2,0) node {\footnotesize $a_{n-1}$};
\fill (6,0) circle (2.3pt); \draw (5.5,0) node {\footnotesize $a_n$};
\fill (8,0) circle (2.3pt); \draw (8.7,-0.05) node {\footnotesize $a_{n+1}$};

\draw (2,-2)--(0,-2);
\draw (2,-2)--(4,0);
\draw (4,-2)--(2,-2);
\draw (4,-2)--(6,0);
\draw (6,-2)--(4,-2);
\draw (6,-2)--(8,0);
\draw (8,-2)--(6,-2);
\end{tikzpicture}
\end{center}
\caption{\small{A subgraph of $\Gamma_d$ induced by the stars of the vertices $b_n$, $b_{n+1}$ and $b_{n+2}$, used in the proof of Lemma~\ref{lem:divd-lower}.}}
\label{fig:bnstar}
\end{figure}

We are now ready to prove Lemma~\ref{lem:divd-lower}.  
In the proof below, given a geodesic $\gamma$, we use the notation 
$\gamma_{[s, t]}$ to denote the segment of $\gamma$ between $\gamma(s)$ and $\gamma(t)$.  
\begin{proof}[Proof of Lemma~\ref{lem:divd-lower}]
We begin with some preliminaries.
By the symmetry of $\Phi_{d+3}$ interchanging $b_n'$ with $b_{n}''$ and $b_{n+1}'$ with $b_{n+1}''$, respectively, we may assume,
without loss of generality, that $\alpha$ travels along a wall of type $b_n'$ and $\beta$ travels along a wall of  type~$b_{n+1}'$.  
By Lemma~\ref{lem:travel}(3), the support of $\alpha$ 
is contained in  the 
star of $b_{n}'$ in $\Phi_{d+3}$, which is equal to 
$\{a_{n-1}^*, b_{n-1}^*, b_n^*, b_{n+1}^*\}$; see Figure~\ref{fig:bnstar}.
 Similarly, the support of $\beta$ is contained in  $\{a_n^*, b_n^*, b_{n+1}^*, b_{n+2}^*\}$, which is the star of $b_{n+1}'$.

 Next, we establish some restrictions on the number and types of walls crossed by both $\alpha$ and $\beta$:

\smallskip

\noindent \emph{Claim~1.} 
Let $\mathcal K = K_1, K_2, \dots $ denote the sequence of  walls crossed by $\beta_{[0,\infty)}$,  where $K_i$ is dual to the edge between $\beta(i-1)$ and $\beta(i)$.   Then: 
\begin{enumerate}
\item[(i)]  There are at most $m$ walls in $\mathcal K$ of type $b_{n+1}^*$. 
\item[(ii)] If $\alpha$ crosses some $K_j\in  \mathcal K$ with $j > 2m$, then $K_j$ has type $b_{n+1}^*$. 
\end{enumerate}

\smallskip
\noindent
\emph{Proof of Claim~1.}
The set of types of walls in  $\mathcal K$ is the support of $\beta$, which is contained in 
 the star of $b'_{n+1}$.
Thus, Lemma~\ref{lem:geod-label}(1) implies that the label of $\beta$ contains at most $m$ instances of $b_{n+1}^*$, proving (i). 

Comparing the supports of $\alpha$ and $\beta$, we see that 
any wall crossed by both $\alpha$ and $\beta$ must be of type $b_n^*$ or $ b_{n+1}^*$.
We will now show that if $K_j \in \mathcal K$ is a wall of type $b_n^*$ crossed by $\alpha$, then 
$j \le 2m$.  Let $K_i$ be any wall in $\mathcal K$ with $i<j$. 
Then $K_i$ intersects either $K_j$ or $\alpha$ (for if it doesn't intersect $K_j$, then it separates $\alpha(0) = \beta(0)$   from $K_j$, and therefore from some vertex $\alpha(k)$ in $N^1_{K_j}$, forcing $\alpha$ to cross it). 
If $K_i$ intersects $K_j$, which is a wall of type $b_n^*$, then $K_i$ has type either $b_n^*$ or $b_{n+1}^*$ since walls of type $a_n^*$ and $b_{n+2}^*$ can't intersect a wall of  type $b_n^*$. If~$\alpha$ crosses $K_i$, then both $\alpha$ and $\beta$ cross $K_i$  and hence $K_i$ has type $b_n^*$ or $b_{n+1}^*$ again. 
It follows that the label of 
$\beta_{[0,j]}$
 is a reduced word with support in $\{b^*_n, b^*_{n+1}\}$. So, by Lemma~\ref{lem:geod-label}(2) it has length at most $2m$.  In particular, we have that $j \le 2m$, proving (ii). 
 This proves Claim~1. 

\smallskip
We now prove Lemma~\ref{lem:divd-lower} by induction on $k$.  For the base step, let $k=1$. 
Fix  
$r\ge 3^2 M > 9m$ and $r_1, r_2\ge r$. 
Fix $n$ satisfying $1 \le n \le d$ and let $\alpha$, $\beta$ and $\mathcal K$ be as above.
Claim~1 readily implies that at least $ r-3m>  r-r/3 = 2r/3$ of the walls crossed by $\beta_{[0,r]}$
do not intersect $\alpha$. 
Any path from $\alpha(r_1)$ to $\beta(r_2)$ must necessarily cross all of these walls, as each of them separates $\alpha(r_1)$ and $\beta(r_2)$. In particular, any $r$-avoidant path from 
$\alpha(r_1)$ to $\beta(r_2)$
has length at least $2r/3$.
 This proves the base step, with $C_1 = 2/3$. 

  Now assume the statement is true for $k-1$. Fix $r\ge3^{k+1} M$ and $n$ satisfying
$ k \le n \le d$. Let $\alpha$, $\beta$,  and $\mathcal K$  be as above.   
 Let $\eta$ be an arbitrary $r$-avoidant path from
 $\alpha(r_1)$ to $\beta(r_2)$, where $r_1, r_2\ge r$. 
To obtain a lower bound on the length of $\eta$, we first pass to a subsequence $\mathcal H$ of $\mathcal K$ given by the following claim:
\smallskip

\noindent \emph{Claim~2.} There exists a 
    sequence of walls $\mathcal H = H_0, \dots, H_l$  
  with the following properties:
  
  \begin{enumerate}
  \item[(i)] $\mathcal H$ is a subsequence of $K_{2m+1}, K_{2m+2}, \dots, K_{\lceil{r/2}\rceil}$; 
  \item[(ii)]  For all $i$, the type of $H_i$ is $b_n^*$ or $ b_{n+2}^*$;  
  \item[(iii)]  For all $i\ge 1$, either the types of $H_{i-1}$ and $H_i$ have different subscripts 
  or there exists $K \in \mathcal K$ of type $a_n^*$ 
  which separates $H_{i-1}$ and $H_i$;
    \item[(iv)] $\alpha$ does not cross $H_i$ for any $i$; and

  \item[(v)]$ l \ge r/(4M)$.
  \end{enumerate}

\noindent\emph{Proof of Claim~2.}
We first observe that any subsegment of $\beta$ of length $2m+1$
has at least one  instance of $b_n^*$ or $b_{n+2}^*$ (for if not, then its label would be a reduced word  of length $ 2m+1$ with support in $\{a_n^*, b_{n+1}^*\}$, violating Lemma~\ref{lem:geod-label}(2)).    By similarly applying 
Lemma~\ref{lem:geod-label}(2) to the pairs $\{b_n^*, b_{n+1}^*\}$ and $\{b_{n+1}^*, b_{n+2}^*\}$, respectively, we  conclude that any subsegment of $\beta$ of length $2m+1$ contains at least one instance of $a_n^*$ or $b_{n+2}^*$, and at least one instance of $a_n^*$ or $b_{n}^*$.

The sequence $\mathcal H$ is chosen as follows.
From the previous paragraph, we know that 
 $\beta_{[2m, 4m+1]}$ is dual to some $K \in \mathcal K$ of type $b_n^*$ or $b_{n+2}^*$, and we set $H_0 = K$.  Now for $j>1$, suppose $H_{j-1}$ has been chosen to be some $K_i\in \mathcal K$, say  of type $b_n^*$. 
 (The $b_{n+2}^*$ case is analogous.)   Then $\beta_{[i-1, i+2m]}$ crosses  some $K_{i'} \in \mathcal K$ of type $b_{n+2}^*$ or $a_n^*$ (from the previous paragraph).  In the former case, we set $H_j=K_{i'}$ and note that the types of $H_{j-1}$ and $H_j$ have different subscripts.  Otherwise, if  $K_{i'}$ has type $a_n^*$, we know that  $\beta_{[i'-1, i'+2m]}$ crosses some $K_{i''}$ of type $b_n^*$ or $ b_{n+2}^*$, and we set $H_j = K_{i''}$. In this case $K_{i'}$ separates $H_{j-1}$ and $H_j$, since walls of type $a_n^*$ cannot intersect walls of type $b_n^*$ or $b_{n+2}^*$.   We continue until we have exhausted all the walls in $\mathcal K$ up until $K_{\lceil{r/2}\rceil}$.

Note that properties (i)--(iii) are obvious from this selection process, and (iv) follows easily from 
(i), (ii) and Claim~1(ii).  Moreover, the selection process guarantees that every subsegment of $\beta_{[0, \lceil{r/2}\rceil ]}$ of length $M = 4m+1$ crosses at least one wall in $\mathcal H$.  Thus
$|\mathcal H| \ge \lceil{r/2}\rceil / {M}$, and since $r \ge 3^{k+1}M\ge 4M,$
\[
l = |\mathcal H|  -1 \ge \frac{\lceil{r/2}\rceil } {M} -1 \ge \frac{r/2 }{M} -1 =\frac{r-2M}{2M}\ge \frac{r}{4M}.
\]
This proves (v) and completes the proof of Claim~2.

\smallskip

As in the base step,
Claim~2(iv) implies that $\eta$ crosses each wall in $\mathcal H$. 
Our next goal is to obtain a lower bound on the length of the part of $\eta$ between $H_{i-1}$ and $H_i$, for each $i$; see Figure~\ref{fig:lower-duplex}.  We do this by applying the inductive hypothesis to a pair of geodesics $\mu_i'$ and $\nu_i$, which we now define.

\begin{figure}
\begin{overpic}[scale=1.3,unit=1mm]%
{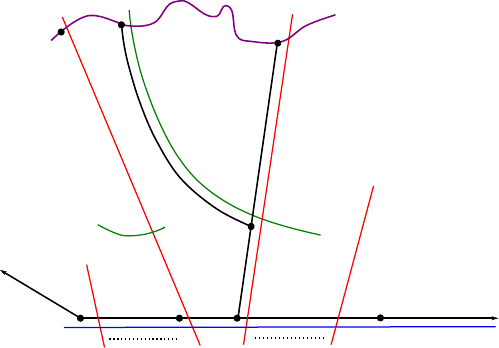}
\put(19,-2.5){\tiny\textcolor{red}{$H_0$}}
\put(38,-2){\tiny\textcolor{red}{$H_{i-1}$}}
\put(48,-2){\tiny\textcolor{red}{$H_i$}}
\put(65,-2){\tiny\textcolor{red}{$H_l$}}
\put(32,34){\small $\nu_i$}
\put(16,7){\small $e$}
\put(45,15){\small $\mu_i$}
\put(49,40){\small $\mu_i'$}
\put(95, 7){\small $\beta$}
\put(2,11){\small $\alpha$}
\put(74,8){\tiny $\beta\left(\lceil \frac{r}{2} \rceil\right)$}
\put(30,7.5){\small $g_{i-1}$}
\put(45,7.5){\small $g_{i}$}
\put(6,65){\small $h_{i-1}$}
\put(22,67){\small $q_i$}
\put(54,63){\small $h_i$}
\put(47,22){\small $p_i$}
\put(40,69){\small \textcolor{Purple}{$\eta$}}
\put(17,26){\tiny \textcolor{OliveGreen}{$L_{i-1}$}}
\put(65,21.5){\tiny \textcolor{OliveGreen}{$L_{i}$}}
\end{overpic}
\caption{\small{Constructions used to obtain a lower bound on the length of the part of $\eta$ between $H_{i-1}$ and $H_i$.}}
\label{fig:lower-duplex}
\end{figure}
 
 For $0 \le i \le l $, let $g_i$ be the first vertex of $\beta$ which is in $N_{H_i}^1$. In particular, $g_i$ is in the same component of the complement of $H_i$ as  $p=\beta(0)$. 
 Let $h_i$ be the first vertex of $\eta$ which is in $N_{H_i}^1$. 
 In particular, $h_i$ is in the same component of the complement of $H_i$ as $g_i$. Then by  
 Lemma~\ref{lem:travel}(2) 
 there is a geodesic $\mu_i$ which travels along $H_i$ and connects $g_i$ to $h_i$, with $\mu_i(0) = g_i$.

\medskip
\noindent \emph{Claim~3.} 
For each $i \ge 1$, there is a wall $L_i$ 
such that the following hold. See Figure~\ref{fig:lower-duplex}.
\begin{enumerate}
\item[(i)]
$L_i$ is dual to some edge in ${\mu_i}_{[0, M]}$.
\item[(ii)]  $L_i$ is of type $b_x^*$, with $x = (\text{the subscript of the type of } H_i) \pm 1$.
\item[(iii)]
$L_i$  separates $h_i$ and $h_{i-1}$.
\end{enumerate}

\noindent
\emph{Proof of Claim~3:} 
Let $i \ge 1$. 
Suppose $H_i$ has type $b_{n}^*$.  (By Claim~2(ii), the only other possibility is 
$b_{n+2}^*$, and the proof of that case is analogous.)  The possible types of walls which intersect $H_i$ are $b_{n-1}^*, b_{n}^*, b_{n+1}^*$ and $a_{n-1}^*$.  Applying 
Lemma~\ref{lem:geod-label}(2) 
as in the proof of Claim~2, we know that ${\mu_i}_{[0, 2m+1] }$ crosses a wall $L$ of type $b_{n-1}^*$ or $a_{n-1}^*$. In the former case, we set $L_i = L$.  Otherwise, if $L$ is of type $a_{n-1}^*$, we know that the segment of $\mu_i$ of length $2m+1$ starting at and including its edge dual to $L$ crosses a wall $L'$ of type $b_{n+1}^*$ or 
$b_{n-1}^*$ (using Lemma~\ref{lem:geod-label}(2) again), and we set $L_i=L'$. Now (i) and (ii) are clear from the construction, as $M = 4m+1$.

To prove (iii), 
we start by showing that
$\mu_{i-1}$ and $\beta$ do not cross $L$, regardless of
whether $L$ is of type $b_{n-1}^*$ or $a_{n-1}^*$.
Firstly, $\beta$ cannot cross $L$, since  $b_{n-1}^*$ and  $a_{n-1}^*$ are not in its support.

To show that $\mu_{i-1}$ does not cross $L$ we consider two cases, depending on whether $H_{i-1}$ has type $b_{n+2}^*$ or $b_n^*$. 
In the former case, we see by comparing types that $H_{i-1}$ and $H_i$ do not intersect, and that $H_{i-1}$ and $L$ do not intersect either. By Lemma~\ref{lem:travel}(1), the point $g_{i-1}$ precedes (along the geodesic $\beta$) the point of intersection of $\beta$ with the wall $H_{i-1}$, and hence $H_{i-1}$ separates $g_{i-1}$ from $H_i$. Since $H_i$ and $L$ intersect, we conclude that $H_{i-1}$ separates $g_{i-1}$ from $L$, and hence, since $\mu_{i-1}$ does not cross $H_{i-1}$ by definition, $H_{i-1}$ separates $\mu_{i-1}$ from $L$.

Now we consider the case when both $H_{i-1}$ and $H_i$ are of type $b_n^*$. 
Then by Claim~2(iii) there exists a wall $K\in\cK$ of type $a_n^*$ which separates $H_{i-1}$ and $H_{i}$. Let $\Sigma_d\setminus K = A\sqcup B$ be the disjoint union of open 
half-spaces such that $H_{i-1}\subset A$, $H_i\subset B$. Since $H_{i-1}$ and $K$ do not intersect (having nonadjacent types in $\Phi_{d+3}$), $H_{i-1}$ cuts $A$ into a disjoint 
union of two subsets $C$ and $D$ so that $\Sigma_d \setminus (H_{i-1}\cup K) = C\sqcup D\sqcup B$. If $C$ is the component which contains $p=\beta(0)$, then we claim that $\mu_{i-1}$ lies entirely in $C$ whereas $L$ lies entirely in $B$. Indeed, let $t_{g_{i-1}}$, $t_{H_{i-1}}$, $t_{K}$ and $t_{H_i}$ be the unique real numbers such that 
$\beta(t_\diamond)=\beta\cap \diamond$, for $\diamond$ being one of $g_{i-1}$, $H_{i-1}$, $K$, $H_i$. (These numbers are unique 
since a geodesic can cross a wall no more than 
once.) By the construction of the $H_i$'s, we have 
$t_{H_{i-1}}<t_K<t_{H_i}$, and we have $t_{g_{i-1}}<t_{H_{i-1}}$ by Lemma~\ref{lem:travel}(1).
Hence $C$ is the component which contains $\beta_{[0,t_{H_{i-1}})}$, and $B$ is the component that contains $\beta_{(t_K,\infty)}$. Since $g_{i-1}\in\beta_{[0,t_{H_{i-1}})}$ and $\mu_{i-1}$ does not cross $H_{i-1}$ (as it travels along $H_{i-1}$), we see that $\mu_{i-1}$ lies entirely in $C$. On the other hand, $H_i$ and $L$ do intersect, whereas $K$ and $L$ do not (again by comparing types). Hence $L$ lies in the component which contains $\beta(t_{H_i})$, i.e.\ in $B$. Since $C$ and $B$ are disjoint, we conclude that $\mu_{i-1}$ does not cross $L$.

Now observe that $L$ separates $h_i$ and $g_i$, since the geodesic $\mu_i$ between $h_i$ and $g_i$ crosses $L$ exactly once.  
Let $\xi$ be the path from $h_{i-1}$ to $g_i$ consisting of $\mu_{i-1}$ and the segment of $\beta$ from $g_{i-1}$ to $g_i$. 
Since $\mu_{i-1}$ and $\beta$ don't cross $L$, neither does $\xi$.
 Thus $L$ separates $h_i$ from $\xi$, and therefore from $h_{i-1}$.  This proves (iii) in the case that $L=L_i$. 

Otherwise $L$ is of type $a_{n-1}^*$ and $L_i$ is of type $b_{n+1}^*$ or 
$b_{n-1}^*$. By comparing types we observe that $L$ and $L_i$ cannot intersect.  By construction, $L_i$ is in the same component of the complement of $L$ as $h_i$, while, as shown in the previous paragraph, $\xi$ is in the other component. It follows that $\xi$ cannot cross $L_{i}$, so its endpoints $h_{i-1}$ and $g_i$ are in a single component of the complement of $L_{i}$.  On the other hand, since $\mu_i$ 
 crosses $L_i$ exactly once, $L_i$ separates $h_i$ and $g_i$.  Consequently $L_i$ separates $h_i$ and $h_{i-1}$. This completes the proof of Claim~3.

\medskip

By Claim~(3)(iii), $L_i$ separates the points $h_i$ and $h_{i-1}$ on $\eta$, so $\eta$ must cross $L_i$ for all $i \ge 1$.  
Let $p_i$ be the first vertex of $\mu_i$ in $N_{L_i}^1$
and let $q_i$ be the first vertex of $\eta$ which is in $N_{L_i}^1$.  Then $p_i$ and $q_i$ are in 
 the same component of the complement of $L_i$.   By Lemma~\ref{lem:travel}(2), there is a geodesic $\nu_i$ which travels along $L_i$ from $p_i$ to $q_i$.  
Let $\mu_i'$ 
denote the geodesic ray $\mu_i$ reparametrised so that  $\mu_i'(0) = p_i$ (see Figure~\ref{fig:lower-duplex}).

We now show that the inductive hypothesis can be applied to the pair $\mu_i', \nu_i$.  
By construction, $\mu_i'(0) = \nu_i(0)$, and $\mu'_i$ (respectively, $\nu_i$) travels along $H_i$ (respectively, $L_i$).  Combining  Claim~2(ii) and 
Claim~3(iii), we conclude that either $H_i$ has type $b_n^*$ and $L_i$ has type $b_{n \pm 1}^*$, or~$H_i$ has type $b_{n+2}^*$ and $L_i$ has type $b_{n+1}^*$ or $b_{n+3}^*$.   Thus, conditions (1)--(3) in the statement of Lemma~\ref{lem:divd-lower} are satisfied, with the role of $\alpha$  being played by whichever of $\mu_i'$ and $\nu_i$ travels along the wall whose type has a lower subscript.  Since the lowest of these subscripts is $n-1$, and since we had chosen $n \ge k$, we know that the subscripts of the types of $\mu_i'$ and $\nu_i'$ are both $\ge k-1$, and so  the inductive hypothesis can be applied to $\mu_i'$ and $\nu_i$.

Finally, we must verify that $\mu_i'$ and $\nu_i$ are long enough (specifically, of length at least $3^{k}M$) for the conclusion of the inductive hypothesis to hold.  
The length of $\mu_i'$ is $d(p_i, h_i)$.
Recall that $h_i$ is a point outside the ball of radius $r$ centred at $\beta(0)$.  
By Claim~2(i) and the fact that $g_i$ is on the same side of $H_i$ as $\beta(0)$ for each $i$, we have $d(\beta(0), g_i)\le \lceil r/2\rceil -1 <r/2 $, and by Claim~3(i) 
we have $d(g_i, p_i) \le M$. 
Combining these facts, we get: 
\[
r \le d(\beta(0), h_i) \le d(\beta(0), g_i) + d(g_i, p_i) +  d(p_i, h_i) \le r/2+M + d(p_i, h_i).
\]
Furthermore, since $r \ge 3^{k+1} M$, it is certainly true that
$M < r/6$.  From this, 
we get, $d(p_i, h_i) \ge r-r/2-M \ge r/2-r/6 > r/3$.  Similarly, $d(p_i, q_i) > r/3$. 
Thus, $\mu_i'$ and $\nu_i$ both have length at least $r/3 \ge 3^{k}M$ as needed, so 
the conclusion of the inductive hypothesis holds.  Specifically, we get, for all $1 \le i \le l$: \[
d^{\mathrm{av}}_{p_i, r}(q_i, h_i) \ge C_{k-1} \left(\frac r3\right)^{k-1}.
\]
Finally, since $l \ge r/(4M)$ by Claim~2(v), we have that the length of $\eta$ is at least 
\[
\sum_{i=1}^l d^{\mathrm{av}}_{p_i, r}(q_i, h_i) \ge \frac{r}{4M} C_{k-1} \left(\frac r3\right)^{k-1}
= \frac{C_{k-1}}{4\cdot 3^{k-1}M}\, r^k.
\]
This proves the inductive step, with $C_k = C_{k-1}/(4\cdot 3^{k-1}M)$.
\end{proof}

\section{Hypergraph index and the first Betti number}\label{sec:Betti}

We prove Theorem~\ref{thm:bettiIntro} and Corollary~\ref{cor:spectrumIntro} in Section~\ref{sec:bound}.  Theorem~\ref{thm:bettiIntro} provides an upper bound on hypergraph index in terms of the first Betti number of the Dynkin diagram.  We then discuss some examples and prove Corollary~\ref{cor:DynkinTreeDivIntro} in Section~\ref{sec:examplesBetti}, and state some related open questions in Section~\ref{sec:questionsBetti}.

\subsection{Proof of Theorem~\ref{thm:bettiIntro}}\label{sec:bound}

Let $(W,S)$ be a Coxeter system with Dynkin diagram $\Delta= \Delta(W,S)$ and finite hypergraph index $h = h(W,S)$.  Our goal is to show that $h \leq \betti \Delta +~1$. We first introduce some terminology and prove a few lemmas; Corollary~\ref{cor:spectrumIntro} will be a consequence of one of these results.  We then consider the case where $\Delta$ is a tree in Proposition~\ref{prop:DynkinTree}, and finally prove Theorem~\ref{thm:bettiIntro} by induction.  

In what follows, we write $\Lambda_i=\Lambda_i(S)$, and say that $T \subseteq S$ is \emph{at level $i$} if $T \in \Lambda_i$.  Note that $T \subseteq S$ may be at more than one level.
Now recall the directed forest with vertex set~$\cup_i \Lambda_i$ described above the statement of Lemma~\ref{lem:monotone}.  Let $\cC$ denote an equivalence class in $\Lambda_{h-1}$ such that $S = \cup_{T \in \cC} T$.  Then given $T \in  \Lambda_i$, with $i <h$, it is not necessarily true that $T$ equals or feeds into an element of $\cC$, even though $T \subseteq S$.  On the other hand, there is a unique element of $\Lambda_{h+1}$, namely $S$, and 
every subset of $S$ at level less than $h+1$ feeds into it. 
Thus when $h$ is finite, this directed forest is actually a tree. 

Next, consider the set \[ \{S\} \cup \cC \cup \{T \mid T \mbox{ feeds into an element of } \cC \}.\]
We call the subtree  induced by this set of vertices 
the \emph{principal subtree for $S$ associated to $\cC$}.  Note that $\cC$, and hence principal subtrees, need not be unique.  
We will often fix a principal subtree, and denote its set of level $i$ vertices by $\widetilde \Lambda_i(S)$.
Note that $T \in \widetilde \Lambda_i(S)$ can be thought of as a vertex of this tree or a subset of $S$ and, by a slight abuse of notation, we will go back and forth between these interpretations without comment.

In order to use an inductive argument to prove Theorem~\ref{thm:bettiIntro}, we will need to find a subsystem of
$(W,S)$ with hypergraph index exactly $h-1$.  This is provided by the following lemma.

\begin{lemma}\label{lem:hi=h-1}
Let $(W, S)$ be a Coxeter system with finite hypergraph index $h \geq 1$.  Then there exists 
$T \subset S$ such that $h(W_T, T) = h-1$. 
\end{lemma}

\begin{proof}
Fix a principal subtree for $S$, associated to an equivalence class $\cC$, and denote its level~$i$ vertices by $\widetilde \Lambda_i(S)$.  Let 
$T_1, \dots, T_k \subset S$ be the elements of  $\widetilde \Lambda_{h-1}(S)$; these are precisely the elements of the equivalence class $\cC$ of $\Lambda_{h-1}(S)$, and these all  feed into the unique element~$S$ of $\widetilde \Lambda_h(S)$.

We first show that at least one of $T_1, \dots, T_k$  is not contained in any set at a level lower than $h-1$.
More precisely, we
claim that there exists $j$ such that if $T_j \subseteq U \in \Lambda_i(S)$ for some $i$, then $i\ge h-1$.  
(Observe that in this statement, $U$ need not be in our fixed principal subtree.)
To prove the claim, suppose instead that this fails for each of $T_1, \dots, T_k$.
 It follows that there is an index $l < h-1$, and (not necessarily distinct) elements
$U_1, \dots, U_k \in \Lambda_l(S)$ such that $T_i \subseteq U_i$.  Since $T_1, \dots, T_k$ 
form a complete $\equiv_{h-1}$-equivalence class in $\Lambda_{h-1}$, 
it follows that $U_1, \dots, U_k$ are all $\equiv_l$-equivalent.  Hence there is an element of $\Lambda_{l+1}(S)$ that contains $\cup_j U_j$, which, by construction, is all of $S$.  Since $l+1<h$, this contradicts the fact that the hypergraph index of $(W, S)$ is $h$.  
This proves the claim.  

Now fix $T=T_j$ which satisfies the claim.  We wish to show that $h(W_T, T) \le h-1$.
It is immediate from Definitions~\ref{defn:wide} and~\ref{defn:slab} that any element of $\widetilde \Lambda_0(S)$ which feeds into $T\in \widetilde \Lambda_{h-1}(S)$ is contained in an element of $\Lambda_0(T)$.  For $0<i\le h-1$, each element $U\in \widetilde\Lambda_i(S)$ which feeds into $T$ is a union of elements of some equivalence class in $\widetilde\Lambda_i(S)$, consisting of elements which also feed into $T$. Hence it follows by induction that every such element is contained in some element of $\Lambda_i (T)$.  In particular, $T$ belongs to $\Lambda_{h-1}(T)$.  This implies that 
$h(W_T, T) \le h-1$.

Now suppose $h(W_T, T) < h-1$.
 Then there exists some  $n< h-1$ such that $\Lambda_n(T)$ contains~$T$.  By Lemma~\ref{lem:monotone}(2a), it follows that there is an element of $\Lambda_n(S)$ containing $T$.  This contradicts the fact that $T=T_j \in \Lambda_{h-1}$ satisfies the claim above.  Thus $h(W_{T}, T)=h-1$.
\end{proof}

Corollary~\ref{cor:spectrumIntro} of the introduction follows from Lemma~\ref{lem:hi=h-1}.

The next lemma shows that in most cases, attaching finite trees to a Dynkin diagram along finitely many vertices does not increase the hypergraph index.

\begin{lemma}\label{lem:add_trees}
Let $\Delta$ be a Dynkin diagram with associated Coxeter system $(W, S)$ such that $h(W, S)< \infty$.
Suppose there are  subsets $T, T_1, \dots, T_k \subset S$, which induce diagrams $\D_T, \D_{T_1}, \dots, 
\D_{T_k}$, 
such that the following hold:
\begin{enumerate}
\item The sets $T, T_1, \dots, T_k$ are pairwise disjoint and their union is $S$. 
\item 
Each $\D_{T_j}$ is a tree.
\item The diagram $\D$ is obtained by taking the disjoint union of $\D_T, \D_{T_1}, \dots, 
\D_{T_k}$ and adding, for each $j$, a single edge connecting a vertex of $\D_{T_j}$ to a vertex $v_j \in \D_T$.
(We do not require that $v_j \neq v_{j'}$ for $j \neq j'$.)
\end{enumerate}
If $h(W_T, T) \neq 0, \infty$, then $h(W, S) \le h(W_T, T)$. 
\end{lemma}

\begin{proof}
Let $h_T=h(W_T, T) $. 
Fix a principal  subtree associated to an equivalence class in $\Lambda_{h_T - 1}(T)$ whose union is $T$, and for $0\le i \le h_T$, let $\widetilde \Lambda_i(T)$  denote the 
set of level $i$ vertices of this subtree. 
Now if every element of $\widetilde \Lambda_{0}(T)$ is of the form 
$A \times K$, where $A$ is minimal nonspherical and $K$ is spherical, then  Lemma~\ref{lem:intersection} implies that $\widetilde \Lambda_0(T) =
\widetilde \Lambda_1(T)=\cdots =\widetilde \Lambda_{h_T}(T)$.
This means that $\widetilde\Lambda_0(T)$ contains $T$, hence $h_T =0$, which contradicts the hypothesis of the lemma.  Thus $\widetilde \Lambda_0(T)$
has at least one set of the form $A \times B$, where $A$ and $B$ are both nonspherical.  

Now consider $T_j$, for any $1 \le j \le k$.  By hypothesis, $v_j$ is the only vertex in $\D$ which is adjacent to $\D_{T_j}$.  
At least one of the nonspherical sets $A$ and $B$ from the previous paragraph,  say $A$, does not contain $v_j$. 
Thus there is a set $A \times (B \cup T_j) \subset S$, with both factors nonspherical.  So there exists $U_j \in  \Lambda_0 (S)$ with 
$A \times B \subset A \times (B \cup T_j) \subseteq U_j$.

Recall that there is a unique element in $\widetilde\Lambda_{h_T}(T)$, namely $T$.  
Let $T' \in \Lambda_{h_T}(S)$ denote a corresponding element provided by Lemma~\ref{lem:monotone}(2).   Thus $T \subseteq T'$ and given any element $X \in \widetilde \Lambda_i(T)$, and any $X'\in \Lambda_i(S)$ containing $X$, it follows that $X'$ feeds into $T'$.  
In particular, $U_j$ feeds into $T'$ for each $j$, since $U_j$ is an element of $\Lambda_0(S)$ containing $A \times B \in \widetilde\Lambda_0(T)$.
Consequently $T'$ 
contains $T\cup \bigl(\bigcup_{1 \le j \le k} U_j\bigr) =S$.  It follows that $h(W, S) \le h_T$. 
\end{proof}

We now consider the situation where $\Delta$ is a tree.  The next result  will be used to establish the base case in our inductive proof of Theorem~\ref{thm:bettiIntro} below.

\begin{prop}\label{prop:DynkinTree}  Suppose the Dynkin diagram $\Delta$ of $(W,S)$ is a tree.  Then $(W,S)$ has hypergraph index $0$, $1$ or $\infty$ only.  
\end{prop}

\begin{proof}
Let $h$ denote the hypergraph index of $(W,S)$.  We may assume $\Lambda_0$ is nonempty and does not contain $S$, for otherwise $h=\infty$ or $h=0$ by Definition~\ref{defn:hypergraphIndex}.   Recall that every element of $\Lambda_0$ is of one of the following types:

\begin{enumerate}
\item $A \times B \in \Omega(S)$, where $A$ and $B$ are both nonspherical.
\item $A \times K \in \Omega(S) \cup \Psi(S)$, with $A$ minimal nonspherical and $K$ spherical.  
\end{enumerate}
If every element of $\Lambda_0$ is of type (2), then Lemma~\ref{lem:intersection} implies that each 
$\equiv_0$-equivalence class contains a single element. Hence by induction $\Lambda_h = \Lambda_0$ for all $h \geq 0$, and $h=\infty$, since $S \notin \Lambda_0$ by assumption. 
Thus we may assume that there is at least one element in $\Omega(S)$ of type (1). We now describe the sets of this type. 

Given $s\in S$, define $T_s
= S \setminus \{ s \}$
and let $\Delta_s$ be the subgraph of $\Delta$ induced by $T_s$. Note that the special subgroups generated by the vertex sets of the components of $\Delta_s$ pairwise commute.  Thus if $\Delta_s$ has at least two nonspherical components, then $T_s$ is a set  of type~(1), as the maximality condition in the definition of $\Omega(S)$  is clearly satisfied.   Conversely, if $A \times B\in \Omega(S)$ is a set of type~(1), then there exists $s \in S \setminus A \times B$ (as $S \notin \Lambda_0$).  
Then by maximality of elements of $\Omega(S)$, we have $A \times B = T_s$
 and it follows that $\Delta_s$ has at least two nonspherical components. 

Suppose there are at least two sets of type (1) in $\Lambda_0$.  Then 
 there are distinct $s, t\in S$ such that both $\Delta_s$ and $\Delta_t$ have at least two nonspherical components.  Write $\Delta_s(t)$ for the component of $\Delta_s$ which contains $t$.  Then as $\Delta$ is a tree, all components but one of $\Delta_t$ are contained in $\Delta_s(t)$, hence $T_s \cap T_t$ is nonspherical.  Since $T_s \cup T_t = S$, it follows that  $h=1$.
 
 Finally, assume that there is a unique set $T_s$ of type (1) in $\Lambda_0$. 
Then every $T \neq T_s$ in $\Lambda_0$ is of type (2), and 
contains $s$ (for otherwise $T\subset T_s$, contradicting either maximality in Definition~\ref{defn:wide} or condition (3) in Definition~\ref{defn:slab}).   If there exists such a $T$ with $T_s\cap T$ nonspherical, then $S = T_s \cup T$, and so $h= 1$.  Otherwise, if $T_s\cap T$ is spherical for every $T \neq T_s$, then 
$T_s$ is the unique element of its $\equiv_0$-equivalence class, as is each $T\neq T_s$, by Lemma~\ref{lem:intersection}.  Hence by induction $\Lambda_h = \Lambda_0$ for all $h \geq 0$, and $h=\infty$ in this last case. 
 \end{proof}

We are now ready to prove Theorem~\ref{thm:bettiIntro}.

\begin{proof}[Proof of Theorem~\ref{thm:bettiIntro}]
The proof is by induction on the first Betti number. 
The base case, where $\betti{\D}=0$ and so $\Delta$  is a forest, follows easily by Proposition~\ref{prop:DynkinTree} and Lemmas~\ref{lem:finite} and \ref{lem:reducible1} and Proposition~\ref{prop:reducible2}.  For the inductive step, we assume the statement is true for all Dynkin diagrams with first Betti number less than $b$, and let $\D$ be a diagram with $\betti{\Delta} = b\ge 1$, such that $\D$ defines a Coxeter system $(W, S)$.  Let $h = h(W, S)$.  Observe that if $h \le 2$, then the statement $h\le b+1$ is automatically true.  So we assume from now on that $h >2$. 
By hypothesis, $h<\infty$. 

\emph{Case 1. $\D$ is connected.}
By Lemma~\ref{lem:hi=h-1}, there exists a set $T \subset S$ such that $h(W_T, T) = h-1$.  Let $\Delta_T \subset \Delta$ be the Dynkin diagram of $(W_T, T)$.  
If $\betti{\D_T} < \betti \D$, then we can apply the induction hypothesis to conclude that 
$h-1 = h(W_T, T) \le \betti{\D_T} +1 \le \betti{\D}$, so that $h \le b+1$ as required.  
Thus, we may assume that $\betti{\D_T} = \betti{\D}=b$.  

First suppose $\D_T$ is not connected, so that $(W_T, T)$ is reducible.  Since we are assuming $h > 2$, which implies $h(W_T, T)>1$, Lemma~\ref{lem:finite} implies that we can write 
$(W_T, T ) = (W_{T_1}, T_1 ) \times (W_{T_2}, T_2)$, where  $(W_{T_1}, T_1 )$ is nonspherical and irreducible, and $(W_{T_2}, T_2)$ is spherical.  Moreover, $h(W_T, T) = h(W_{T_1}, T_1) =h-1$. 
Let $\Delta_{T_1}$ be the subdiagram of $\D$ induced by $T_1$.  If $\betti{\D_{T_1}} <b$, then we may apply the induction hypothesis to conclude that $h(W_{T_1}, T_1) \le \betti{\D_{T_1}}+1$, which 
implies $h \le b+1$.  Thus we may assume that $\betti{\D_{T_1}} =b$. 

Now define $U\subset S$ as follows.  Put 
$U=T$ if $\D_T$ is connected, and put $U = T_1$  (from the previous paragraph) if $\D_T$ is not connected.  Let $\D_U$ be the diagram induced by $U$.  Then 
by our assumptions, $\betti{\D_U} =\betti{\D}$, $h(W_U, U) = h-1$ and $\D_U$ is connected.  
Consider the graph obtained from $\D$ by deleting the interiors of all edges of $\D$ which have exactly one vertex in $\D_U$, and let $\D_U, \D_1, \D_2, \dots, \D_k$ be its components. 
Now $\betti{\D_U} + \Sigma_i \betti{\D_i} \le \betti{\D}$, since we have deleted edges.  
Since $\betti{\D_U} =\betti{\D}$, it follows that $\betti {\D_j}=0$ for all $1 \leq j \leq k$, and therefore, $\D_j$ is a tree.   Furthermore, there is exactly one edge in $\D$ which is incident to both $\D_U$ and $\D_j$.  If there were two such edges, then there would be a cycle in $\D$ which is not entirely contained  in $\D_U$.  It follows easily from the definition of the first Betti number that in this case  $\betti{\D_U} < \betti{\D}$, which contradicts our assumption.  We have established that $\D$ has the structure described in the statement of Lemma~\ref{lem:add_trees}, given by $\D_U, \D_1, \D_2, \dots, \D_k$.  By the conclusion of that lemma, $h=h(W, S) \le h(W_U, U) = h-1$, a contradiction.

\emph{Case 2. $\D$ is not connected.}  In this case $(W,S)$ is reducible, and since it is nonspherical, it has at least one nonspherical  factor.  As we have assumed $h>2$, Lemma~\ref{lem:finite} implies that
$(W, S) = (W_1, S_1)\times (W_2, S_2)$, with the first factor 
nonspherical and irreducible and the second factor spherical, and furthermore, 
$h=h(W_1, S_1)$.
Now the Dynkin diagram $\D_1$ of $(W_1, S_1)$ is connected, and $\betti{\D_1} \le  \betti{\Delta}=b$.  
 Since we have already established the conclusion of the lemma for connected Dynkin diagrams   with first Betti number at most $b$, we conclude  that 
  $h(W_1, S_1) \le \betti{\D_1} +1$.
 Hence $h \le b+1$ as required. 
 
 This completes the proof of the inductive step. 
\end{proof}

\subsection{Examples}\label{sec:examplesBetti}

Let $(W,S)$ be a Coxeter system with Dynkin diagram $\Delta$ and hypergraph index~$h$.
In this section we consider some examples where $\Delta$ has small first Betti number, with a view to determining which values of $h$ are actually realised.  We also prove Corollary~\ref{cor:DynkinTreeDivIntro}.

\subsubsection{Dynkin diagram a tree}\label{sec:DynkinTree}

In this case, $\betti \Delta = 0$, and by Proposition~\ref{prop:DynkinTree}, we have $h \in \{0,1,\infty\}$.  We provide examples to show that these values of $h$ are all  realised, and hence prove Corollary~\ref{cor:DynkinTreeDivIntro}.

We recall an infinite family of examples from the introduction to~\cite{caprace}. Let $\Delta_n$ be the labelled graph with vertex set $S_n = \{s_1,\dots,s_n\}$ such that $m_{ij} = 4$ if $|i-j| = 1$ and $m_{ij} = 2$ if $|i-j| \geq 2$.  Thus $\Delta_n$ is the path on $n$ vertices (hence a tree), with all edge labels~$4$.  See Figure~\ref{fig:8path} for the case $n = 8$.

\bigskip

\begin{figure}[!ht]
\centering
\begin{tikzpicture}[double distance=3pt, thick,scale=0.8]
\begin{scope}[xshift=-1.5cm]
\foreach \x in {1,2,...,8}
  \fill (-6.75+1.5*\x,0) circle[radius=2.5pt] node[shift={(0,-0.4)}] {{$s_\x$}};
 \draw (-5.25,0)--(5.25,0) ; 
 \foreach \x in {1,2,...,7}
 \draw (-6+1.5*\x,0.3) node {$4$};

\end{scope}

\end{tikzpicture}
    \caption{\small{Dynkin diagram $\Delta_8$.}}
    \label{fig:8path}
\end{figure}
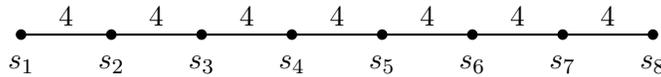
Write $W_n$ for the Coxeter group with Dynkin diagram $\Delta_n$.  First observe that the Coxeter system $(W_3,S_3)$ is irreducible affine of type $\widetilde C_2$,  so $(W_3,S_3)$ has hypergraph index $0$.  Then, as noticed in~\cite{caprace}, by Theorem~A of~\cite{caprace} (corrected to Theorem~A$'$ of~\cite{caprace-erratum}), for $4 \leq n \leq 7$ the group $W_n$ is relatively hyperbolic, but for $n \geq 8$ the group $W_n$ is not relatively hyperbolic with respect to any family of special subgroups.  Thus for $4 \leq n \leq 7$ the Coxeter system $(W_n,S_n)$ has hypergraph index $\infty$, by Theorem~\ref{thm:hyp_thick_div_Intro}(4) (or by direct computation).  Finally, for $n \geq 8$, it is not difficult to check that $(W_n,S_n)$ has hypergraph index $1$.  Thus this family of examples exhibits hypergraph indexes $0$, $1$ and $\infty$.  We note that by Theorem~8.7.2 of~\cite{davis-book}, each of the groups $W_n$ is $1$-ended.

Corollary~\ref{cor:DynkinTreeDivIntro} of the introduction can now be obtained by combining  Proposition~\ref{prop:DynkinTree} with the above family of examples and Parts~(1), (2) and (4) of Theorem~\ref{thm:hyp_thick_div_Intro}, for hypergraph indexes $0$, $1$ and $\infty$ respectively.

\subsubsection{Dynkin diagram a cycle}\label{sec:DynkinCycle}

In this case, $\betti \Delta = 1$, and so by Theorem~\ref{thm:bettiIntro}, we have $h \in \{0,1,2,\infty\}$.  We now show that each of these possibilities is realised.  

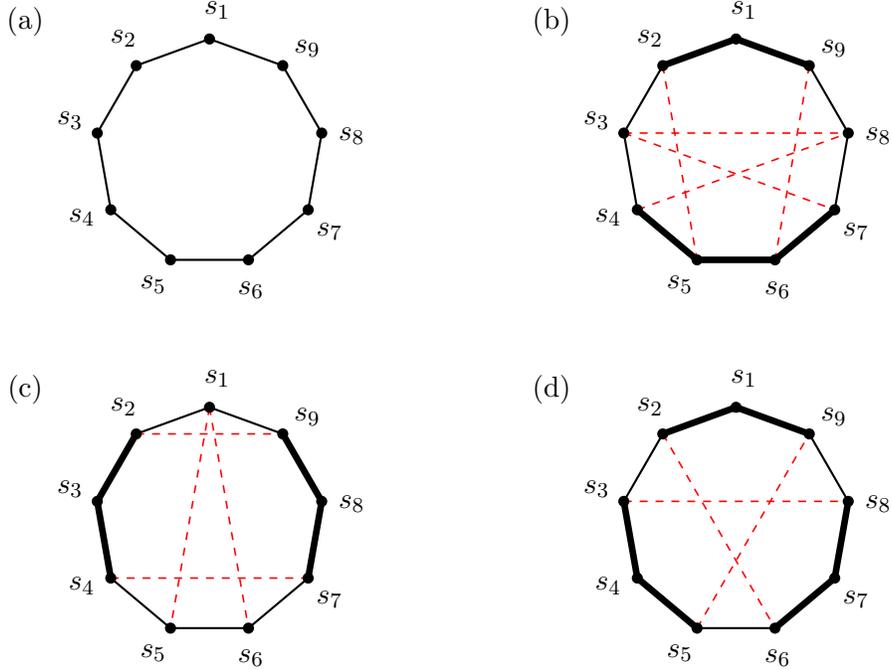
\begin{figure}
\centering
\begin{tikzpicture}[double distance=3pt, thick, scale=0.7]
\begin{scope}[xshift=10cm]
\node[draw=none,minimum size=3cm,regular polygon,regular polygon sides=9] (a) {};

\draw[line width=2.5pt] (a.corner 1)--(a.corner 2); %\draw (a.corner 1) node [shift={(1*360/9+130:0.6)}] {$4$};
\draw (a.corner 2)--(a.corner 3);
\draw (a.corner 3)--(a.corner 4);
\draw[line width=2.5pt] (a.corner 4)--(a.corner 5);%\draw (a.corner 4) node [shift={(4*360/9+130:0.6)}] {$4$};
\draw[line width=2.5pt] (a.corner 5)--(a.corner 6);%\draw (a.corner 5) node [shift={(5*360/9+130:0.6)}] {$4$};
\draw[line width=2.5pt] (a.corner 6)--(a.corner 7);%\draw (a.corner 6) node [shift={(6*360/9+130:0.6)}] {$4$};
\draw (a.corner 7)--(a.corner 8);
\draw (a.corner 8)--(a.corner 9);
\draw[line width=2.5pt] (a.corner 9)--(a.corner 1);% \draw (a.corner 1) node [shift={(1*360/9-30:0.55)}] {$4$};

\begin{scope}[red,semithick,dashed]
\draw (a.corner 2)--(a.corner 5);
\draw (a.corner 9)--(a.corner 6);
\draw (a.corner 3)--(a.corner 7);
\draw (a.corner 4)--(a.corner 8);
\draw (a.corner 3)--(a.corner 8);
\end{scope}

\foreach \x in {1,2,...,9}
  \fill (a.corner \x) circle[radius=3pt] node[shift={(\x*360/9+35:0.4)}] {$s_{\x}$};
\draw (-3.5,2.5) node {(b)};  
\end{scope}

\begin{scope}[yshift=-7cm]
\node[draw=none,minimum size=3cm,regular polygon,regular polygon sides=9] (a) {};

\draw (a.corner 1)--(a.corner 2);
\draw[line width=2.5pt](a.corner 2)--(a.corner 3);%\draw (a.corner 3) node [shift={(3*360/9-30:0.6)}] {$4$};
\draw[line width=2.5pt] (a.corner 3)--(a.corner 4);%\draw (a.corner 4) node [shift={(4*360/9-30:0.6)}] {$4$};
\draw (a.corner 4)--(a.corner 5);
\draw (a.corner 5)--(a.corner 6);
\draw (a.corner 6)--(a.corner 7);
\draw[line width=2.5pt] (a.corner 7)--(a.corner 8);%\draw (a.corner 8) node [shift={(8*360/9-30:0.6)}] {$4$};
\draw[line width=2.5pt](a.corner 8)--(a.corner 9);%\draw (a.corner 9) node [shift={(9*360/9-30:0.6)}] {$4$};
\draw (a.corner 9)--(a.corner 1);

\begin{scope}[red,semithick,dashed]
\draw (a.corner 1)--(a.corner 5);
\draw (a.corner 1)--(a.corner 6);
\draw (a.corner 2)--(a.corner 9);
\draw (a.corner 4)--(a.corner 7);
\end{scope}

\foreach \x in {1,2,...,9}
  \fill (a.corner \x) circle[radius=3pt] node[shift={(\x*360/9+35:0.4)}] {{$s_\x$}};
\draw (-3.5,2.5) node {(c)};  
\end{scope}

\begin{scope}
\node[draw=none,minimum size=3cm,regular polygon,regular polygon sides=9] (a) {};

\draw (a.corner 1)--(a.corner 2);
\draw (a.corner 2)--(a.corner 3);
\draw (a.corner 3)--(a.corner 4);
\draw (a.corner 4)--(a.corner 5);
\draw (a.corner 5)--(a.corner 6);
\draw (a.corner 6)--(a.corner 7);
\draw (a.corner 7)--(a.corner 8);
\draw (a.corner 8)--(a.corner 9);
\draw (a.corner 9)--(a.corner 1);

\foreach \x in {1,2,...,9}
  \fill (a.corner \x) circle[radius=3pt] node[shift={(\x*360/9+35:0.4)}] {{$s_\x$}};
\draw (-3.5,2.5) node {(a)};  
\end{scope}

\begin{scope}[yshift=-7cm, xshift=10cm]
\node[draw=none,minimum size=3cm,regular polygon,regular polygon sides=9] (a) {};

\draw[line width=2.5pt](a.corner 1)--(a.corner 2);%\draw (a.corner 1) node [shift={(1*360/9-30:0.6)}] {$4$};
\draw (a.corner 2)--(a.corner 3);
%\draw (a.corner 2) node [shift={(2*360/9-30:0.6)}] {$4$};
\draw[line width=2.5pt] (a.corner 3)--(a.corner 4);%\draw (a.corner 4) node [shift={(4*360/9-30:0.6)}] {$4$};
\draw[line width=2.5pt](a.corner 4)--(a.corner 5);%\draw (a.corner 5) node [shift={(5*360/9-30:0.6)}] {$4$};
\draw (a.corner 5)--(a.corner 6);
\draw[line width=2.5pt] (a.corner 6)--(a.corner 7);%\draw (a.corner 7) node [shift={(7*360/9-30:0.6)}] {$4$};
\draw[line width=2.5pt] (a.corner 7)--(a.corner 8);%\draw (a.corner 8) node [shift={(8*360/9-30:0.6)}] {$4$};
\draw (a.corner 8)--(a.corner 9);
\draw[line width=2.5pt] (a.corner 9)--(a.corner 1);

\begin{scope}[red,semithick,dashed]
\draw (a.corner 2)--(a.corner 6);
\draw (a.corner 3)--(a.corner 8);
\draw (a.corner 5)--(a.corner 9);
\end{scope}

\foreach \x in {1,2,...,9}
  \fill (a.corner \x) circle[radius=3pt] node[shift={(\x*360/9+35:0.4)}] {{$s_\x$}};
\draw (-3.5,2.5) node {(d)};  
\end{scope}

\end{tikzpicture}
    \caption{\small{Dynkin diagrams of Coxeter systems having hypergraph index (a) $0$, (b) $1$, (c) $2$ and (d) $\infty$. To ease notation, edges labelled $4$ are represented here as thick edges (whereas edges labelled $3$ are represented as thin edges, as usual).
    Red dashed lines encode the complements of subsets in $\Lambda_0$.}}
    \label{fig:cycle}
\end{figure}

Consider the Dynkin diagrams depicted in Figure~\ref{fig:cycle}.  Here, subsets forming elements of~$\Lambda_0$ are represented by dashed red lines, with such a line between vertices $s_i$ and $s_j$ encoding the subset $T_{i,j}:=S\setminus\{s_i,s_j\}$. 

The system in (a) is the irreducible affine system of type $\widetilde A_8$, so it has hypergraph index~$0$. 

In (b), subsets $T_{3,7}$ and $T_{4,8}$ are both of the form $A\times B$ for nonspherical $A$ and $B$, and $T_{3,7}\cap T_{4,8}$ is nonspherical of the form $A'\times K$, with $A'=\{s_2,s_1,s_9\}$ generating an affine system of type $\widetilde C_2$ and $K=\{s_5,s_6\}$ a finite one of type $B_2$. Since $T_{3,7}\cup T_{4,8}=S$, the hypergraph index in this case is $1$. 

In (c), $\Lambda_0$ consists of the four subsets $T_{1,5}$, $T_{1,6}$, $T_{2,9}$ and $T_{4,7}$, but only two of them have a nonspherical intersection, with $T_{1,5}\cap T_{1,6}$ of type $\widetilde C_2\times \widetilde C_2$, and $T_{1,5}\cup T_{1,6}=T_1:=S\setminus\{s_1\}\in \Lambda_1$. Then $T_1\cap T_{2,9}$ is nonspherical (irreducible affine of type $\widetilde C_5$), and hence $S=T_1\cup T_{2,9}$ belongs to $\Lambda_2$, so that the hypergraph index equals $2$ here. 

In (d), $\Lambda_0$ consists of three subsets,  $T_{2,6}$, $T_{3,8}$ and $T_{5,9}$, all generating systems of type $\widetilde C_3\times\widetilde C_2$, but all their pairwise intersections are spherical (of type $B_2\times B_2\times A_1$), hence %$\Lambda_h=\Lambda_0\ni \negthickspace \negthickspace \negthickspace / \,S$, 
$\Lambda_h=\Lambda_0\not\ni S$,
for all $h$. Thus the hypergraph index is infinite in (d).

\subsubsection{Further examples with hypergraph index $\betti{\Delta}+1$} 
\label{sec:DynkinGenTheta}

Aided by computer, we have also found examples of Dynkin diagrams with no edge labels $\infty$ and first Betti number $i$ such that the corresponding Coxeter system has hypergraph index $i+1$, for $i=2,3,4,5$. The Dynkin diagram in Figure~\ref{fig:GenThetaGraphs} is such an example for $i = 5$, and has the additional property that its subdiagram induced by the vertices $\{s_1,s_2,\dots,s_{8+i}\}$ has first Betti number $i$ and hypergraph index $i+1$ for all $i=1,\dots,5$.   
It is not easy to verify the hypergraph index in this example by hand, so the interested reader may check this using our  \textsf{GAP} code~\cite{hindex_gap}, which also provides the Coxeter matrices for the key examples of this paper.

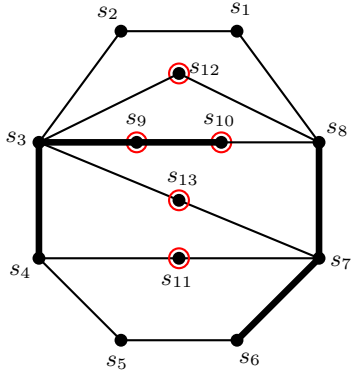
\begin{figure}[!ht]
\centering
\begin{tikzpicture}[double distance=3pt, thick, scale=0.75]
%\draw [help lines] (-4,-4) grid (4,4);
\node[draw=none,minimum size=4cm,regular polygon,regular polygon sides=8] (a) {};
\path let \p1 = (a.corner 1) in coordinate (e1) at (\x1,3);
\path let \p1 = (a.corner 2) in coordinate (e2) at (\x1,3);
\draw (e1)--(e2);
\draw (e2)--(a.corner 3);
\draw[line width=2.5pt] (a.corner 3)--(a.corner 4);
\draw (a.corner 4)--(a.corner 5);
\draw (a.corner 5)--(a.corner 6);
\draw[line width=2.5pt] (a.corner 6)--(a.corner 7);
\draw[line width=2.5pt] (a.corner 7)--(a.corner 8);
\draw (a.corner 8)--(e1);
\foreach \x in {3,4,...,8}
  \fill (a.corner \x) circle[radius=3.25pt] node[shift={(\x*360/8+35:0.3)}] {{\footnotesize$s_\x$}};
\fill (e1) circle[radius=3.25pt] node[shift={(1*360/8+35:0.3)}] {{\footnotesize$s_1$}};
\fill (e2) circle[radius=3.25pt] node[shift={(2*360/8+35:0.3)}] {{\footnotesize$s_2$}};

\draw (a.corner 4)--(a.corner 7);
\draw (a.corner 3)--(a.corner 7);
\draw [red] (0,0) circle[radius=5pt];
\fill (0,0) circle[radius=3.25pt] node[shift={(1*360/8+35:0.3)}] {{\footnotesize$s_{13}$}};
\path let \p1 = (a.corner 3) in coordinate (c) at (-0.75,\y1);
\draw [red] (c) circle[radius=5pt];
\fill (c) circle[radius=3.25pt] node[shift={(1*360/8+45:0.3)}] {{\footnotesize$s_9$}};
\path let \p1 = (a.corner 3) in coordinate (b) at (0.75,\y1);
\draw [red] (b) circle[radius=5pt];
\fill (b) circle[radius=3.25pt] node[shift={(2*360/8+5:0.3)}] {{\footnotesize$s_{10}$}};
\draw [line width=2.5pt] (a.corner 3)--(b);
\draw (b)--(a.corner 8);
\path let \p1 = (a.corner 3) in coordinate (d) at (0,-\y1);
\draw [red] (d) circle[radius=5pt];
\fill (d) circle[radius=3.25pt] node[shift={(5*360/8+40:0.3)}] {{\footnotesize$s_{11}$}};
\draw [red] (0,2.25) circle[radius=5pt];
\fill (0,2.25) circle[radius=3.25pt] node[shift={(7*360/8+55:0.35)}] {{\footnotesize$s_{12}$}};
\draw (a.corner 3)--(0,2.25)--(a.corner 8);
\end{tikzpicture}
\caption{\small{A Dynkin diagram with first Betti number $5$ whose associated Coxeter system has hypergraph index 6. To ease notation, 
edge labels~4 are represented as thick edges (whereas thin edges represent edge labels~3, as usual). For each $i=1,\dots,5$, a subdiagram on the subset of vertices $\{s_1,\dots,s_{8+i}\}$ has first Betti number $i$ and hypergraph index $i+1$. }   \label{fig:GenThetaGraphs}}
\end{figure}

\subsection{Questions}\label{sec:questionsBetti}

The following open questions are suggested by Theorem~\ref{thm:bettiIntro}  and the examples in Section~\ref{sec:examplesBetti}.

\begin{questions}\label{q:betti}   
Let $\Delta$ be a connected Dynkin diagram. 
\begin{enumerate}
    \item Which hypergraph indexes are realised as we vary over all connected Dynkin diagrams having the same first Betti number as $\Delta$?  In particular, is hypergraph index $b_1(\Delta) + 1$ realised?
    \item Which hypergraph indexes are realised as we vary over all connected Dynkin diagrams having the same underlying unlabelled graph as $\Delta$?  In particular, for which underlying graphs is hypergraph index $b_1(\Delta) + 1$ realised?
\end{enumerate}
\end{questions}
 
We restrict to connected Dynkin diagrams here because, by Lemmas~\ref{lem:finite} and~\ref{lem:reducible1} and Proposition~\ref{prop:reducible2}, the main interest for hypergraph index is in irreducible Coxeter systems.   We also observe that by Proposition~\ref{prop:labels7up}, once the underlying graph of $\Delta$ is fixed, we only need to consider edge labels in the finite set $\{2,3,4,5,6,7,\infty\}$ for (2).  Note as well that by Theorem~\ref{thm:hyp_thick_div_Intro}(1) and Corollary~\ref{cor:linearIntro}, a Coxeter system which is irreducible and has hypergraph index $0$ must be irreducible affine, and hence its Dynkin diagram is restricted to being a tree or a cycle.

\frenchspacing
\bibliographystyle{alpha}
\bibliography{refs}

\end{document}